\documentclass[11pt,reqno]{amsart}
\usepackage{amsthm,amsmath,amsfonts,amssymb,bm,bbm,mathtools} 
\usepackage{amsaddr}

\usepackage{datetime} \usdate
\usepackage{currfile} 
\usepackage{hyperref}          
\usepackage{color}
\usepackage{subfigure}

\date{\today } 
\newif\ifdraft

\newcommand{\red}{} 

\ifdraft
\renewcommand{\red}{\color{red}}

\fi

\iffalse 

 \usepackage[landscape,includehead,headheight=12pt,headsep=10pt,left=\columnsep,right=\columnsep,top=12.98999pt,bottom=\columnsep,]{geometry}
 \twocolumn
\else
\usepackage[margin=30mm]{geometry}
\fi

\usepackage{pdfsync,hyperref}

\numberwithin{equation}{section}

\newtheorem{theorem}{Theorem}[section]

\newtheorem{corollary}[theorem]{Corollary}

\newtheorem{lemma}[theorem]{Lemma}
\newtheorem{sublemma}[theorem]{Sublemma}
\newtheorem{proposition}[theorem]{Proposition}
\newtheorem{definition}[theorem]{Definition}

\theoremstyle{remark}
\newtheorem{remark}[theorem]{Remark}

\newcommand{\nwc}{\newcommand}

\nwc{\pdfx}[2]{\texorpdfstring{#1}{#2}} 
\nwc{\hide}[1]{}  
\nwc{\qref}[1]{(\ref{#1})}
\nwc{\ip}[1]{\langle #1 \rangle}

\nwc{\D}{\partial}
\nwc{\grad}{\nabla}
\nwc{\eps}{\varepsilon}
\nwc{\inv}{^{-1}}
\nwc{\wkto}{\xrightharpoonup{\star}}

\nwc{\tr}{\mathop{\rm tr}\nolimits}
\nwc{\re}{\mathop{\rm Re}\nolimits}
\nwc{\im}{\mathop{\rm Im}\nolimits}
\nwc{\one}{{\mathbbm{1}}}
\nwc{\calA}{{\mathcal A}}
\nwc{\calB}{{\mathcal B}}
\nwc{\calC}{{\mathcal C}}
\nwc{\calH}{{\mathcal H}}
\nwc{\calL}{{\mathcal L}}
\nwc{\calK}{{\mathcal K}}
\nwc{\calM}{{\mathcal M}}
\nwc{\calP}{{\mathcal P}}
\nwc{\calU}{{\mathcal U}}

\newcommand{\R}{\mathbb{R}}
\newcommand{\N}{\mathbb{N}}

\newcommand{\te}{\textrm}
\newcommand{\tacka}{\,\cdot\,}
\newcommand{\veps}{\varepsilon}
\newcommand{\ulam}{\underline \lambda}
\newcommand{\olam}{\overline \lambda}
\newcommand{\oor}{\overline r}

\newcommand{\tT}{S^\eps} 

\DeclareMathOperator{\id}{id}
\DeclareMathOperator{\supp}{supp}
\nwc{\esssup}{\mathop{\rm ess\, sup}}

\DeclareMathOperator{\diam}{diam}
\DeclareMathOperator{\diag}{diag}
\DeclareMathOperator{\dive}{div}
\DeclareMathOperator{\Hess}{Hess}

\definecolor{mygreen}{rgb}{0.1,0.75,0.2}

\newcommand{\Qdom}{Q} 

\newcommand{\supn}{^{n}}
\nwc{\wksto}{\xrightharpoonup{\star}}
\nwc{\tlpto}{\overset{{T\!L^p}}{\longrightarrow} }
\nwc{\tltwoto}{\overset{{T\!L^2}}{\longrightarrow} }
\nwc{\tlp}{{T\!L^p}}
\nwc{\tltwo}{{T\!L^2}}

\date{\today} 
\ifdraft \date{WORKING DRAFT \today: PLEASE DO NOT CIRCULATE} \fi

\begin{document}

\ifdraft
\title[DRAFT \mmddyyyydate\today\quad \currenttime \ \ PLEASE DO NOT CIRCULATE]
{Least action principles for incompressible flows and geodesics between shapes}

\else
\title[Least action, incompressible flow and geodesics]
{Least action principles for incompressible flows and geodesics between shapes}
\fi

\author{Jian-Guo Liu}
\address{
Department of Physics and Department of Mathematics\\
Duke University, Durham, NC 27708, USA}
\email{jliu@phy.duke.edu}

\author[Jian-Guo Liu, Robert L. Pego and Dejan Slep{\v{c}}ev]{Robert L. Pego \ \ and \ \  Dejan Slep{\v{c}}ev}
\address{Department of Mathematical Sciences and Center for Nonlinear Analysis\\
Carnegie Mellon University, Pittsburgh, Pennsylvania, PA 12513, USA}
\email{rpego@cmu.edu, slepcev@math.cmu.edu}

\begin{abstract}  %
As V.~I.~Arnold observed in the 1960s, the Euler equations
of incompressible fluid flow correspond formally to geodesic equations
in a group of volume-preserving diffeomorphisms.
Working in an Eulerian framework, we study incompressible flows
of shapes as critical paths for action (kinetic energy)
along transport paths constrained to have characteristic-function densities.
The formal geodesic equations for this problem
are Euler equations for incompressible, inviscid potential flow
of fluid with zero pressure and surface tension on the free boundary.
The problem of minimizing this action exhibits an instability associated with
microdroplet formation, with the following outcomes:
Any two shapes of equal volume can be approximately connected by 
an Euler spray---a countable superposition of ellipsoidal geodesics.
The infimum of the action is the Wasserstein distance squared,
and is almost never attained except in dimension 1.
Every Wasserstein geodesic between bounded densities of compact support provides
a solution of the (compressible) pressureless Euler system 
that is a weak limit of (incompressible) Euler sprays.
Each such Wasserstein geodesic is also the unique minimizer of 
a relaxed least-action principle for a two-fluid mixture theory
corresponding to incompressible fluid mixed with vacuum.
\end{abstract}


\maketitle



\section{Introduction}\label{s:intro}

\subsection{Overview}
%
The geometric interpretation of solutions of the Euler equations 
of incompressible inviscid fluid flow as 
geodesic paths in the group of volume-preserving diffeomorphisms
was famously pioneered by V.~I.~Arnold~\cite{Arnold66}.
If we consider an Eulerian description for an incompressible body of 
constant-density fluid moving freely in space,
such geodesic paths correspond to critical paths for the action
\begin{equation}\label{e:act1}
{\mathcal A}=\int_0^1\!\!\int_{\R^d}\, \rho|v|^2\,dx\,dt\,, \qquad
\end{equation}
where $\rho=(\rho_t)_{t\in[0,1]}$ is a path of 
characteristic-function densities
transported by a velocity field $v\in L^2(\rho\,dx\,dt)$ according
to the continuity equation
\begin{equation}
\partial_t\rho+\nabla\cdot(\rho v)=0\,. 
\label{e:inf1a}
\end{equation}
Such characteristic-function densities $\rho_t$ 
represent a fluid having \emph{shape} $\Omega_t$ at time $t$:
\begin{equation}\label{e:inf1b}
\rho_t = \one_{\Omega_t} \,, \quad t\in[0,1].
\end{equation}
Naturally, the velocity field must be divergence free 
in the interior of the fluid domain $\Omega_t$, satisfying $\nabla\cdot v=0$ there.
Equation~\eqref{e:inf1a} holds in the sense of distributions in $\R^d\times[0,1]$,
interpreting $\rho v$ as $0$ wherever $\rho=0$.

In this Eulerian framework, it is natural to study the action 
in \eqref{e:act1} subject to given endpoint conditions of the form
\begin{equation}
\rho_0=\one_{\Omega_0}\,,\quad \rho_1=\one_{\Omega_1}\,.
\label{e:inf1a2} 
\end{equation}
These endpoint conditions differ from Arnold-style conditions 
that fix the flow-induced volume-preserving diffeomorphism
between $\Omega_0$ and $\Omega_1$, 
and correspond instead to fixing only the \textit{image} of this diffeomorphism.
Imposing endpoint conditions in an Eulerian transport framework 
as in \eqref{e:inf1a2} 
is exactly analogous to the fundamental study of Benamou and Brenier \cite{BenBre00} 
that relates the minimization of the action \eqref{e:act1} 
without incompressibility constraints
to Wasserstein (Monge-Kantorovich) distance with quadratic cost.  

As we show in section~\ref{s:geo} below, it turns out that
the geodesic equations that result are precisely the Euler equations for  
\emph{potential flow} of an incompressible, inviscid
fluid occupying domain $\Omega_t$, with 
\emph{zero pressure and zero surface tension} 
on the free boundary $\D\Omega_t$.
In short, the geodesic equations are classic water wave equations with zero gravity
and surface tension. The initial-value problem for these equations has recently
been studied in detail---the works \cite{Lindblad,CoutShko2007,CoutShko2010}
extend the breakthrough works of Wu \cite{Wu97,Wu99} 
to deal with nonzero vorticity and zero gravity, and
establish short-time existence and uniqueness for sufficiently smooth initial data
in certain bounded domains.

The problem of \textit{minimizing} the action in \eqref{e:act1} subject to the constraints above turns out to be ill-posed if the dimension $d>1$, as we will show in this paper.  
By this we mean that action-minimizing paths that satisfy all the constraints \eqref{e:inf1a}, \eqref{e:inf1b} and \eqref{e:inf1a2} do not exist in general,
even locally.  Nevertheless, the infimum of the action defines a distance
between equal-volume sets which we may call \textit{shape distance}, 
determined by 
\begin{equation}\label{d:ds}
d_s(\Omega_0,\Omega_1)^2 = \inf \calA \,,
\end{equation}
where the infimum is taken subject to the constraints \eqref{e:inf1a}, \eqref{e:inf1b}, \eqref{e:inf1a2} above.
By the well-known result of Benamou and Brenier~\cite{BenBre00}, 
it is clear that
\begin{equation}\label{d:ge1}
d_s(\Omega_0,\Omega_1) \ge d_W(\one_{\Omega_0},\one_{\Omega_1}),
\end{equation}
where $d_W(\one_{\Omega_0},\one_{\Omega_1})$
 denotes the usual Wasserstein distance 
(Monge-Kantorovich distance with quadratic cost)  
between the measures with densities $\one_{\Omega_0}$ and $\one_{\Omega_1}$.
This is so because the result of \cite{BenBre00} characterizes the squared
Wasserstein distance $d_W(\one_{\Omega_0},\one_{\Omega_1})^2$ 
as the infimum in \eqref{d:ds} subject to the same 
transport and endpoint constraints as in \eqref{e:inf1a} and \eqref{e:inf1a2}, 
but \emph{without} the constraint \eqref{e:inf1b} 
that makes $\rho$  a characteristic function.

Our objective in this paper is to develop several results that precisely relate 
the infimum in \eqref{d:ds} and corresponding geodesics 
(critical paths for action) on the one hand,
to Wasserstein distance and corresponding length-minimizing 
Wasserstein geodesics---also known as displacement interpolants---on the other hand. 
Wasserstein geodesic paths typically do not have characteristic-function
densities, and thus do not correspond to geodesics for the shape distance $d_s$. 
A common theme in our results is the observation that the least-action
problem in \eqref{d:ds} is subject to 
an instability associated with \textit{microdroplet} formation.

The idea that Arnold's least action principle for incompressible flows
may suffer analytically from instability or non-attainment
appears to have led Brenier and others starting in the late 1980s
to investigate various forms of relaxed least-action problems for incompresible flows  
\cite{Brenier89,Brenier92,Shnirelman94,Brenier97,Brenier99,Brenier2008,AmbrosioFigalli09,LopesNP}.
Such relaxed problems involve generalizing the notion of flows of diffeomorphisms 
to formulate a framework in which existence of minimizers can be proved,
using convex analysis or variational methods.
Our microdroplet constructions also provide a precise connection between
Wasserstein geodesic paths (which correspond to pressureless, compressible fluid flows)
and relaxed least-action problems for flows of incompressible-fluid--vacuum mixtures.

\subsection{Main results}  
Broadly speaking, our aim is to investigate the geometry of the space of shapes
(corresponding to characteristic-function densities), focusing on the geodesics
for shape distance and the corresponding distance induced by \eqref{d:ds}.
Studies of this type have been carried out by many other authors, as will be discussed
in subsection~\ref{ss:shapespace}.  

One issue about which we have little to say is that of geodesic completeness,
in the sense this term is used in differential geometry.
Here this concept corresponds to global existence in time 
for weak solutions of the free-boundary Euler equations.  
But in addition to other well-known difficulties for Euler equations,
in the present situation there arise further thorny problems, 
such as collisions of fluid droplets, for example. 

\textit{Geodesics between shapes.}
Our principal results instead address the question of determining 
which targets and sources are connected by geodesics for shape distance, 
and how these relate to the infimization in \eqref{d:ds}.
The general question of determining all exact connecting critical paths
is an interesting one that seems difficult to answer.
In regard to a related question in a space of smooth enough volume-preserving diffeomorphisms of a fixed manifold, 
Ebin and Marsden in \cite[15.2(vii)]{EbinMarsden} established a covering
theorem showing that the geodesic flow starting from the identity diffeomorphism 
covers a full neighborhood.  
By contrast, what our first result will show is that for an arbitrary 
bounded open source domain $\Omega_0$,
targets for shape-distance geodesics are globally dense in the `manifold'
of bounded open sets of the same volume. 
The idea is to construct geodesics comprised of tiny disjoint droplets 
(which we call \emph{Euler sprays}) that approximately reach an 
arbitrarily specified $\Omega_1$ as closely as desired
in terms of an optimal-transport distance.

Below, it is convenient to denote the distance between two bounded measurable sets
$\Omega_0$, $\Omega_1$ that is induced by Wasserstein distance 
by the overloaded notation 
\begin{equation}
 d_W(\Omega_0,\Omega_1) = d_W(\one_{\Omega_0},\one_{\Omega_1}),
\end{equation}
and similarly with $L^p$-Wasserstein distance $d_p$ for any value of $p\in[1,\infty]$. 
\begin{theorem}\label{th2}
Let $\Omega_0$, $\Omega_1$ be any pair of bounded open sets in $\R^d$ with equal volume.
Then for any $\eps>0$, there is an Euler spray which transports 
the source  $\Omega_0$ (up to a null set) to a target $\Omega_1^\eps$
satisfying $d_\infty(\Omega_1,\Omega_1^\eps)<\eps$.
The action ${\mathcal A}^\eps$ of the spray satisfies
\[
d_s(\Omega_0,\Omega_1^\eps)^2
\le
{\mathcal A}^\eps \le d_W(\Omega_0,{\Omega_1})^2+\eps\,.
\]
\end{theorem}

The precise definition of an Euler spray and the proof of this result 
will be provided in section~\ref{s:spray}.  
A particular, simple geodesic for shape distance will play a special role in our analysis.
Namely, we observe in Proposition~\ref{p:Edrop} that a path $t\mapsto\Omega_t$ of {ellipsoids} determines a 
critical path for the action~\eqref{e:act1} constrained
by \eqref{e:inf1a}--\eqref{e:inf1a2}
if and only if the $d$-dimensional vector $a(t)=(a_1(t),\ldots,a_d(t))$,
formed by the principal axis lengths, follows a geodesic curve on 
the hyperboloid-like surface in $\R^d$ determined by the constraint
that corresponds to constant volume,
\begin{equation}
a_1a_2\cdots a_d = \te{const}.
\end{equation}

 \begin{figure}[ht]  
     \centering
 \subfigure[Source disk $\Omega_0$ decomposed into microdroplets $B_i$ at $t=0$.]
 {  \raisebox{2.6mm}{\includegraphics[width=40mm]{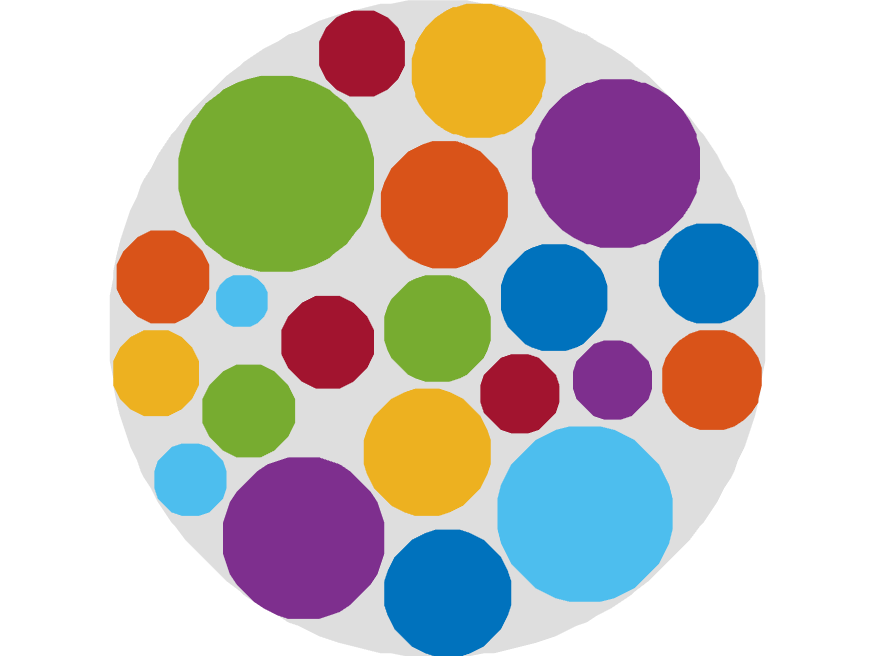} }} \hspace*{1pt} \nolinebreak
 \subfigure[Displacement interpolants at path midpoint $t=\frac12$.]
 {\raisebox{1.0mm}{\includegraphics[width=4.5cm]{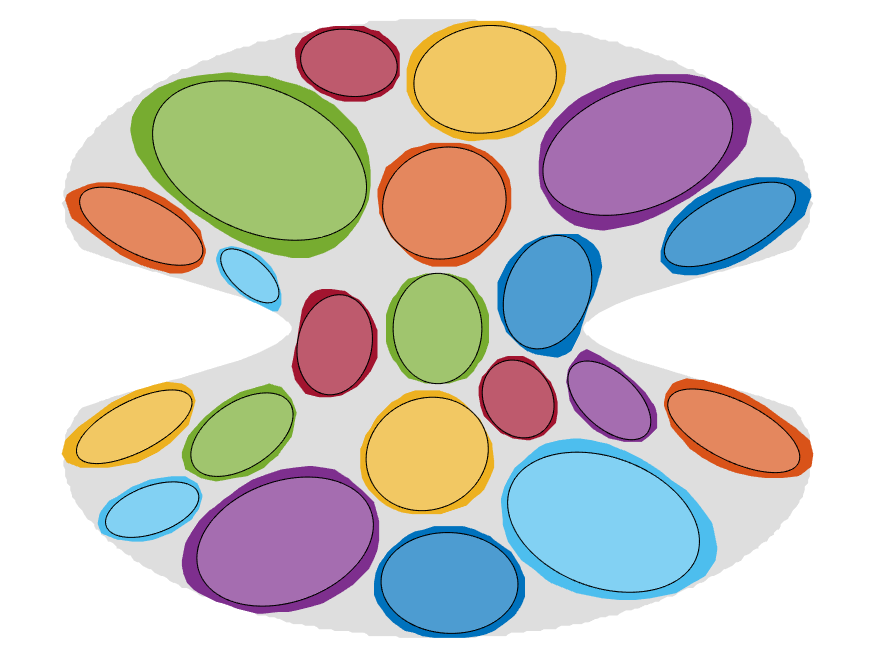}}\hspace*{8pt}}  \hspace*{4pt}\nolinebreak
 \subfigure[Expanded target $(1+\eps)T(\Omega_0)$ at $t=1$, indicating 
expanded microdroplet images $(1+\eps)T(B_i)$ (dark) and ellipsoidal 
approximation of $T(B_i)$ (light). $\eps=0.25$.]
 {\includegraphics[ width=4.8cm]{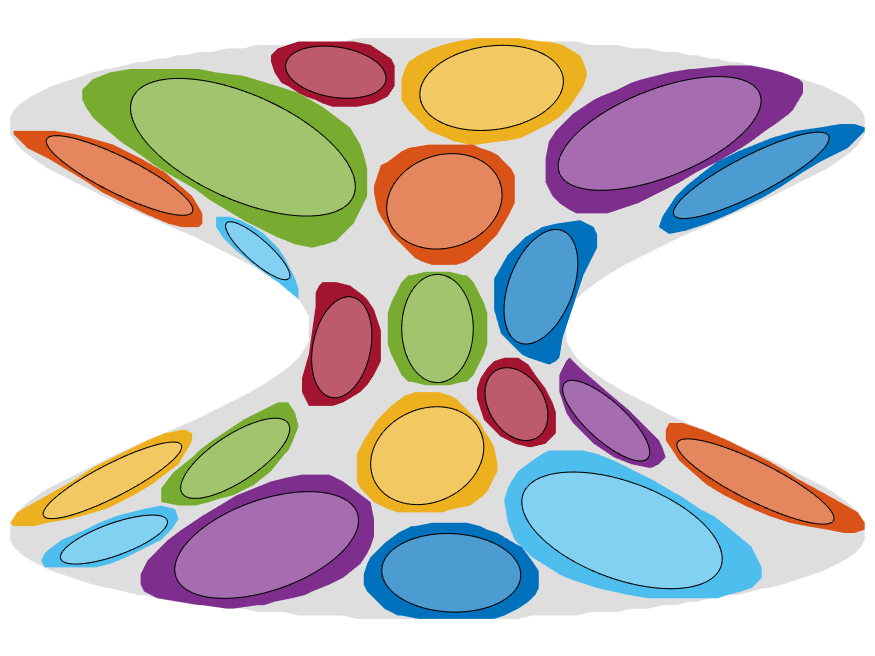}}
\caption{Illustration of Wasserstein geodesic flow from $\Omega_0$ to $\Omega_1=T(\Omega_0)$, where $T$ is the Brenier map.
Source $\Omega_0$ is decomposed into countably many small balls, few shown. 
Matching shades indicate corresponding droplets transported by displacement interpolation.
Euler spray droplets are nested inside Wasserstein ellipsoids and remain disjoint.
}
\label{fig:vitali}
 \end{figure}

To prove Theorem~\ref{th2},
we decompose the source domain $\Omega_0$, up to a set of measure zero,
as a countable union of tiny disjoint open balls using a Vitali covering lemma. 
These `microdroplets' are transported by ellipsoidal geodesics 
that approximate a local linearization of the Wasserstein geodesic 
(displacement interpolant)
which produces
straight-line transport of points from the source $\Omega_0$ to the target
$\Omega_1$. 
Crucially, the droplets remain disjoint {(essentially due to the 
convexity of the density along the straight Wasserstein transport paths).}
The total action or cost along the resulting path of
`spray' densities is then shown to be close to that attained by the
Wasserstein geodesic.

The ideas behind the construction of the Euler sprays are illustrated 
in Figure~\ref{fig:vitali}. The shaded background in panel (c) indicates
the target $\Omega_1=T(\Omega_0)$, expanded by a factor $(1+\eps)$,
where $T\colon\Omega_0\to\Omega_1$ 
is a computed approximation to the Brenier (optimal transport) map.
The expanded images $(1+\eps)T(B_i)$ of balls $B_i$ in the source
are shown in dark shades, and (nested inside) ellipsoidal approximations 
to $T(B_i)$ in corresponding light shades.
We show that along Wasserstein geodesics (displacement interpolants),
nested images remain nested, and that the ellipsoidal Euler geodesics (not shown)
remain nested inside the Wasserstein-transported ellipses indicated in light shades.

The result of Theorem~\ref{th2}  directly implies that a natural relaxation
of the shape distance $d_s$---the lower semicontinuous envelope 
with respect to Wasserstein distance---agrees with the induced Wasserstein distance $d_W$.
(See~\cite[section 1.7.2]{Braides} regarding the general notion of relaxation 
of variational problems.) 
In fact, by a rather straightforward completion
argument we can identify the shape distance in \eqref{d:ds} as follows.

\begin{theorem}\label{th1}
For every pair of  bounded measurable sets in $\R^d$ of equal volume,
\begin{equation*}
d_s(\Omega_0,\Omega_1) = d_W({\Omega_0},{\Omega_1}).
\end{equation*}
\end{theorem}

As is well known, Wasserstein distance between measures
of a given mass that are supported inside a fixed compact set
induces the topology of weak-$\star$ convergence.
In this topology, the closure of the set of such measures with 
characteristic-function densities 
is the set of measurable functions $\rho\colon\R^d\to[0,1]$ with compact support.
Theorem~\ref{th1} above is a corollary of the following more general result 
that indicates how Euler-spray geodesic paths approximately connect 
arbitrary endpoints in this set.

\begin{theorem}\label{t:complete} 
Let $\rho_0$, $\rho_1\colon\R^d\to[0,1]$ be measurable functions of compact support
that satisfy 
\[
\int_{\R^d}\rho_0 = \int_{\R^d}\rho_1\,.
\]  
Then
\begin{itemize}
\item[(a)] For any $\eps>0$ there are open sets $\Omega_0$,
$\Omega_1$ which satisfy 
\[
d_\infty(\rho_0,\one_{\Omega_0}) + d_\infty(\rho_1,\one_{\Omega_1})<\eps,
\]
and are connected by an Euler spray
whose total action $\calA^\eps$ satisfies
\[
\calA^\eps \le d_W(\rho_0,\rho_1)^2 + \eps.
\]
\item[(b)] For any $\eps>0$ there is a 
path $\rho^\eps=(\rho_t^\eps)_{t\in(0,1)}$ on $(0,1)$ consisting of
a countable concatenation of Euler sprays, such that 
\[
\rho_t^\eps \wkto \rho_0 \quad\mbox{as $t\to0^+$}\,,\qquad
\rho_t^\eps \wkto \rho_1 \quad\mbox{as $t\to1^-$}\,,
\]
and the total action $\calA^\eps$ of the path satisfies
\[
\calA^\eps   
=  \int_0^1\int_{\R^d} \rho_t^\eps |v^\eps|^2\,dx\,dt 
\le d_W(\rho_0,\rho_1)^2+ \eps.
\]
\end{itemize}
\end{theorem}

The results of Theorems~\ref{th2} and \ref{t:complete} concern geodesics for shape distance
that  only approximately connect  arbitrary sources $\Omega_0$ and targets $\Omega_1$.
A uniqueness property of Wasserstein geodesics
allows us to establish the following sharp criterion for existence and non-existence
of \textit{length-minimizing} shape geodesics that exactly connect source to target. 
 
\begin{theorem}\label{th4} 
Let $\Omega_0$, $\Omega_1$ be bounded open
sets in $\R^d$ with equal volume, 
and let $\rho=(\rho_t)_{t\in[0,1]}$ be the density along the Wasserstein geodesic path 
that connects $\one_{\Omega_0}$ and $\one_{\Omega_1}$.
Then the infimum for shape distance in \eqref{d:ds} is achieved
by some path satisfying the constraints \eqref{e:inf1a},\eqref{e:inf1b},\eqref{e:inf1a2}
if and only if $\rho$ is a characteristic function.
\end{theorem}

For dimension $d=1$ the Wasserstein density is always a characteristic function.
For dimension $d>1$ however, this property of being a characteristic function,
{together with convexity of the density along transport lines,}
requires that the Wasserstein geodesic is given piecewise by rigid translation.
See Corollary~\ref{cor:rigid} and Remark~\ref{r:rho3} in {subsection \ref{ss:rigid}.} 

\textit{Limits of Euler sprays.}
For the Euler sprays constructed in the proof of Theorem~\ref{th2},
the fluid domains $\Omega_t$ do not typically have smooth boundary, 
due to the presence of cluster points of the countable set of microdroplets.
The geodesic equations that they satisfy, then, are not quite 
classical free-boundary water-wave equations. 
Rather, our Euler sprays 
provide a family of weak solutions $(\rho^\eps,v^\eps,p^\eps)$
to the following system of Euler equations:
\begin{align}
&\D_t\rho + \nabla\cdot(\rho v) = 0, 
\label{e:euler1} \\
&\D_t(\rho v) + \nabla\cdot(\rho v\otimes v) + \nabla p = 0,
\label{e:euler2}
\end{align}
with the ``incompressibility'' constraint that $\rho^\eps$ is a shape density, meaning
it is a characteristic function as in \eqref{e:inf1b}.
Both of these equations hold in the sense of distributions on $\R^d\times[0,1]$,
which means the following:
For any smooth test functions 
$q\in C_c^\infty(\R^d\times[0,1],\R)$ and 
$\tilde v\in C_c^\infty(\R^d\times[0,1],\R^d)$,
\begin{align}
\int_0^1\int_{\R^d}\rho(\D_t q + v\cdot \nabla q) \,dx\,dt &= 
\left. \int_{\R^d} \rho q\,dx \right|^{t=1}_{t=0} \,,
\label{e:eulerwk1}
\\
\int_0^1\int_{\R^d} \rho v\cdot (\D_t \tilde v +  v\cdot\nabla \tilde v)+p\nabla \cdot \tilde v \,dx\,dt &= 
\left. \int_{\R^d} \rho v\cdot \tilde v\,dx \right|^{t=1}_{t=0} \,.
\label{e:eulerwk2}
\end{align}

Now, limits as $\eps\to0$ of these Euler-spray geodesics can be considered.
By dealing with a sequence of initial and final data 
$\rho_0^k=\one_{\Omega_0^k}$, $\rho_1^k=\one_{\Omega_1^k}$
that converge weak-$\star$, we find that it is possible to approximate
a general family of Wasserstein geodesic paths, in the following sense.

\begin{theorem}\label{th3}
Let $\rho_0,\rho_1\colon\R^d\to[0,1]$ be measurable functions of compact support 
that satisfy
\[
\int_{\R^d}\rho_0 = \int_{\R^d}\rho_1\,.
\]  
Let $(\rho,v)$ be the density and transport velocity determined by 
the unique Wasserstein geodesic 
that connects the measures with densities $\rho_0$ and $\rho_1$
as described in section~2.

Then there is a sequence of weak solutions $(\rho^k,v^k,p^k)$
to \eqref{e:eulerwk1}--\eqref{e:eulerwk2},
associated to Euler sprays as provided by Theorem~\ref{th2}, that 
converge to  $(\rho,v,0)$, and $(\rho,v)$ is a weak solution
of the pressureless Euler system 
\begin{align}
&\D_t\rho + \nabla\cdot(\rho v) = 0, 
\label{e:euler1p} \\
&\D_t(\rho v) + \nabla\cdot(\rho v\otimes v) = 0.
\label{e:euler2p}
\end{align}
The convergence holds in the the following sense: $p^k\to0$ uniformly,
and 
\begin{equation}\label{e:rvwk}
\rho^k\wksto \rho,\qquad
\rho^k v^k\wksto \rho v, \qquad
\rho^k v^k\otimes v^k \wksto \rho v\otimes v ,
\end{equation}
weak-$\star$ in $L^\infty$ on $\R^d\times[0,1]$.
\end{theorem}
This result shows that one can approximate a large family of solutions of
pressureless Euler equations, ones coming from Wasserstein geodesics having
bounded densities of compact support, by solutions of incompressible Euler
equations with vacuum. (For densities taking values in $[0,R]$ instead of $[0,1]$,
one can simply scale the densities coming from the Euler sprays,
by multiplying by $R$.)

The convergence in \eqref{e:rvwk} can be strengthened in terms
of the $\tlp$ topology that was introduced in \cite{GTS} to compare two functions 
that are absolutely continuous with respect to 
different probability measures---see subsection~\ref{r:TLp}.
The result of Theorem~\ref{t:tlptime} essentially shows that while oscillations 
exist in space and time for the densities $\rho^k$ and velocities $v^k$ in 
Theorem~\ref{th3}, there are no oscillations following the flow lines.
Our analysis of convergence in the $\tlp$ topology is based upon
an improved stability result regarding the stability of transport maps.
We describe and establish this stability result separately in an Appendix,
due to its potential for independent interest.

\textit{Relaxed least-action principles.}
Our next result establishes a precise connection between Wasserstein geodesics
and a relaxed least-action principle for incompressible flow of two-fluid mixtures.
In particular this relates to work of Brenier on relaxations of Arnold's least-action 
principle for incompressible flow~\cite{Brenier89,Brenier92,Brenier97,Brenier99,Brenier2008,Brenier2013}.
Our mixture model is a variant of Brenier's model for 
homogenized vortex sheets \cite{Brenier97}, and is related
to the variable-density model studied by Lopes et al.~\cite{LopesNP}.
These models involve Lagrangian rather than Eulerian endpoint conditions;
we make a comparison and prove an existence theorem for our model
in Appendix~\ref{a:LAP}.
Our model also  allows one fluid to have zero density, however,
corresponding to a fluid-vacuum mixture.
In this degenerate case,
we show that the Wasserstein geodesic provides the unique minimizer
of the relaxed least-action principle---see Theorem~\ref{t:W}.

An important point of contrast between our results and those of 
Brenier \cite{Brenier99} and Lopes et al.~\cite{LopesNP}
concerns the issue of consistency of the relaxed theory with classical solutions.
The results of \cite{Brenier99} and \cite{LopesNP} establish
that classical smooth solutions of the incompressible fluid equations
do provide action minimizers locally, for sufficiently short time.
However, the result of Theorem~\ref{th4} above shows that 
for any smooth free-boundary fluid motion (corresponding to a shape geodesic)
that is not given locally by rigid motion,
the solution \emph{never} achieves minimum action,
over any positive interval of time.

\textit{Shape distance without volume constraint.}
Our investigations in this paper were motivated in part by 
an expanded notion of shape distance that was introduced and 
examined by Schmitzer and Schn\"orr in \cite{SchSch2013}. 
These authors considered a shape distance determined 
by restricting the Wasserstein metric to smooth paths 
of `shape measures' consisting of uniform distributions on 
bounded open sets in $\R^2$ with connected smooth boundary.  
This allows one to naturally compare shapes of different volume.
In our present investigation, the only
 smoothness properties of shapes and paths that we require are
those intrinsically associated with Wasserstein distance.
Thus, we investigate the geometry of a `submanifold' of 
the Wasserstein space consisting of uniform distributions on 
shapes regarded as arbitrary bounded measurable sets in $\R^d$.
As we will see in Section~\ref{s:discuss} below,
geodesics for this extended shape distance correspond to a modified water-wave
system with spatially uniform compressibility and zero \textit{average} pressure.
In Theorem~\ref{t:distC} below we extend the result of Theorem~\ref{th1},
for volume-constrained paths of shapes, to deal with
paths of uniform measures connecting two arbitrary bounded measurable sets. 
We show that the extended shape distance 
again agrees with the Wasserstein distance between the endpoints.
The proof follows directly from the construction of concatenated Euler
sprays used to prove Theorem~\ref{t:complete}(b).
%
%
%

\subsection{Related work on geometry of image and shape spaces}\label{ss:shapespace}

The shape distance that we defined in \eqref{d:ds} is related to a large
body of work in imaging science and signal processing.

The general problem of finding good ways to compare two signals 
(such as time series,  images, or shapes)
is important in a number of application areas, 
including computer vision, machine learning, and computational anatomy.
The idea to use deformations as a means of comparing images goes back to
pioneering work of D'Arcy Thompson \cite{Thompson}.  

Distances derived from optimal transport theory 
(Monge-Kantorovich, Wasserstein, or earth-mover's distance)
have been found useful in analyzing images 
by a number of workers~\cite{GanMacshape, HZTA, RubTomGui, SchSch15, WOSLCR, LOT}.  
The transport distance with quadratic cost (Wasserstein distance)  is special 
as it provides a (formal) Riemannian structure  on spaces of measures 
with fixed total mass \cite{AGS, Otto01, Villani03}. 

Methods which endow the space of signals with the metric structure 
of a Riemannian manifold are of particular interest, 
as they facilitate a variety of image processing tasks. 
This geometric viewpoint, pioneered by 
Dupuis, Grenander \& Miller \cite{DuGrMi98, Grenander_Miller_CA}, 
Trouv\'e \cite{Trouv95}, Younes \cite{Youn98} and collaborators,
has motivated the study of a variety of metrics on spaces of images 
over a number of years---see \cite{ DuGrMi98, Holm2013,  Holm2009,  SchSch2013, 
Younes-book} for a small selection. 

The main thrust of these works is to
study Riemannian metrics and the resulting distances in the space of image
deformations (diffeomorphisms).
Connections with the Arnold viewpoint of fluid
mechanics were noted from the outset \cite{Youn98}, and   have been further
explored by Holm, Trouv\'e, Younes and others
\cite{Holm2013,Holm2009,Younes-book}.  This work has led to the 
{\em Euler-Poincar\'e theory of metamorphosis} \cite{Holm2009}, which sets up 
a formalism for analyzing  least-action principles based on Lie-group
symmetries generated by diffeomorphism groups.

A different way to consider shapes is to study them only via their
boundary, and consider Riemannian metrics defined in terms of normal
velocity of the boundary.  Such a point of view has been taken by Michor,
Mumford and collaborators \cite{BrMiMu14, MicMum05, MicMum06,
YMSM_metric_shape}. As they show in \cite{MicMum05}, a metric given by only the
$L^2$ norm of normal velocity does not lead to a viable geometry, as any
two states can be connected by an arbitrarily short curve.  On the other
hand it is shown in \cite{BrMiMu14} that if two or more derivatives of the
normal velocity are penalized, then the resulting metric on the shape space
is geodesically complete. 

In this context, we note that what our work shows is that if the metric is
determined by the $L^2$ norm
of the transport velocity in the bulk, then the global metric
distance is not zero, but that it is still degenerate in the sense that a
length-minimizing geodesic typically may not exist in the shape space.  
While our results do not directly involve smooth deformations of smooth shapes,
it is arguably interesting to consider shape spaces which 
permit `pixelated' approximations, and our results apply in that context.

We speculate that to create a shape distance that (even locally) admits
length-minimizing paths in the space of shapes, one needs to prevent the
creation a large perimeter at negligible cost.  This is somewhat analogous
to the motivation for the metrics on the space of curves considered by
Michor and Mumford \cite{MicMum05}.  Possibilities include introducing a
term in the metric which penalizes deforming the boundary, or a term which
enforces greater regularity for the vector fields considered. 

A number of existing works    
obtain regularity of geodesic paths and resulting diffeomorphisms
by considering Riemannian metrics given in terms of 
(second-order or higher) derivatives of velocities, 
as in the Large Deformation Diffeomorphic Metric Mapping (LDDMM) 
approach of \cite{BMTY05}, see \cite{BrVi17}. 
Metrics based on conservative transport which penalize only one derivative of the
velocity field are connected with viscous dissipation in fluids and have
been considered by Fuchs et al.~\cite{FuchsEtal}, 
Rumpf, Wirth and collaborators \cite{RuWi13, WBRS11}, as
well as by Brenier, Otto, and Seis \cite{BOS11}, who established a
connection to optimal transport.

\subsection{Plan}
The plan of this paper is as follows. In section~2 we collect some basic
facts and estimates that concern geodesics for Monge-Kantorovich/Wasserstein distance. 
In section~\ref{s:geo} we derive formally the geodesic equations for paths of shape
densities and describe the special class of ellipsoidal solutions. 
The construction of Euler sprays and the proof of Theorem~\ref{th2}
is carried out in section~\ref{s:spray}.  Theorems~\ref{th1} and \ref{t:complete} 
are proved in section~\ref{s:dsdw}. 
In section~\ref{s:weaklimits} we study weak convergence of Euler sprays 
and provide the proof of theorem~\ref{th3}.
The connection between Wasserstein geodesics and 
a relaxed least-action priniciple motivated by Brenier's work 
is developed in section~\ref{s:brenier}.
The main part of the paper concludes in section \ref{s:discuss} 
with a treatment of the extended notion of shape distance
related to that
examined by Schmitzer and Schn\"orr in \cite{SchSch2013}. 
Three appendices provide (a) proofs of a few basic facts about subgradients,
(b) a treatment of the $\tlp$ topology used in subsection~\ref{r:TLp},
and (c) an existence proof for the relaxed least-action model
for fluid mixtures.

%
%
\section{Preliminaries: Wasserstein geodesics between open shapes}\label{s:prelim}

In this section we recall some basic properties of the 
standard minimizing geodesic paths (displacement interpolants) for the
Wasserstein or Monge-Kantorovich distance 
between shape densities on open sets, and establish some basic estimates.  
Two properties that are key in the sequel are that the density $\rho$ is
(i) smooth on an open subset of full measure, and (ii) it is
\emph{convex} along the corresponding particle paths, see Lemma 2.1.


Let $\Omega_0$ and $\Omega_1$ be two bounded open sets 
in $\R^d$ 
with equal volume. Let $\mu_0$ and $\mu_1$ be measures with respective densities
\[
\rho_0=\one_{\Omega_0}\,,\qquad \rho_1=\one_{\Omega_1}.
\]
As is well known \cite{Bre91, Knosmi84}, there exists a convex
function $\psi$ such that the a.e.-defined map 
$T = \nabla \psi$ (called the \textit{Brenier map}
in~\cite{Villani03}) is the optimal transportation map between $\Omega_0$ and
$\Omega_1$, pushing $\mu_0$ forward to $\mu_1$,
corresponding to the quadratic cost. 
Moreover, this map is unique a.e.~in $\Omega_0$;
see \cite{Bre91} or \cite[Thm. 2.32]{Villani03}.

McCann \cite{McCann97} later introduced a natural curve $t\mapsto \mu_t$
that interpolates between $\mu_0$ and $\mu_1$, called the 
\textit{displacement interpolant}, which 
can be described as the push-forward of the measure $\mu_0$ 
by the interpolation map $T_t$ given by 
\begin{equation}\label{e:Tt}
T_t(z) = (1-t)z+t\nabla \psi(z), \qquad 0\le t\le 1.
\end{equation}
Because $\psi$ is convex, $\nabla\psi$ is monotone, satisfying
$\ip{\nabla\psi(z)-\nabla\psi(\hat z),z-\hat z}\ge0$ for all $z$, $\hat z$.
Hence the interpolating maps $T_t$ are injective for $t\in[0,1)$, satisfying 
\begin{equation}\label{e:Ttbound}
|T_t(z)-T_t(\hat z)|\ge (1-t)|z-\hat z| \,.
\end{equation}
Note that particle paths $z \mapsto T_t(z)$ follow straight lines with constant velocity:
\begin{equation}\label{e:v}
v(T_t(z),t) = \nabla\psi(z)-z.   
\end{equation}
Furthermore \cite{BenBre00}, $\mu_t$ has density $\rho_t$ that satisfies the continuity equation 
\begin{equation} \label{e:Wv}
\partial_t \rho + \dive(\rho v)  = 0,
\end{equation}
and in terms of these quantities, the Wasserstein distance satisfies 
\begin{equation}\label{d:dwdef}
d_W(\mu_0,\mu_1)^2 = 
\int_{\Omega_0} |\nabla\psi(z)-z|^2\,dz = \int_0^1\int_{\Omega_t}\rho |v|^2\,dx\,dt\ .
\end{equation}

The displacement interpolant has the property that
\begin{equation}\label{e:tsmu}
d_W(\mu_s,\mu_t) = (t-s) d_W(\mu_0,\mu_1), \quad 0\le s\le t\le 1.
\end{equation}
The property \eqref{e:tsmu} implies that the displacement interpolant
is a \textit{constant-speed geodesic} (length-minimizing path) 
with respect to Wasserstein distance.  The displacement interpolant 
$t\mapsto \mu_t$ is the \textit{unique} constant-speed geodesic
connecting $\mu_0$ and $\mu_1$,
due to the uniqueness of the Brenier map and Proposition 5.32 of
\cite{Santa} (or see \cite[Thm. 3.10]{AGuser}). 
For brevity the path $t\mapsto\mu_t$ is called the \textit{Wasserstein geodesic}
from $\mu_0$ to $\mu_1$. 

At this point it is convenient to mention that the result of Theorem~\ref{th4}, 
providing a sharp criterion for the existence of a minimizer for the shape distance in \eqref{d:ds},  will be derived by combining 
the uniqueness property of Wasserstein geodesics
with the result of Theorem~\ref{th1}---see the end of section~\ref{s:dsdw} below.

We note here that the $L^\infty$ transport distance may be defined 
as a minimum over maps $S$ that push forward the measure $\mu_0$ to $\mu_1$
\cite[Thm. 3.24]{Santa} and satisfies the estimate
\begin{equation}\label{d:dinfty}
    \begin{split}
        d_\infty(\mu_0,\mu_1) &=
    \min\{ \|S-{\id}\|_{L^\infty(\mu_0)} : S_\sharp\mu_0=\mu_1\} 
       \\ & \geq
|\Omega_0|^{-1/2} 
    \min\{ \|S-{\id}\|_{L^2(\mu_0)} : S_\sharp\mu_0=\mu_1\} 
        \\ & =
|\Omega_0|^{-1/2} 
        d_W({\Omega_0},{\Omega_1}).
    \end{split}
\end{equation}

\section{Geodesics and incompressible fluid flow}\label{s:geo}

\subsection{Incompressible Euler equations for smooth critical paths}\label{ss:geo}

In this subsection, for completeness we derive the Euler fluid equations 
that formally describe smooth geodesics (paths with stationary action) 
for the shape distance in \eqref{d:ds}. 
To cope with the problem of moving domains we work in a Lagrangian framework, 
computing variations with respect to flow maps that preserve density and 
the endpoint shapes $\Omega_0$ and $\Omega_1$.

Toward this end, suppose that 
\begin{equation}\label{d:Omega}
\Qdom = \bigcup_{t\in[0,1]} \Omega_t\times\{t\} \ \ \subset \R^d\times[0,1]
\end{equation}
is a space-time domain generated by smooth deformation of $\Omega_0$ due to a smooth velocity field $v\colon\bar\Qdom\to\R^d$.
That is, the $t$-cross section of $\Qdom$ is given by
\begin{equation}\label{e:Omegat}
\boxed{
\ \ \Omega_t=X(\Omega_0,t), \ %
}
\end{equation} where 
$X$ is the Lagrangian flow map associated to $v$, satisfying
\begin{equation}\label{e:Xt}
\dot X(z,t) = v(X(z,t),t),
\qquad
X(z,0)= z,
\end{equation}
for all $(z,t)\in \Omega_0\times[0,1]$.  

For any (smooth) extension of $v$
to $\R^d\times[0,1]$, the solution of the mass-transport equation in \eqref{e:inf1a} with given initial density $\rho_0$ supported in $\bar\Omega_0$
is
\[
\rho(x,t)= \rho_0(z) \det\left( \frac{\D X}{\D z}(z,t)\right)\inv, \quad x=X(z,t)\in\Omega_t,
\]
with $\rho=0$ outside $\Qdom$. 

Considering a family $\eps\mapsto X_\eps$ of flow maps smoothly depending 
on a variational parameter $\eps$, the variation 
$\delta X = (\D X_\eps/\D\eps)|_{\eps=0}$ induces a variation in density 
$\delta\rho=(\D\rho_\eps/\D\eps)|_{\eps=0}$ satisfying
\begin{equation}\label{e:drho}
-\frac{\delta \rho}\rho = \delta\log \det\left( \frac{\D X}{\D z}(z,t)\right)
= \tr \left(\frac{ \D\delta X}{\D z }\left(\frac{\D X}{\D z}\right)\inv\right)
\end{equation}
Introducing $\tilde v(x,t)=\delta X(z,t)$, $x=X(z,t)$, it follows
\begin{equation}\label{delrho}
-\frac{\delta\rho}\rho = \sum_j \frac{\D\tilde v_j}{\D x_j} =  \nabla\cdot \tilde v.
\end{equation}
For variations that leave the density invariant, necessarily $\nabla\cdot\tilde v=0$.

We now turn to consider the variation of the action or transport cost
as expressed in terms of the flow map:
\begin{equation}\label{d:calA}
{\mathcal A} =  \int_0^1\int_{\R^d} \rho(x,t)|v(x,t)|^2 dx\,dt
= \int_0^1\int_{\Omega_0}  |\dot X(z,t)|^2 dz\,dt\,.
\end{equation}
For flows preserving $\rho=1$ in $\Qdom$, 
of course $\nabla\cdot v=0$.
Computing the first variation of ${\mathcal A}$ about such a flow,
after an integration by parts in $t$ and changing to Eulerian variables, we find
\begin{align}\label{e:varA}
\frac{ \delta{\mathcal A}}2 &=  
\int_0^1\int_{\Omega_0}  \dot X \cdot \delta\dot X \,dz\,dt
\nonumber\\&
=  \left.\int_{\Omega_0} \dot X\cdot \delta X\,dz\,\right|_{t=1}
-\int_0^1  \int_{\Omega_0}
\ddot X \cdot \delta X \,dz\,dt
\nonumber\\&
= \left.\int_{\Omega_t} v\cdot \tilde v\,dx\,\right|_{t=1}
-\int_0^1  \int_{\Omega_t}
(\D_t v+v\cdot\nabla v)\cdot \tilde v \,dx\,dt.
\end{align}

Recall that any $L^2$ vector field $u$ on $\Omega_t$ 
has a unique Helmholtz decomposition 
as the sum of a gradient and a field $L^2$-orthogonal 
to all gradients, which is divergence-free with zero normal component
at the boundary:
\begin{equation}\label{e:Helm}
u = \nabla p + w, \qquad \nabla\cdot w = 0 \ \ \te{in $\Omega_t$},
\quad w\cdot n = 0\ \ \te{on $\D\Omega_t$}.
\end{equation}
If we loosen the requirement that $w\cdot n=0$ on the boundary,
it is still the case that
\[
\int_{\D\Omega_t} w\cdot n\,dS = 
\int_{\Omega_t} \nabla\cdot w \,dx = 0,
\]
It follows that the space orthogonal to all divergence-free fields on $\Omega_t$
is the space of gradients $\nabla p$ such that $p$ is constant on the
boundary, and we may take this constant to be zero:
\begin{equation}\label{e:zerop}
\boxed{
\ \  p=0 \quad\text{on $\D\Omega_t$.}
}
\end{equation}
Requiring $\delta{\mathcal A}=0$ for arbitrary virtual displacements having 
$\nabla \cdot \tilde v=0$ (and $\tilde v=0$ at $t=1$ at first), we 
find that necessarily $u=-(\D_t v+v\cdot\nabla v)$ is such a gradient.
Thus the incompressible Euler equations hold in $\Qdom$:
\begin{equation}\label{e:euler-v}
\D_t v+v\cdot\nabla v+\nabla p=0\,,\quad \nabla\cdot v=0 \quad\text{in $\Qdom$},
\end{equation}
where $p:\bar\Qdom\to\R$ is smooth and satisfies~\eqref{e:zerop}.

Finally, we may consider variations $\tilde v$ that do not vanish at $t=1$.
However, we require $\tilde v\cdot n=0$ on $\D\Omega_1$ in this case
because the target domain $\Omega_1$ should be fixed. 
That is, the allowed variations in the flow map $X$ must fix the image
at $t=1$: 
\begin{equation}
\Omega_1= X(\Omega_0,1).
\end{equation}
The vanishing of the integral term at $t=1$ in \eqref{e:varA} 
then leads to the requirement that $v$ is a gradient at $t=1$.
For $t=1$ we must have
\begin{equation}\label{e:vphi}
v=\nabla \phi
\quad\te{in $\Omega_t$}.
\end{equation}

We claim this gradient representation actually must hold for all $t\in[0,1]$.
Let $v=\grad\phi+w$ be the Helmholtz decomposition of $v$, and 
for small $\eps$ consider the family of flow maps generated by 
\begin{equation}
\dot X_\eps(z,t) = (v+\eps w)(X_\eps(z,t),t)\,,
\qquad
X_\eps(z,0)= z.
\end{equation}
Corresponding to this family, the action from \eqref{d:calA} takes the form
\begin{equation}
{\mathcal A}_\eps = 
\int_0^1\int_{\Omega_0}  |\dot X_\eps(z,t)|^2 dz\,dt
= \int_0^1\int_{\Omega_t} |\grad\phi|^2 + |(1+\eps)w|^2\,dx\,dt
\end{equation}
Because $w\cdot n=0$ on $\D\Omega_t$, the domains $\Omega_t$ do not
depend on $\eps$, and the same is true of $\nabla\phi$ and $w$, so
this expression is a simple quadratic polynomial in $\eps$. Thus
\begin{equation}
\frac12 \left.\frac{d{\mathcal A}_\eps}{d\eps} \right|_{\eps=0} = \int_0^1\int_{\Omega_t}  |w|^2\,dx\,dt \,
\end{equation}
and consequently it is necessary that $w=0$ if $\delta\mathcal A=0$.
This proves the claim.

The Euler equation in \eqref{e:euler-v} is now a spatial gradient, and one 
can add a function of $t$ alone to $\phi$ to ensure that
\begin{equation}\label{e:phit}
\boxed{
\D_t\phi + \frac12|\nabla\phi|^2 + p = 0,
\quad\Delta \phi = 0
\quad\te{in $\Omega_t$}.
}
\end{equation}
The equations boxed above,  including \eqref{e:phit}
together with the zero-pressure boundary condition \eqref{e:zerop}
and the kinematic condition that the boundary of $\Omega_t$ moves with
normal velocity $v\cdot n$ (coming from \eqref{e:Omegat}-\eqref{e:Xt}), comprise what we shall call the 
\emph{Euler droplet} equations, for incompressible, inviscid, potential flow 
of fluid with zero surface tension and zero pressure at the boundary.

\begin{definition}\label{d:Eulerdroplet}
A smooth \emph{solution} of the Euler droplet equations is a triple 
$(\Qdom,\phi,p)$ such that $\phi,p\colon\bar\Qdom\to\R$ are smooth
and  the equations \eqref{d:Omega}, \eqref{e:Omegat}, \eqref{e:Xt}, 
\eqref{e:vphi}, \eqref{e:phit}, \eqref{e:zerop}
all hold.
\end{definition}

\begin{proposition}\label{p:Eulersmooth}
For smooth flows $X$ that deform $\Omega_0$ as above,
that respect the density constraint $\rho=1$ and fix 
$\Omega_1=X(\Omega_0,1)$,
the action $\mathcal A$ in \eqref{d:calA}
is critical with respect to smooth variations if and only if 
$X$ corresponds to a smooth solution of the Euler droplet equations.
\end{proposition}

\subsection{Weak solutions and Galilean boost}

Here we record a couple of simple basic properties
of solutions of the Euler droplet equations.

\begin{proposition}\label{p:Eulerweak1}
Let $(\Qdom,\phi,p)$ be a smooth solution of the Euler droplet equations.
Let $\rho=\one_\Qdom$ and $v=\one_\Qdom\grad\phi$,
and extend $p$ as zero outside $\Qdom$. 

\noindent
(a) The  Euler equations \eqref{e:euler1}-\eqref{e:euler2}
hold in the sense of distributions on 
$\R^d\times[0,1]$.

\noindent
(b)  The mean velocity 
\begin{equation}
\bar v = \frac1{|\Omega_t|} \int_{\Omega_t} v(x,t)\,dx
\end{equation}
is constant in time, and the action decomposes as 
\begin{equation}
{\mathcal A} =  \int_0^1\int_{\Omega_t} |v-\bar v|^2 dx\,dt + 
|\Omega_0| |\bar v|^2\,.
\end{equation}
(c) Given any constant vector $b\in\R^d$, another smooth solution
$(\hat Q,\hat\phi,\hat p)$
of the Euler droplet equations is given by a Galilean boost, via
\begin{gather}
\hat Q= \bigcup_{t\in[0,1]} (\Omega_t+bt)\times\{t\} \,,
\\
\hat\phi(x+bt,t) = \phi(x,t)+b\cdot x+\frac12|b|^2t\,,\quad
\hat p(x+bt,t)=p(x,t)\,.
\label{e:boostphip}
\end{gather}
\end{proposition}

\begin{proof} To prove (a), what we must show is the following:
For any smooth test functions 
$q\in C_c^\infty(\R^d\times[0,1],\R)$ and 
$\tilde v\in C_c^\infty(\R^d\times[0,1],\R^d)$,
\begin{align}
\int_Q (\D_t q + v\cdot \nabla q) \,dx\,dt &= 
\left. \int_{\Omega_t} q\,dx \right|^{t=1}_{t=0}
\label{e:wk1Qa}
\\
\int_Q v\cdot (\D_t \tilde v +  v\cdot\nabla \tilde v)+
p\nabla \cdot \tilde v \,dx\,dt &= 
\left. \int_{\Omega_t} \tilde v\cdot v\,dx \right|^{t=1}_{t=0}
\label{e:wk1Qb}
\end{align}
Changing to Lagrangian variables via $x=X(z,t)$,
writing $\hat q(z,t)=q(x,t)$, and using
incompressibility, equation \eqref{e:wk1Qa} is equivalent to 
\begin{align}
\int_0^1\int_{\Omega_0} \frac{d}{dt}\hat q(z,t)\,dz\,dt
 &= 
\left. \int_{\Omega_0} \hat q(z,t)\,dz \right|^{t=1}_{t=0} \,.
\end{align}
Evidently this holds.
In \eqref{e:wk1Qb}, we integrate the pressure term by parts,
and treat the rest as in \eqref{e:varA} to find that \eqref{e:wk1Qb}
is equivalent to 
\begin{align}
\int_Q (\D_t v+v\cdot\nabla v+\nabla p)\cdot\tilde v \,dx\,dt = 0.
\end{align}
Then (a) follows. The proof of parts (b) and (c) is straightforward.
\end{proof}

\subsection{Ellipsoidal Euler droplets}

The initial-value problem for the Euler droplet equations is a difficult
fluid free boundary problem, one that may be treated by the methods
developed by Wu \cite{Wu97,Wu99}. For flows with vorticity and smooth enough
initial data, smooth solutions for short time have been shown to exist in 
\cite{Lindblad,CoutShko2007,CoutShko2010}.  

In this section,  we describe simple, particular Euler droplet 
solutions for which the fluid domain $\Omega_t$ remains ellipsoidal for all $t$.
Our main result is the following.
\begin{proposition}\label{p:Edrop} Given a constant $r>0$,
let $a(t)=(a_1(t),\ldots,a_d(t))$ be any constant-speed geodesic on
the surface in $\R_+^d$ determined by the relation
\begin{equation}\label{ed:c}
a_1\cdots a_d=r^d. 
\end{equation}
Then this determines an Euler droplet solution $(\Qdom,\phi,p)$
with $\Omega_t$ equal to the ellipsoid $E_{a(t)}$ 
given by
\begin{equation}\label{ed:omega}
E_a = \Bigl\{ x\in\R^d: \sum_j (x_j/a_j)^2 <1\Bigr\},
\end{equation}
and potential and pressure given by 
\begin{equation}\label{ed:phip}
\phi(x,t) = \frac12 \sum_j  \frac{\dot a_jx_j^2}{a_j}  - \beta(t)\,,
\qquad
p(x,t) = \dot\beta \left(1-\sum_j \frac{x_j^2}{a_j^2}\right) ,
\end{equation}
with 
\begin{equation}\label{ed:beta}
\dot\beta(t) = \frac12
\frac{\sum_j \dot a_j^2/a_j^2}{\sum_j 1/ a_j^2 } .
\end{equation}
\end{proposition}

\hide{\red
\subsubsection{Droplets in dimension $d=2$.}
We seek incompressible flows inside a time-dependent elliptical domain where
\begin{equation}\label{eq2:domain}
\frac{x^2}{a(t)^2}+\frac{y^2}{b(t)^2}< 1,
\end{equation}
with the geometric mean $r=(ab)^{1/2}$ constant in time for volume conservation. 
We will find such flows as time-stretched straining flows $(X,Y)$, 
satisfying
\[
(\dot X,\dot Y) = v(X,Y,t) = \alpha(t)(X,-Y) \,.
\]
Such flows have velocity potential satisfying $v=\nabla\phi$, with 
\begin{equation}\label{eq2:phi}
\phi(x,y,t) = \frac12\alpha(t) (x^2-y^2) -\beta(t),
\end{equation}
\[ \quad \D_t\phi  = \frac12 \dot\alpha(x^2-y^2)-\dot\beta,
\quad \frac12|\nabla\phi|^2 = \frac12\alpha^2(x^2+y^2) \,.
\]
To satisfy the Bernoulli equation we require $\D_t\phi+\frac12|\nabla\phi|^2=0$
on the boundary of the ellipse $(x,y)=(a\cos\theta,b\sin\theta)$, or 
\[
(\dot \alpha +\alpha^2)a^2\cos^2\theta 
+(-\dot\alpha+\alpha^2)b^2\sin^2\theta = 2\dot\beta
\]
In order for this to hold independent of $\theta$,  we require
\[
(\dot \alpha +\alpha^2)a^2  = -(\dot\alpha-\alpha^2)b^2=2\dot\beta .
\]
Due to the motion of the boundary points $(a,0)$, $(0,b)$ we need 
\[
\dot a = \alpha a, \quad \dot b = -\alpha b,
\]
whence
\[
2\dot\beta = 
a \ddot a = \frac{2b^2 \dot a^2}{ (a^2+b^2)} = \frac{2r^4 \dot a^2}{ (a^4+r^4)} 
\]
because $r^2=ab$ is constant. Notice $\ddot a>0$ in all cases. 
There is a first integral (because kinetic energy is conserved)
which we can find by writing
\[
\frac{\ddot a}{\dot a} =  2\dot a\left(\frac1a-\frac{a^3}{r^4+a^4}\right),
\]
whence we find that $a(t)$ and $b(t)$ are 
determined by solving
\begin{equation}\label{eq2:abdot}
 \frac{\dot a}a = \frac{c}{\sqrt{a^2+b^2}} = -\frac{\dot b}b = \alpha(t).
\end{equation}
for some real constant $c$.
From the derivation of the Bernoulli equation, inside the ellipse the pressure is
\begin{equation}\label{eq2:p}
p =- \D_t \phi - \frac12|\nabla\phi|^2 = \dot\beta\left(1-\frac{x^2}{a^2}-\frac{y^2}{b^2}\right)\,.
\end{equation}
where $\dot\beta$ is recovered from the equation
\begin{equation}\label{eq2:betadot}
 \dot\beta(t) = \left( \frac{c ab}{a^2+b^2}\right)^2.
\end{equation}

To summarize, an elliptical Euler droplet solution $(\Qdom,\phi,p)$ 
is determined in terms of any solution $(a(t),b(t))$ of \eqref{eq2:abdot} 
(with any real $c$) by \eqref{eq2:domain}, \eqref{eq2:phi}, 
\eqref{eq2:p}, and \eqref{eq2:betadot}.
We note that the speed of motion of the point $(a,b)$ on the hyperbola
$ab=r^2$ is constant: by \eqref{eq2:abdot},
\begin{equation}
\dot a^2+\dot b^2 = c^2.
\end{equation}
In the context of the fixed-endpoint problem, then, $|c|$ is
the distance along the hyperbola betweeen $(a(0),b(0))$ and $(a(1),b(1))$.
} 

%
\begin{proof}
The flow $X$ associated with a velocity potential of the form in \eqref{ed:phip}
must satisfy
\begin{equation}\label{ed:flowX}
\dot X_j = \alpha_j(t)X_j, \qquad \alpha_j = \frac{\dot a_j}{a_j}, \quad j=1,\ldots,d.
\end{equation}
Then $(X_j/a_j)\,\dot{}=0$ for all $j$, so the flow is purely dilational along
each axis and consequently ellipsoids are deformed to ellipsoids as claimed. 
Note that incompressibility corresponds to the relation
\[
\Delta\phi =  \sum_j \alpha_j = \sum_j \frac{\dot a_j}{a_j} = \frac{d}{dt}\log (a_1\cdots a_d )= 0.
\]

From \eqref{ed:phip} we next compute
\[ \D_t\phi_t+\frac12|\nabla\phi|^2 = -\dot\beta +
\frac12\sum_j (\dot\alpha_j+\alpha_j^2)x_j^2
= -\dot\beta + \frac12\sum_j \frac{\ddot a_j x_j^2}{a_j} \,.
\]
This must equal zero on the boundary where $x_j=a_j s_j$ with $s\in S_{d-1}$ arbitrary.
We infer that for all $j$,
\begin{equation}\label{ed:dotb}
a_j\ddot a_j = 2\dot\beta \,.
\end{equation}
The expression for pressure in \eqref{ed:phip} in terms of $\dot\beta$ then 
follows from \eqref{e:phit},
and $p=0$ on $\D\Omega_t$.

We  recover $\dot\beta$ by differentiating
the constraint twice in time. We find
\begin{eqnarray*}
0 &=& \sum_j \left(\sum_k a_1\cdots a_d \frac{\dot a_k}{a_k}\frac{\dot a_j}{a_j}
+ a_1\dots a_d \frac{a_j\ddot a_j-\dot a_j^2}{a_j^2} \right)
\\ &=& 0 + \sum_j \frac{2\dot\beta-\dot a_j^2}{a_j^2}
\end{eqnarray*}
whence \eqref{ed:beta} holds.

To get the first integral that corresponds to kinetic energy, multiply 
\eqref{ed:dotb}
by $2\dot a_j/a_j$
and sum to find
\[
0 = \sum_j \dot a_j\ddot a_j , \quad\te{whence}\quad 
\sum_j \dot a_j^2 = c^2 
\]
and we see that $c=|\dot a(t)|$ is the constant speed of motion.

It remains to see that \eqref{ed:dotb} are the geodesic equations on the
constraint surface. This follows because \eqref{ed:dotb} says that $\ddot a$
is parallel to the gradient of $F(a)=\sum_j\log a_j$, and the constraint
\eqref{ed:c} corresponds to staying on the level set $F(a)=\log r^d$.
%
This finishes the demonstration of Proposition~\ref{p:Edrop}.
\end{proof}

\begin{remark}\label{r:aconvex}
For later reference, we note that $\ddot a_j\ge 0$ for all $t$, due to 
\eqref{ed:dotb} and \eqref{ed:beta}. 
\end{remark}
\begin{remark} Given any two points on the surface described by 
the constraint \eqref{ed:c}, there exists a constant-speed geodesic
connecting them. This fact is
 a straightforward consequence of the Hopf-Rinow theorem
 on geodesic completeness~\cite[Theorem 1.7.1]{Jost}, 
because all closed and bounded subsets on the surface are compact.
\end{remark}
\begin{remark}\label{r:induce}
The Euclidean metric on the hyperboloid-like surface arises, in fact, as the metric induced by the Wasserstein distance \cite[Chap. 15]{Villani09}, because,
given any incompressible path $t\mapsto X(\cdot,t)$ 
of dilations, 
satisfying \eqref{ed:flowX} for some smooth $\alpha(t)$ 
and with $a_1\cdots a_d=r^d$, 
\[
\int_{\Omega_t} |v|^2\,dx = \int_{\Omega_t}\sum_j \alpha_j^2 x_j^2\,dx
 = \sum_j  \dot a_j^2 \int_{|z|\le 1} z_j^2\,dz\, r^d =
 \frac{\omega_dr^d}{d+2} \sum_j \dot a_j^2\,,
\]
where $\omega_d=|B(0,1)|$ is the volume of the unit ball in $\R^d$.
For a geodesic, this expression is constant for $t\in[0,1]$ and equals the action
${\mathcal A}_a$ in \eqref{d:calA} for the ellipsoidal Euler droplet.
\end{remark}

\subsection{Ellipsoidal Wasserstein droplets}\label{ss:Wdrop}

Let $(\Qdom,\phi,p)$ be an ellipsoidal Euler droplet solution as given by 
Proposition~\ref{p:Edrop}, so that  $\Omega_0=E_{a(0)}$ and $\Omega_1=E_{a(1)}$ are co-axial ellipsoids. 
We will call the optimal transport map $T$ between these co-axial ellipsoids 
an \emph{ellipsoidal Wasserstein droplet}. 
This is described and related to the Euler droplet as follows.

Given $A\in\R^d$, let $D_A=\diag(A_1,\ldots,A_d)$
denote the diagonal matrix with diagonal $A$. 
Then, given $\Omega_0=E_{a(0)}$, $\Omega_1=E_{a(1)}$ as above, 
the particle paths for the Wasserstein geodesic between the 
corresponding shape densities are given by linear interpolation via
\begin{equation}
T_t(z)= D_{A(t)}D_{A(0)}\inv z\,, \qquad
A(t) =(1-t) a(0)+t a(1)\,.
\end{equation}
Note that a point $z\in E_A$ if and only if 
$D_A\inv z$ lies in the unit ball $B(0,1)$ in $\R^d$.
Thus the Wasserstein geodesic flow takes ellipsoids to ellipsoids:
\[
T_t(\Omega_0) = E_{A(t)}  
\,, \quad t\in[0,1].
\]

Let $a(t)$, $t\in[0,1],$ be the geodesic on the hyperboloid-like
surface that corresponds to the Euler droplet that we started with. 
Recall that $\Omega_t= E_{a(t)}$ from Proposition~\ref{p:Edrop}.
Because each component $t\mapsto a_j(t)$ is convex by
Remark~\ref{r:aconvex}, it follows that for each $j=1,\ldots,d$,
\begin{equation}\label{i:aA}
a_j(t)\le A_j(t), \quad t\in[0,1].
\end{equation}
Because $E_{A}=D_A B(0,1)$,
we deduce from this the following important nesting property,
which is illustrated in 
Figure~\ref{fig:eulerdrop}  
(where for visibility the ellipses at times $t=\frac12$ and $t=1$ are
offset horizontally by $\frac{b}{2}$ and $b$ respectively).

\begin{proposition}\label{p:nest}
Given any corresponding elliptical Euler droplet and Wasserstein droplet
that deform one ellipsoid $\Omega_0=E_{a(0)}$ to another $\Omega_1=E_{a(1)}$,
the Euler domains remain nested inside their Wasserstein counterparts,
with
\begin{equation}
X(\Omega_0,t)=\Omega_t  \subset T_t(\Omega_0), \quad t\in[0,1].
\end{equation}
\end{proposition}

\begin{remark}\label{r:nest}
In terms of the notation of this subsection, the straining flow $X$ associated
with the Euler droplet is given by $X(z,t)=D_{a(t)}D_{a(0)}\inv z$ in terms of the
constant-speed geodesic $a(t)$ of Proposition~\ref{p:Edrop}.
Due to \eqref{i:aA}, this flow satisfies, for each $j=1,\ldots,d$ and $z\in\R^d$, 
\[
|X_j(z,t)| = \frac{a_j(t)}{a_j(0)} |z_j| \le \frac{A_j(t)}{A_j(0)} |z_j| = |T_t(z)_j|.
\]
For the nesting property $X(\hat\Omega,t)\subset T_t(\hat\Omega)$ 
to hold, convexity of $\hat\Omega$ is not sufficient in general. 
However, a sufficient condition is
that  whenever $\alpha_j\in[0,1]$ for $j=1,\ldots,d$, 
 \[
x=(x_1,\ldots,x_d)\in\hat\Omega
\quad\mbox{implies}\quad
D_\alpha x = (\alpha_1x_1,\ldots,\alpha_nx_n)\in\hat\Omega.
\]
\end{remark}

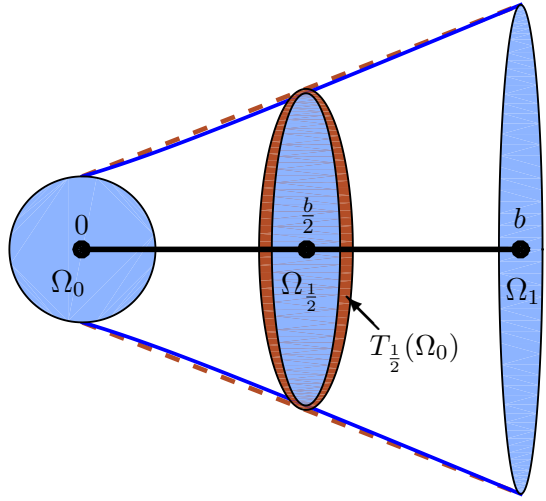
\begin{figure}
\includegraphics[trim=00 00 0 00, clip, height=7cm]{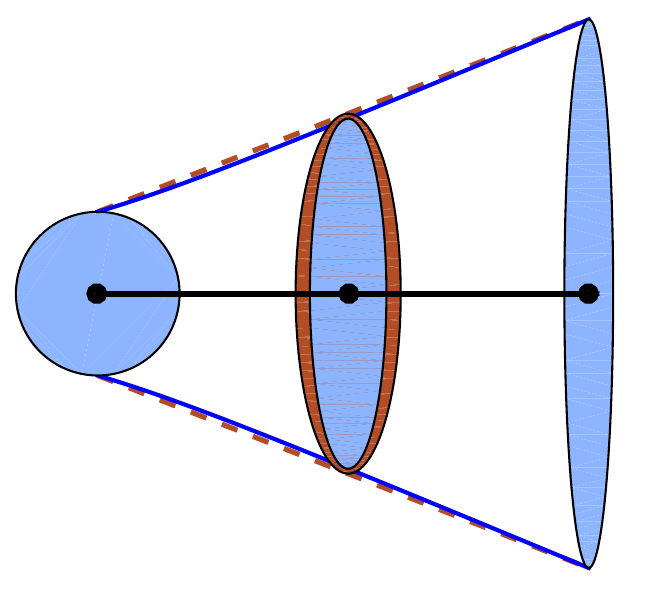}
\put(-199,84){\large $\Omega_0$}
\put(-112,83){\large $\Omega_{\frac12}$}
\put(-27,81){\large $\Omega_1$}
\put(-190,106){$0$}
\put(-105,110){$\frac{b}{2}$}
\put(-24,108){$b$}
\put(-82,60){ $T_{\frac12}(\Omega_0)$}
\thicklines
\put(-75,70){\vector(-1,1){12}}
\caption{Euler droplet (light blue) deforming a circle to an ellipse, nested 
inside a Wasserstein droplet (dark orange). 
Tracks of the center and endpoints of vertical major axis are indicated for both droplets.}
\label{fig:eulerdrop}
\end{figure}

\subsection{Action estimate for ellipsoidal Euler droplets}
For later use below, we describe how to bound the action for a boosted elliptical Euler droplet 
in terms of  action for the corresponding boosted elliptical Wasserstein droplet, 
in the case when the source and target domains
are respectively a  ball and translated ellipse:
\begin{lemma}\label{l:Eactionest}
Given $r>0$, $\hat a\in\R^d_+$ with $\hat a_1\cdots\hat a_d=r^d$, and $b\in \R^d$,  let
\[
\Omega_0=B(0,r), \qquad \Omega_1 = E_{\hat a}+b\,.
\]
Let $a(t)$, $t\in[0,1]$,
be the minimizing geodesic on the surface \eqref{ed:c} with
 \[
 a(0)=\hat r=(r,\ldots,r), \qquad a(1)=\hat a=(\hat a_1,\ldots,\hat a_d)\,.
 \]
Let $(\Qdom,\phi,p)$ be the elliptical Euler droplet solution corresponding to 
the geodesic $a$, and let ${\mathcal A}_a$ denote the corresponding action.
Then
\begin{equation}\label{e:Aa-action}
 d_W(\one_{\Omega_0},\one_{\Omega_1})^2
\le
{\mathcal A}_a \le 
 d_W(\one_{\Omega_0},\one_{\Omega_1})^2
+ \frac{\olam^4}{\ulam^2}\,\omega_d r^{d+2} \,,
\end{equation}
where
\begin{equation}\label{ed:ulol}
\ulam = \min \frac{\hat a_i}r, \qquad \olam = \max \frac{\hat a_i}r\,.
\end{equation}
\end{lemma}

\begin{proof}
First, consider the transport cost for mapping $\Omega_0$ to 
$\Omega_1$.  The (constant) velocity of particle paths
starting at $x\in B(0,r)$ is 
\[
u(x)= (r\inv D_{\hat a}-I) x+ b, 
\]
 and the squared transport cost or action is (substituting $x=rz$)
\begin{align}
 d_W(\one_{\Omega_0},\one_{\Omega_1})^2
& = \int_{B(0,r)} |u(x)|^2\,dx  
=
\sum_j \int_{B(0,r)} \left(\frac{\hat a_j}r-1\right)^2 z_j^2+ b_j^2 \,dz \,
\nonumber
\\&
= {\omega_dr^d} \left( |b|^2+ \frac{|\dot A|^2}{d+2}  \right)\,,
\label{ew:dW}
\end{align}
where $A(t)=(1-t)\hat r+t\hat a$ is the straight-line path from $\hat r$ to $\hat a$.

The mass density inside the transported ellipsoid $T_t(\Omega_0)$ is 
constant in space, given by 
\[
\rho(t) = \det DT_t\inv =  \prod_i  \frac r{A_i(t)} = 
\prod_i \left(1-t+t\frac{\hat a_i}r \right)\inv \,.
\]
Due to Remark~\ref{r:induce},
the corresponding action for the Euler droplet is bounded by that of the constant-volume path found
by dilating the elliptical Wasserstein droplet: Let
\[
\gamma_j(t) = \rho(t)^{1/d}A_j(t)\,.
\]
Then the flow $S_t(z)=r\inv D_{\gamma(t)}z $ is 
dilational and volume-preserving (with $\prod_j \gamma_j(t) \equiv r^d$)
and has zero mean velocity.
The flow $z\mapsto S_t(z)+tb$ takes $\Omega_0$ to $\Omega_1$,  
as on Figure  \ref{fig:eulerdrop}, with action
\begin{align}
{\mathcal A}_\gamma &= 
\int_0^1 \int_{B(0,r)} \sum_j \left( b_j+\frac{\dot\gamma_j z_j}{r}\right)^2 dz\,dt
\nonumber
\\&
= \omega_dr^d\left( |b|^2+ \frac1{d+2}\int_0^1 |\dot \gamma|^2\,dt \right)\,.
\label{ed:Ag}
\end{align}
Note that $\sum_j (\dot\gamma_j/\gamma_j)^2\le \sum_j (\dot A_j/A_j)^2$,
because
 \[
\frac{\dot\gamma_j}{\gamma_j} = \frac{\dot A_j}{A_j}+\frac{\dot\rho}{d\rho}
= \frac{\dot A_j}{A_j} - \frac1d \sum_i \frac{\dot A_i}{A_i}.
\] 
Because $\rho$ is convex we have
$\rho\le 1$, hence $\gamma_j^2\le \max A_i^2$.
Thus
\begin{equation}\label{e:gammabd}
|\dot \gamma|^2 \le (\max A_i^2) \,\sum_j \frac{\dot A_j^2}{A_j^2} 
\le \left(\frac{\max A_i^2}{\min A_i^2}\right) |\dot A|^2
\le \left(\frac{\max \hat a_i^2}{\min \hat a_i^2}\right) |\hat a-\hat r|^2\,.
\end{equation}
Plugging this back into \eqref{ed:Ag} and using \eqref{ew:dW}, we deduce that
\begin{equation}
{\mathcal A}_\gamma \le 
 d_W(\one_{\Omega_0},\one_{\Omega_1})^2
+ \frac{\omega_d r^d}{d+2}
\left(\frac{\max \hat a_i^2}{\min \hat a_i^2} \right)|\hat a-\hat r|^2\,.
\end{equation}
With the notation in \eqref{ed:ulol}, $\ulam$ and $\olam$ respectively
are the maximum and minimum eigenvalues of $DT_t$, and
because $|1-\hat a_i/r|\le\max(1,\hat a_i/r)\le \olam$ for all $i=1,\ldots,d$, 
this estimate implies
\begin{equation}
{\mathcal A}_a \le 
{\mathcal A}_\gamma \le 
 d_W(\one_{\Omega_0},\one_{\Omega_1})^2
 + \frac{d}{d+2} \frac{\olam^4}{\ulam^2}\,\omega_d r^{d+2} \,.
\end{equation}
This yields \eqref{e:Aa-action} and completes the proof.
\end{proof}

\subsection{Velocity and pressure estimates}

Lastly in this section we provide bounds on the velocity $v=\nabla\phi$
and pressure $p$ for the ellipsoidal Euler droplet solutions.
Note that because $1/a_j^2 \le \sum_i (1/a_i^2)$, 
\[
0 \le p \le \dot\beta \le \frac12 \sum_j \dot a_j^2
\le \frac12 \int_0^1 |\dot\gamma|^2\,dt 
\]
Using \eqref{e:gammabd} and the notation in \eqref{ed:ulol},
it follows
\begin{equation}\label{e:pbound}
0\le p\le 
\frac{\olam^4}{\ulam^2}\,r^{2} d\,.
\end{equation}
For the velocity, it suffices to note that in \eqref{ed:flowX},
$|X_j/a_j|\le1$ hence
$
|\dot X|^2 \le \sum_j \dot a_j^2 
$.
Thus the same bounds as above apply and we find 
\begin{equation}\label{e:vbound}
|\nabla \phi|^2 \le  \frac{\olam^4}{\ulam^2}\,r^{2} d.
\end{equation}

Finally, for a boosted elliptical Euler droplet, with velocity
boosted as in \eqref{e:boostphip} by a constant vector $b\in\R^d$, the same
pressure bound as above in \eqref{e:pbound} applies, and the 
same bound on velocity becomes
\begin{equation}\label{e:vboundboost}
|\nabla \hat\phi-b|^2 \le \frac{\olam^4}{\ulam^2}\,r^{2} d.
\end{equation}

We remark that in the constructions that we make in the next section,
for a given distortion ratio $\olam^4/\ulam^2$, the bounds
in \eqref{e:pbound}--\eqref{e:vboundboost}  can be made arbitrarily small
by requiring $r^2$ is small.


\section{Euler sprays}\label{s:spray}

Heuristically, an Euler spray is a countable disjoint superposition of 
solutions of the Euler droplet equations. Recall that the notation
$\sqcup_n \Qdom_n$ means the union of disjoint sets $\Qdom_n$.

\begin{definition}
An \textbf{Euler spray} is a triple $(\Qdom,\phi,p)$,
with $\Qdom$ a bounded open subset of $\R^d\times[0,1]$
and $\phi,p:\Qdom\to\R$,
such that there is a sequence $\{(\Qdom_n,\phi_n,p_n)\}_{n\in\N}$
of smooth solutions of the Euler droplet equations, such that 
$
\Qdom = \sqcup_{n=1}^\infty \Qdom_n
$
is a disjoint union of the sets $\Qdom_n$, and for each $n\in\N$, 
$\phi_n = \phi|_{\Qdom_n}$ and $p_n=p|_{\Qdom_n}$. 
\end{definition}

With each Euler spray that satisfies appropriate bounds
 we may associate a weak solution $(\rho,v,p)$ of the Euler system
 \eqref{e:euler1}-\eqref{e:euler2}.
The following result is a simple consequence of the weak formulation
in \eqref{e:eulerwk1}-\eqref{e:eulerwk2} together with Proposition~\ref{p:Eulerweak1}(a)
 and the dominated convergence theorem.
\begin{proposition}\label{p:Esprayweak}
 Suppose $(\Qdom,\phi,p)$ is an Euler spray such
that $|\nabla\phi|^2$ and $ p$ are integrable on $\Qdom$.
Then with $\rho=\one_\Qdom$ and $v=\one_\Qdom\grad\phi$
and with $p$ extended as zero outside $Q$,
the triple $(\rho,v,p)$ satisfies
the Euler system \eqref{e:euler1}-\eqref{e:euler2}
 in the sense of distributions on $\R^d\times[0,1]$.
\end{proposition}

Our main goal in this section is to prove Theorem~\ref{th2}.  
The strategy of the proof is simple to outline: 
We will approximate the optimal transport map 
$T\colon\Omega_0\to\Omega_1$ for the Monge-Kantorovich distance,
up to a null set,
by an `ellipsoidal transport spray' built from a countable collection
of ellipsoidal Wasserstein droplets.
The spray maps
$\Omega_0$ to a target $\Omega_1^\eps$ whose shape distance
from $\Omega_1$ is as small as desired. 
Then from the corresponding ellipsoidal Euler droplets nested inside
the Wasserstein ones, 
we construct the desired Euler spray $(Q,\phi,p)$ that connects $\Omega_0$ to $\Omega_1^\eps$ by a critical path for the action in \eqref{e:act1}.

\begin{remark} 
In general, for the Euler sprays that we construct, the domain $\Qdom=\sqcup_{n=1}^\infty Q_n$
has an irregular boundary $\D\Qdom$ strictly larger than the 
infinite union $\sqcup_{n=1}^\infty\D Q_n$ of smooth boundaries of individual 
ellipsoidal Euler droplets, since $\D Q$ contains limit points of sequences
belonging to infinitely many $Q_n$. 
\end{remark}

\subsection{ Approximating optimal transport by an
ellipsoidal transport spray}

Heuristically, an ellipsoidal transport spray is a countable disjoint superposition
of  transport maps on ellipsoids, whose particle trajectories do not intersect.

\begin{definition}\label{d:tspray}
An \textbf{ellipsoidal transport spray} on $\Omega_0$ is a map
$S\colon\Omega_0\to \R^d$, such that
\[
\Omega_0 = \bigsqcup_{n\in\N} \Omega_0\supn 
\]
is a disjoint union of ellipsoids, the restriction of $S$ to $\Omega_0\supn$
is an ellipsoidal Wasserstein droplet, and the linear interpolants $S_t$ defined by 
\[
S_t(z) = (1-t)z+t S(z), \qquad z\in\Omega_0,
\]
remain injections for each $t\in[0,1]$. 
\end{definition}

\begin{proposition}\label{p:ETspray}
Let $\Omega_0$, $\Omega_1$ be a pair of bounded open sets in $\R^d$ of equal volume,
and let $T\colon\Omega_0\to\Omega_1$ be the optimal transport map
for the Monge-Kantorovich distance with quadratic cost.
For any $\eps>0$, there is an ellipsoidal transport spray 
$S^\eps \colon\Omega_0^\eps\to \R^d$ such that 
\begin{itemize}
\item[(i)] $\Omega_0^\eps$ is a countable union of balls 
with $|\Omega_0\setminus\Omega_0^\eps|=0$, 
\item[(ii)] $\displaystyle
\sup_{z\in \Omega_0^\eps} |T(z)-S^\eps(z)| <  \eps \diam\Omega_1\,,
$
and
\item[(iii)] the $L^\infty$ transportation distance  between the uniform
distributions on $\Omega_1^\eps$ and $\Omega_1$ satisfies 
$d_\infty(\Omega_1^\eps,\Omega_1)< \eps \diam\Omega_1$.
\end{itemize}
\end{proposition}

The proof of this result will comprise the remainder of this subsection.
The strategy is as follows. 
Due to Alexandrov's theorem on the twice differentiability of convex functions,
the Brenier map $T=\nabla\psi$ is differentiable a.e. 
The set $\Omega_0^\eps$ will be chosen to be the union of a suitable Vitali covering 
of $\Omega_0$ a.e.\ by balls $B_i$, whose centers are points 
of differentiability of $T$.
On each ball $B_i$ we approximate $T$ by an 
affine map $\tT$ which takes the ball center $x_i$ to $(1+\eps)T(x_i)$,
taking the form
\begin{equation} \label{def:tildeT}
{ 
\tT(x) = (1+\eps) T(x_i) + DT(x_i)(x - x_i)\,,
}\qquad x\in B_i\,.
\end{equation}
The corresponding displacement interpolation map $\tT_t$
has three key properties: (i) it is locally affine
so maps balls to ellipsoids, (ii) it is volume-preserving, and 
(iii) spreading out the ball centers by the dilation factor $1+\eps$ 
ensures that the ball images remain non-overlapping,
because they are nested inside corresponding images under 
a dilated version of the displacement interpolation map $T_t$.


%

\subsubsection{Nesting by subgradient approximation}

%

It turns out to be quite convenient to construct
this dilated version based on the \emph{subgradient} $\D\psi$
of the Brenier potential $\psi$, 
to deal with the problem that 
the Brenier map $T$ may be discontinuous,
perhaps on a complicated set.

We recall that the subgradient of $\psi$ is a set-valued function defined by 
\begin{equation}
\D\psi(x) = \{z\in\R^d: \psi(x+h)\ge \psi(x)+\ip{z,h} \ \ \forall h\in\R^d\}.
\end{equation}
where $\ip{\cdot,\cdot}$ denotes the standard inner product on $\R^d$.
For each $x\in\R^d$, the set $\D\psi(x)$ is closed, convex, and nonempty.
In appendix~\ref{a:subg} we provide proofs of the few basic facts about
subgradients that we will use. 

According to Alexandrov's theorem (see \cite[Thm.~1.3]{Mignot} or \cite{EvansGariepy}),
for almost every $x_0\in\R^d$ 
the subgradient $\D\psi$ admits a local first-order expansion
\begin{equation}\label{e:alex1}
\D\psi(x)\subset T(x_0) + H (x-x_0) + B(0,\omega(x_0,r)) 
\qquad \forall x\in B(x_0,r)\,,
\end{equation}
where $H$ is a positive semidefinite matrix and 
$\omega(x_0,r)=o(r)$ as $r\to0$. 
Note we may assume $\omega(x_0,r)/r$ is decreasing in $r$.
The quantity $T(x_0)=\nabla\psi(x_0)$
provides the Brenier transport map at $x_0$, and 
we let $\Hess \psi(x_0)$ denote the matrix $H$.

Let us say $x_0$ is an \emph{Alexandrov point}
if \eqref{e:alex1} holds. 
Because $T=\nabla \psi$ pushes forward the
Lebesgue measure on $\Omega_0$ to that on $\Omega_1$, 
it follows that $\det\Hess\psi(x)=1$ for a.e. 
Alexandrov point $x$ in $\Omega_0$ 
(see~\cite[Thm.~4.4]{McCann97} or \cite[Thm.~4.8]{Villani03}).
Denoting by $\Omega_A$ the set of these points, we have 
$|\Omega_0\setminus\Omega_A|=0$, and 
\begin{equation}\label{e:lamprod}
\lambda_1(x) \cdots \lambda_d(x) =1 \qquad\mbox{for all $x\in\Omega_A$,}
\end{equation} 
where $\lambda_1(x), \dots, \lambda_d(x)$ denote
the eigenvalues of $\Hess \psi(x)$.

Our construction involves an expanded, subgradient extension
of the displacement interpolating map $T_t$. Namely, given $\eps>0$
and $t\in(0,1)$, we define
\begin{equation}
 \psi^\eps_t(x) = \frac12(1-t)|x|^2+t(1+\eps)\psi(x).
\end{equation}
The subgradient of this function is (see Prop.~\ref{p:subg}.ii),
\begin{equation}\label{e:psit}
 \D\psi^\eps_t(x) = (1-t)x+t(1+\eps)\D\psi(x).
\end{equation}
In case $\eps=0$, this map extends $T_t$ in the sense that 
$\D\psi_t^0(x)=\{T_t(x)\}$ for all $x\in\Omega_A$.
Further, the range of this subgradient is all of $\R^d$
(by Prop.~\ref{p:subg}.iii).
Just as in \eqref{e:Ttbound}, due to the monotonicity of the subgradient
(Prop.~\ref{p:subg}.i)
one has 
\begin{equation}\label{e:psitinv}
|z-\hat z| \ge (1-t)|x - \hat x| \quad\mbox{whenever}\quad
z\in \D\psi^\eps_t(x), \  
\hat z\in \D\psi^\eps_t(\hat x).  
\end{equation}
By consequence, the inverse $L^\eps_t := (\D\psi^\eps_t)\inv$ is a single-value
Lipschitz map with Lipschitz constant bounded by $(1-t)\inv$.

\begin{lemma}\label{p:alex}
Let $\eps>0$, 
and let $x_0\in\Omega_A$. 
Choose $r_0>0$ such that 
\begin{equation}\label{c:r0}
(1+\eps)\omega(x_0,r_0)< \frac12\eps\ulam(x_0) r_0\,.
\end{equation}
where $\ulam(x_0)\in(0,1)$ is the smallest eigenvalue of $H=\Hess\psi(x_0)$.
Then whenever $0<r<r_0$ and 
$0<t<1$, the ellipsoid defined by 
\[
E^\eps_t(x_0,r) 
=\big\{\, (1-t)x+ t(1+\eps)T(x_0) + tH(x-x_0)\,:\, x\in B(x_0,r)\big\}
\]
satisfies 
\begin{equation}\label{e:EBt}
E^\eps_t(x_0,r) \subset \D\psi^\eps_t( B(x_0,r)).
\end{equation}
\end{lemma}
We note that the fact that the term $tH(x-x_0)$ 
does not contain a factor $1+\eps$ is needed to guarantee the inclusion
in \eqref{e:EBt}.
\begin{proof}
Writing
\[
x^\eps_t = \nabla\psi^\eps_t(x_0) = (1-t)x_0 + t(1+\eps)T(x_0),
\qquad H_t = (1-t)I + tH,
\]
we have $E^\eps_t(x_0,r) = x^\eps_t+H_t B(0,r)$.
For all $x\in\R^d$ define 
\begin{equation}
f^\eps_t(x) = L^\eps_t(x^\eps_t+H_t x) -L^\eps_t(x^\eps_t)
\end{equation}
and note that $L^\eps_t(x^\eps_t)=x_0$.
Then $z=x^\eps_t+H_tx\in\D\psi^\eps_t(y)$ where 
$y=x_0+f^\eps_t(x)$.

We claim that $f^\eps_t(x)\in B(0,r)$ whenever 
$0<r\le r_0$, $t\in(0,1)$, and $x\in B(0,r)$.
The proof shall be by a continuation argument in $r$ using the fact
that $L^\eps_t$ is Lipschitz with Lipschitz constant $(1-t)\inv$.
The continuation is based on the following.
\begin{sublemma}\label{l:fepst}
For each $t\in(0,1)$ there exists $\theta_t<1$ such that if
$|x|\le r_0$ and we assume $|f^\eps_t(x)|\le r_0$, then
\[
|f^\eps_t(x)|\le \theta_t|x|.
\]
\end{sublemma}
\begin{proof}
Under the stated assumption, we have $x^\eps_t+H_t x\in\D\psi^\eps_t(x_0+y)$
where $y=f^\eps_t(x)$. Due to \eqref{e:psit} and \eqref{e:alex1}, 
there exists $w\in B(0,\omega(x_0,|x|))$ such that 
\[
x^\eps_t+H_tx = (1-t)(x_0+y)+t(1+\eps)(T(x_0)+H y + w),
\]
whence
\[
y = (H_t+ \eps t H)\inv(H_t x-t(1+\eps)w).
\]
By diagonalizing $H$ and noting 
$\lambda_t = 1-t + t\ulam$ is the smallest eigenvalue of $H_t$, 
one finds
\[
|(H_t+ \eps t H)\inv 
H_t x| \le \frac{\lambda_t}{\lambda_t+\eps t \ulam} |x|,
\qquad
|(H_t+ \eps t H)\inv w| \le \frac{\omega(x_0,|x|)}{\lambda_t+\eps t\ulam}.
\]
Using the fact that
$\omega(x_0,|x|)\le \frac12\eps\ulam |x|$, the result of the 
sublemma follows by taking
\[
\theta_t = \frac{\lambda_t+\frac12\eps t\ulam}{\lambda_t+\eps t\ulam}<1.
\]
\end{proof}
Now we finish the proof of Lemma~\ref{p:alex}. Fix $t\in(0,1)$ and let
\begin{equation}\label{d:rt}
r_t = \sup\big\{r\in[0,r_0] \,:\, \mbox{ 
$|x|\le r$ \quad implies \quad $|f^\eps_t(x)|\le\theta_t|x|$} \,\}
\end{equation}
(without making the extra assumption in the sublemma). 
The set in \eqref{d:rt} is closed and $r_t>0$,
because $f^\eps_t$ is continuous and $f^\eps_t(0)=0$.
Note that $|x|\le r_t$ implies 
$|f^\eps_t(x)|\le \theta_t r_t<r_t$.
Now it follows $r_t=r_0$, because if $r_t<r_0$, 
then it follows from continuity 
that for some $r\in(r_t,r_0]$, $|x|\le r$ implies 
$|f^\eps_t(x)|\le r_t<r_0$,
whence $|f^\eps_t(x)|\le\theta_t|x|$ by the sublemma, 
contradicting the definition of $r_t$.
\end{proof}

\subsubsection{Proof of Proposition~\ref{p:ETspray}}
We suppose $0<\eps<1$. The first step in the proof is to produce a suitable Vitali covering 
of $\Omega_0$, up to a null set,  by a countable disjoint union of balls. 
By a simple translation of  target and source 
so that the origin is the midpoint of two points in $\bar\Omega_1$ 
separated by distance $\diam\Omega_1$, because the distance
from any point in $\Omega_1$ to each of the two points is also no more than
$\diam\Omega_1$ we may assume that
\begin{equation}\label{e:center}
\sup_{x\in \Omega_0}  |T(x)|
<\diam \Omega_1\,. 
\end{equation}

We may choose $\oor(x,\eps)>0$ for each $x\in\Omega_A$ and $\eps>0$
such that whenever $0<r<\oor(x,\eps)$ we have (see \eqref{c:r0})
\begin{equation}\label{c:oor}
(1+\eps)\omega(x,r)<\frac12\eps\ulam(x)r
,\qquad   \frac{\olam(x)}{\ulam(x)} r < \eps\diam\Omega_1,
\end{equation}
where $\olam(x)$ is the largest eigenvalue of $\Hess(x)$.
(The second condition on $r$ will be used in the next subsection.)
Then $|\Omega_0\setminus\Omega_A|=0$, and 
the family of balls 
\[
\{ B(x, r) : x \in \Omega_A ,\ 0<r < \oor(x,\eps) \}
\]
forms a Vitali cover of $\Omega_A$. Therefore, by Vitali's covering theorem \cite[Theorem III.12.3]{DSbook},
there is a countable family of mutually disjoint balls $B(x_i,r_i)$,  with $x_i \in \Omega_A$ 
and $0< r_i < \oor(x_i,\eps)$, such that
\[
| \Omega_A \backslash \cup_{i \in \N} B(x_i,r_i) | = 0\ .
\] 
We let
\begin{equation}\label{d:Om0eps}
\Omega_0^\eps = \bigsqcup_{i \in \N} B_i \,,
\qquad B_i=B(x_i,r_i).
\end{equation}

Define the map $\tT$ by \eqref{def:tildeT}.
To show $\tT$ is an ellipsoidal transport spray on $\Omega_0^\eps$,
we first prove that the linear interpolants defined by 
\[
\tT_t(z) = (1-t)z+t \tT(z), \qquad z\in\Omega_0^\eps,
\]
remain injections for each $t\in[0,1)$. Clearly the restriction
to each $B_i$ is an injection. 
But by invoking Lemma~\ref{p:alex} with $H=\Hess\psi(x_i)$, 
we conclude that the image of $B_i$ under $\tT_t$ satisfies
\begin{equation}
 \tT_t(B_i) = E_t^\eps(x_i,r_i) \subset \D\psi_t^\eps(B_i)\,.
\end{equation}
Recalling that the inverse $(\D\psi^\eps_t)\inv$ is a single-value
Lipschitz map by \eqref{e:psitinv}, this implies that 
the images $\tT_t(B_i)$ are pairwise disjoint.

Now, for $t=1$ we necessarily have $\tT$ is injective, for 
if not then for some $i\ne j$,
the open set $\tT_t(B_i)\cap\tT_t(B_j)$ is nonempty for $t=1$
and hence for $t$ near 1, contradiction.
This proves that $\tT$ is an ellipsoidal transport spray on 
the set $\Omega_0^\eps$ in \eqref{d:Om0eps}, so that property (i) holds.

Next we prove property (ii). Using \eqref{c:oor},
for each $x\in B_i$ we have, since $T(x)\in\D\psi(x)$,
\begin{align}
|T(x)-S^\eps(x)| &\le |T(x)-T(x_i)-DT(x_i)(x-x_i)|+\eps|T(x_i)| 
\nonumber \\& \le 
\omega(x_i,r_i) +\eps|T(x_i)|
\nonumber \\& \le 
\frac12 \eps\ulam(x_i)  r_i +\eps\diam\Omega_1 
\le \frac32\eps\diam\Omega_1.
\label{e:taylor}
\end{align}
This shows (ii) after replacing $\eps$ by $\eps/2$. 
For part (iii), we note that the set 
$\Omega_0^\eps = (S^\eps)\inv(\Omega_1^\eps)$ 
has full measure in $\Omega_0$, and the map $T$
pushes forward Lebesgue measure on $\Omega_0$
to Lebesgue measure on $\Omega_1$.
It follows that the map
$T\circ(S^\eps)\inv\colon\Omega_1^\eps\to\Omega_1$ 
pushes forward uniform measure to uniform measure
and satisfies
\[
\sup_{x\in \Omega_1^\eps} |T\circ (S^\eps)\inv(x)-x| < 
\eps \diam\Omega_1.
\]
The result claimed in part (iii) follows, due to \eqref{d:dinfty}.
This finishes the proof of Proposition~\ref{p:ETspray}.

\subsection{Action estimate for Euler spray}

Each of the ellipsoidal Wasserstein droplets that make up the 
ellipsoidal transport spray $S^\eps$ is associated with a boosted ellipsoidal
Euler droplet nested inside, due to the nesting property in Proposition~\ref{p:nest}. 
The disjoint superposition of these Euler droplets
make up an Euler spray that deforms $\Omega_0^\eps$ to the same 
set $\Omega_1^\eps$. 

In order to complete the proof of Theorem~\ref{th2}, it remains to 
bound the action of this Euler spray in terms of the Wasserstein
distance between the uniform measures on $\Omega_0$ and $\Omega_1$.
Toward this goal, we first note that 
because the maps $T$ and $S^\eps$ are volume-preserving,
due to the estimate in part (ii) of Proposition~\ref{p:ETspray}
and \eqref{d:dinfty} we have
\[
d_W(T(B_i),S^\eps(B_i))^2 \le 
\left( \eps K_1\right)^2 |B_i| \,,
\qquad K_1=\diam\Omega_1.
\]
Now by the triangle inequality, 
\begin{align}
 d_W(B_i,S^\eps(B_i))^2  
&\le \left(d_W(B_i,T(B_i)) + \eps K_1|B_i|^{1/2}\right)^2
\nonumber\\&
\le d_W(B_i,T(B_i))^2 (1+\eps) + (\eps+\eps^2) K_1^2 |B_i|
\end{align}

Recall that by inequality \eqref{e:Aa-action} of Lemma~\ref{l:Eactionest},
the action  of the $i$-th ellipsoidal Euler droplet, 
denoted by ${\mathcal A}_i$, satisfies
\begin{align}
{\mathcal A}_i &\le d_W(B_i,S^\eps(B_i))^2  
+ \frac{\olam(x_i)^4}{\ulam(x_i)^2}r_i^{2}|B_i| 
\nonumber \\&
\le d_W(B_i,T(B_i))^2(1+\eps) +3\eps K_1^2|B_i| \,,
\end{align}
where we make use of the second constraint in \eqref{c:oor}.

By summing over all $i$, we obtain the required bound,
\[
{\mathcal A}^\eps =\sum_i {\mathcal A}_i \le d_W(\one_{\Omega_0},\one_{\Omega_1})^2
+ K\eps
\]
where
\[
K = d_W(\one_{\Omega_0},\one_{\Omega_1})^2
+ 4 |\Omega_0|(\diam\Omega_1)^2.
\]
This concludes the proof of Theorem~\ref{th2}.

\section{Shape distance equals Wasserstein distance}\label{s:dsdw}

Our main goal in this section 
is to prove Theorem~\ref{t:complete}, which establishes the existence
of paths of shape densities (as countable concatenations of Euler sprays)
that exactly connect any two compactly supported measures having densities 
with values in $[0,1]$ and have action as close as desired to the 
Wasserstein distance squared between the measures.
Theorem~\ref{th1} follows as an immediate corollary, showing 
that shape distance between 
arbitrary bounded measurable sets with positive, equal volume
is the Wasserstein distance between the corresponding characteristic functions.

Theorem~\ref{t:complete} will be deduced from Theorem~\ref{th2} by essentially `soft' 
arguments. Theorem~\ref{th2} shows that the relaxation 
 of shape distance, in the sense of lower-semicontinuous envelope with respect
 to the topology of weak-$\star$ convergence of characteristic functions,
 is Wasserstein distance. 
Essentially, here we use this result to compute the completion of the shape distance
in the space of bounded measurable sets.

\begin{lemma}\label{l:approxrho} 
Let $\rho\colon\R^d\to [0,1]$ be a measurable function of compact support.
Then for any $\eps>0$
there is an open set $\Omega$ such that its volume is the total mass of $\rho$
and the $L^\infty$ transport distance from $\rho$ to its characteristic function 
is less than $\eps$: 
\[
|\Omega| = \int_{\R^d} \rho \, dx \quad\mbox{and}\quad
d_\infty(\rho,\one_{\Omega})<\eps\,.
\]
\end{lemma}
\begin{proof}  We recall that weak-$\star$ convergence of probability measures  supported
in a fixed compact set is equivalent to convergence in (either $L^2$ or $L^\infty$) 
Wasserstein distance. Given $k\in\N$,
cover the support of $\rho$ a.e.\ by a grid of disjoint open rectangles
of diameter less than $\eps_k=1/k$. For each rectangle $R$ in the grid,
shrink the rectangle homothetically from any point inside to obtain a sub-rectangle
$\hat R\subset R$ such that $|\hat R| = \int_R \rho\,dx.$ Let $\Omega_k$ be the 
disjoint union of the non-empty rectangles $\hat R$ so obtained.
Then the sequence of characteristic functions $\one_{\Omega_k}$ evidently
converges weak-$\star$ to $\rho$ in the space of fixed-mass measures on a fixed 
compact set:  for any continuous test function $f$ on $\R^d$,
as $k\to\infty$ we have
\[
\int_{\Omega_k} f(x)\,dx \to \int_{\R^d} f(x)\rho(x)\,dx\, .
\]
Choosing $\Omega=\Omega_k$ 
for some sufficiently large $k$ yields the desired result.
\end{proof}

\begin{proof}[Proof of Theorem~\ref{t:complete} part (a)]
 Let $\rho_0$, $\rho_1$ have the properties stated, and
suppose $D:=d_W(\rho_0,\rho_1)>0$. (The other case is trivial.)
Let $\eps>0$.
By Lemma~\ref{l:approxrho} we may choose 
open sets $\Omega_0$ and $\hat\Omega_1$ whose volume
is $\int_{\R^d}\rho_0$ and such that
\begin{equation}\label{e5:ax1}
d_\infty(\rho_0,\one_{\Omega_0}) + d_\infty(\rho_1,\one_{\hat\Omega_1})<\frac\eps2,
\qquad
d_W(\Omega_0,\hat\Omega_1)^2 \le d_W(\rho_0,\rho_1)^2 + \frac\eps2.
\end{equation}
Then we can apply Theorem~\ref{th2} to find an Euler spray that connects 
$\Omega_0$ to a set $\hat\Omega_1^\eps=:\Omega_1$ close to $\hat\Omega_1$ with the properties
\begin{equation}\label{e5:ax2}
d_\infty(\Omega_1,\hat\Omega_1) <\frac\eps3, \qquad
\calA^\eps \le  d_W(\Omega_0,\hat\Omega_1)^2 + \frac\eps3,
\end{equation}
where $\calA^\eps$ is the action of this Euler spray.
By combining the inequalities in \eqref{e5:ax1} and \eqref{e5:ax2} 
we find that the sets $\Omega_0$, $\Omega_1$ have the properties required. 
\end{proof}

Before we establish part (b), we separately discuss the concatenation
of transport paths.  Let $\rho^k=(\rho^k_t)_{t\in[0,1]}$ be a path 
of shape densities for each $k=1,2,\dots,K$, with associated
transport velocity field $v^k\in L^2(\rho^k\,dx\,dt)$ and action 
\[
\calA_k = \int_0^1\int_{\R^d}\rho^k_t(x) |v^k(x,t)|^2\,dx\,dt.
\]
We say this set of paths forms a \textit{chain} if  $\rho^k_1=\rho^{k+1}_0$
for $k=1,\ldots,K-1$.  Given such a chain, and numbers $\tau_k>0$ 
such that $\sum_{k=1}^K \tau_k=1$, we define the \textit{concatenation of 
the chain of paths $\rho^k$ compressed by $\tau_k$} to be 
the path $\rho = (\rho_t)_{t\in[0,1]}$  given by 
\begin{equation}\label{d:concat}
\rho_t = \rho^k_s \quad\mbox{ for } \ \ t = \tau_k s+ \sum_{j<k}\tau_k\,, \quad s\in[0,1].
\end{equation}
The transport velocity associated with the concatenation is 
\begin{equation}\label{d:concatv}
v(\cdot,t) = \tau_k\inv v^k(\cdot,s) \quad\mbox{ for } \ \ t = \tau_k s+ \sum_{j<k}\tau_k\,, \quad s\in[0,1],
\end{equation}
and the action is
\begin{equation}\label{d:concatA}
\calA = \int_0^1\int_{\R^d}\rho_t|v|^2\,dx\,dt = 
\sum_{k=1}^K \tau_k\inv \int_0^1\int_{\R^d} \rho^k_s(x)|v(x,s)|^2\,dx\,ds 
= \sum_{k=1}^K \tau_k\inv \calA_k.
\end{equation}

\begin{remark} We mention here how the triangle inequality for the shape distance
defined in \eqref{d:ds} follows directly from this concatenation procedure.
Given the chain $\rho^k$ as above with actions $\calA_k$,
let $\delta_k = \sqrt{\calA_k}$ and  set
\[
\tau_k = \frac{\delta_k}{\sum_j \delta_j}, \quad k=1,\ldots,K.
\]
Let $\calA$ be the action of the 
concatenation of paths $\rho^k$ compressed by $\tau_k$,
and let $\delta=\sqrt{\calA}$. Then
\[
\calA = \delta^2 = \sum_k \tau_k\inv \delta_k^2 = \left( \sum_k \delta_k\right)^2\,.
\]
From this the triangle inequality follows.
\end{remark}

\begin{proof}[Proof of Theorem~\ref{t:complete} part (b)]
Next we establish part (b). The idea is to construct a path of shape densities
$\rho=(\rho_t)_{t\in[0,1]}$ connecting $\rho_0$ to $\rho_1$
by concatenating the Euler spray from part (a) together 
with two paths of small action that themselves are concatenated chains of Euler sprays 
that respectively 
connect ${\Omega_0}$ to sets that approximate $\rho_0$, 
and connect $\Omega_1$ to sets that approximate $\rho_1$.

Let $\eps>0$, and let $\rho^\eps$ be a shape density  
determined by an Euler spray as from part (a) that connects 
bounded open sets $\Omega_0$ and $\Omega_1$ of volume $\int_{\R^d} \rho_0$, 
but with the (perhaps tighter) conditions
\[
d_W(\one_{\Omega_0},\rho_0) + d_W(\one_{\Omega_1},\rho_1)<\frac14 \eps 2\inv,
\qquad
\calA^\eps < d_W(\rho_0,\rho_1)^2 + \eps\,,
\]
where $\calA^\eps$ is the action of this spray.

Next we construct a chain of Euler sprays with shape densities $\rho^k$,  
$k=1,2,\ldots$, with action $\calA_k$ that connect $\Omega_1$ with a chain of sets $\Omega_k$ such
that $\one_{\Omega_k}\wkto \rho_1$ as $k\to\infty$ and 
\begin{equation}\label{e:pathkgoals}
d_W(\one_{\Omega_k},\rho_1)<\frac14\eps 2^{-k}, \qquad 
\calA_k < (\eps 2^{-k})^2. 
\end{equation}
We proceed by recursion by applying Theorem~\ref{th2} like in the proof of part (a).
Given $k\ge1$, suppose $\Omega_k$ is defined and $\rho^j$ are defined for $j<k$.
Using Lemma~\ref{l:approxrho} we can find a bounded open set 
$\hat\Omega_{k+1}$  such that
\[
|\hat\Omega_{k+1}| = \int_{\R^d}\rho_0 
\qquad\mbox{and}\qquad
d_W(\one_{\hat\Omega_{k+1}},\rho_1)< \frac18\eps 2^{-k-1} \,.
\]
Then by invoking Theorem~\ref{th2} and the triangle inequality for $d_W$,
we obtain an Euler spray with action $\calA_k$ that connects
$\Omega_k$ to a bounded open set $\Omega_{k+1}$, such that 
\[
d_W(\Omega_{k+1},\hat\Omega_{k+1})< \frac18\eps 2^{-k-1}
\quad\mbox{and}\quad
\calA_k < d_W(\Omega_k,\hat\Omega_{k+1})^2 + \frac12 (\eps 2^{-k})^2
< (\eps 2^{-k})^2 \ .
\]
We let $\rho^k=(\rho^k_t)_{t\in[0,1]}$ be the path of shape densities for this spray, so that
$\rho^k_0=\one_{\Omega_k}$ and $\rho^k_1=\one_{\Omega_{k+1}}$. 
This completes the construction of the chain of paths $\rho^k$ satisfying
\eqref{e:pathkgoals}.

It is straightforward to see that $d_W(\rho^k_t,\rho_1)\to0$ as $k\to \infty$
uniformly for $t\in[0,1]$. Now we let $\rho^+=(\rho^+_t)_{t\in[0,1]}$
be the countable concatenation of this chain of paths $\rho^k$ compressed by
$\tau_k=2^{-k}$ according to the formulas \eqref{d:concat}--\eqref{d:concatA} above
taken with $K\to\infty$, and with $\rho^+_1=\rho_1$.   
The action $\calA^+$ of this concatenation then satisfies
\begin{equation}\label{e:act+}
\calA^+ = \sum_{k=1}^\infty 2^k \calA_k < \eps^2. 
\end{equation}

In exactly analogous fashion we can construct a countable concatenation $\hat\rho^-$
of a chain of paths coming from Euler sprays, that connects 
$\hat\rho^-_0=\one_{\Omega_0}$ with $\hat\rho^-_1 = \rho_0$ 
and having action $\calA^-<\eps^2$.   Then define $\rho^-$ be the
\textit{reversal} of $\hat\rho^-$, given by 
\[
\rho^-_t = \hat\rho^-_{1-t} \,.
\]
This path $\rho^-$ has the same action $\calA^-$.

Finally,  define the path $\rho$ by concatenating $\rho^-$, $\rho^\eps$, $\rho^+$
compressed by $\eps$, $1-2\eps$ and $\eps$ respectively.
This path satisfies the desired endpoint conditions and has action $\calA$ that
satisfies
\[
\calA = \eps\inv \calA^- + (1-2\eps)\inv \calA^\eps+ \eps\inv\calA^+
< d_W(\rho_0,\rho_1)^2 + K\eps,
\]
for some constant $K$ independent of $\eps$ small. The result 
of part (b) follows.
\end{proof}

\begin{remark}
Our construction here of a sequence of action-infimizing paths involves
connecting geodesics given by Euler sprays only for simplicity.
A more general approach to constructing non-geodesic near-optimal 
incompressible paths can be taken that exploits the convexity of density 
along Wasserstein transport paths.
Such an approach was implemented in an earlier preprint version of this 
article \cite{DLPv1}.
\end{remark}

\subsection{Rigidity of minimizing incompressible paths}
\label{ss:rigid}

The result of Theorem~\ref{th4}, providing a sharp criterion for the existence
of a minimizer for the shape distance in \eqref{d:ds},  follows by combining
the uniqueness property of Wasserstein geodesics with the result of
Theorem~\ref{th1}.
  
\begin{proof}[Proof of Theorem~\ref{th4}] 
Let $\rho=(\rho_t)_{t\in[0,1]}$ be the density along the 
Wasserstein geodesic path that connects $\one_{\Omega_0}$ and $\one_{\Omega_1}$,
where $\Omega_0$, $\Omega_1$ are bounded open sets in $\R^d$ with equal volume. 
Clearly, if $\rho$ 
is a characteristic function, then the Wasserstein geodesic provides a
minimizing path for shape distance.  On the other hand, if a minimizer for
\eqref{d:ds} exists, it must have constant speed by a standard
reparametrization argument.  Then by Theorem~\ref{th1} it provides a
constant-speed minimizing path for Wasserstein distance as well, hence
corresponds to the unique Wasserstein geodesic.  Thus the Wasserstein geodesic
density $\rho$ is a characteristic function.  \end{proof}

A consequence of Theorem~\ref{th1} is that 
existence of a minimizer among incompressible transport paths in \eqref{d:ds}
imposes a rigid structure on the optimal transport map $T$.
To discuss this it is convenient to invoke the 
regularity theory of Caffarelli \cite{caff91},  
Figalli \cite{Fig10} and Figalli \& Kim~\cite{FK2010}.
These authors have shown 
(see Theorem 3.4 in \cite{figall_dephil_BAMS} and also \cite{dPF2015}) 
that, due to the fact that the characteristic
functions are smooth inside $\Omega_0$ and $\Omega_1$,
the  optimal transportation potential $\psi$ is smooth 
away from a set of measure zero. 
More precisely, there exist relatively closed sets of measure zero, 
$\Sigma_i \subset \Omega_i$ for $i=0,1$ such that 
$T: \Omega_0 \backslash \Sigma_0 \to  \Omega_1 \backslash \Sigma_1$ 
is a smooth diffeomorphism between two open sets.

\begin{remark} In a previous draft of this paper, this regularity theory 
was used to prove Theorem~\ref{th2} through a Vitali covering argument.
The present approach to the proof in section~\ref{s:spray}
exploits the simpler property that the subgradient maps
$\D\psi_t$ have single-valued inverses.
\end{remark}

Along the particle paths of displacement interpolation starting
from any  $z\in\Omega_0\setminus\Sigma_0$, the mass density satisfies
\begin{equation}
\rho(T_t(z),t)^{-1} = \det \frac{\D T_t}{\D z} = \det ( (1-t)I+ t\Hess\psi(z) )
= \prod_{j=1}^d (1-t+t\lambda_j(z)) \,.
\label{rhoinv}\end{equation}

We now study the rigidity of incompressible transport paths by 
examining a simple proof that the density $\rho$ is log-convex 
along these paths. A related fact stated in \cite[Lemma 2.1]{McCann97} 
implies the (stronger) property that $\rho^{-1/d}$ is concave 
along particle paths and is connected to a well-known proof of 
the Brunn-Minkowski inequality by Hadwiger and Ohmann.

%
\begin{lemma} \label{lem:dens}
Along the particle paths $t\mapsto T_t(z)$ of displacement interpolation
between the measures $\mu_0$ and $\mu_1$ with respective densities
$\one_{\Omega_0}$ and $\one_{\Omega_1}$ as above, 
the density is log-convex, that is $t \mapsto \log\rho(T_t(z),t)$ is convex,
for $z\in\Omega_0\setminus\Sigma_0$.  
Moreover, this function is constant if and only if $\Hess\psi(z)=I$.

\end{lemma}
\begin{proof}
We compute 
\begin{equation}\label{eq:d2logrho}
\frac{d^2}{dt^2} \log \rho = - \frac{d}{dt} \sum_{j=1}^d \frac{\lambda_j-1}{1-t+t\lambda_j} = 
\sum_{j=1}^d \left(\frac{\lambda_j-1}{1-t+t\lambda_j} \right)^2 \ge0,
\end{equation}
and this vanishes if and only if $\lambda_j=1$ for all $j$.
\end{proof}

\begin{corollary}\label{cor:rigid}
The Wasserstein geodesic density $\rho$ is a characteristic function
if and only if the displacement interpolant is piecewise a rigid translation:
\[
T_t(z)= z+ tb(z) \,,
\]
where $b(\cdot)$ is constant on each component of the open set
$\Omega_0\setminus\Sigma_0$.
\end{corollary}

%
%
In case the result of the Corollary applies,
$\Omega_1=T(\Omega_0)$ represents some kind
of decomposition of $\Omega_0$ by fracturing into pieces that can separate
without overlapping. 


\begin{remark}\label{r:rho3} 
In the case of one dimension ($d=1$)
it is always the case that the Wasserstein geodesic density
$\rho(T_t(z),t)\equiv 1$ for all $z$ in the non-singular set.
This is so because the diffeomorphism
$T: \Omega_0 \backslash \Sigma_0 \to  \Omega_1 \backslash \Sigma_1$ 
must always be a rigid translation on each component,
as it pushes forward Lebesgue measure to Lebesgue measure.

As a nontrivial example, 
let $\mathcal C\subset [0,1]$ be the standard Cantor set, 
and let $\Omega_0=(0,1)$.
Define the Brenier map $T(x)=x+c(x)$ with $c$ given by the Cantor function, 
increasing and continuous on $[0,1]$ with $c(0)=0$, $c(1)=1$ and 
$c'=0$ on $(0,1)\setminus\mathcal C$. 
Then $T(\Omega_0) = (0,2)$, but the pushforward of uniform measure on 
$\Omega_0$ is the uniform measure on the set
$\Omega_1 = T(\Omega_0\setminus \mathcal C)$, 
which has countably many components, and total length $|\Omega_1|=1$.
\end{remark}

\begin{remark} The proof of Lemma~\ref{lem:dens} can be easily modified
to show $\rho^{-q}$ is concave along particle paths
for any $q\in(0,1/d]$.  E.g., for $g(t)=\rho(T_t(z),t)^{-1/d}$, one has 
\begin{equation}\label{eq:d2logg}
\frac{g''}g = \left(\frac1d\sum_{j=1}^d \frac{\lambda_j-1}{1-t+t\lambda_j} \right)^2 -
\frac1d \sum_{j=1}^d \left(\frac{\lambda_j-1}{1-t+t\lambda_j} \right)^2  \le 0\,,
\end{equation}
due to the Cauchy-Schwartz (or Jensen's) inequality. 
\end{remark}


\section{Displacement interpolants as weak limits}\label{s:weaklimits}

\subsection{Proof of Theorem~\ref{th3}.}
Next we supply the proof of Theorem~\ref{th3}. 
We will accomplish this in two steps, first dealing with the case
that the endpoint densities $\rho_0, \rho_1$ are characteristic functions
of bounded open sets. To extend this result to the general case,
we will make use of fundamental results on stability of optimal transport plans
from \cite{AGS} and \cite{Villani09}.

\begin{proposition}\label{p:wkconv} 
Let $\Omega_0$, $\Omega_1$ be bounded open sets of equal volume. 
Let $(\rho,v)$ be the density and transport velocity determined by 
the unique Wasserstein geodesic (displacement interpolant)
that connects the uniform measures on $\Omega_0$ and $\Omega_1$
as described in section~2.

Then, as $\eps\to0$, the weak solutions $(\rho^\eps,v^\eps,p^\eps)$
associated to the Euler sprays of Theorem~\ref{th2} 
converge to $(\rho,v,0)$, and $(\rho,v)$ is a weak solution 
to the pressureless Euler system
\eqref{e:euler1p}--\eqref{e:euler2p}.
The convergence holds in the following sense: $p^\eps\to0$ uniformly,
and 
\begin{equation}\label{e:rvwk6}
\rho^\eps\wksto \rho,\qquad
\rho^\eps v^\eps\wksto \rho v, \qquad
\rho^\eps v^\eps\otimes v^\eps \wksto \rho v\otimes v ,
\end{equation}
weak-$\star$ in $L^\infty$ on $\R^d\times[0,1]$.
\end{proposition}

As our first step toward proving this result, we describe
the bounds on pressure and velocity that come from the construction
of the Euler sprays constructed  above, for any given $\eps\in(0,1)$.
\begin{lemma} \label{lem:vlv}
Let $(Q^\eps,\phi^\eps,p^\eps)$, $0<\eps<1$, denote the Euler sprays
constructed in the proof of Theorem~\ref{th2}, and let $X^\eps\colon\Omega_0^\eps\times[0,1]\to\R^d$ denote the associated
flow maps, which satisfy
\[
\dot X^\eps(z,t) = \nabla\phi^\eps(X^\eps(z,t),t), \quad (z,t)\in\Omega_0^\eps\times[0,1],
\]
with $X^\eps(z,0)=z$. Then for some $\hat K>0$ independent
of $\eps$, we have 
\begin{equation}\label{e:pboundeps}
0\le p^\eps(x,t) \le  \hat K\eps
\end{equation}
{for all $(x,t)\in Q^\eps$}, and 
\begin{equation}\label{e:XTbound}
|X^\eps(z,t)- T_t(z)| +
|\dot X^\eps(z,t)-\dot T_t(z)| \le \hat K{\eps}
\end{equation}
{for all $(z,t)\in \Omega_0^\eps\times[0,1]$}, where $(z,t)\mapsto T_t(z)$
is the flow map from \eqref{e:Tt} for the Wasserstein geodesic.
\end{lemma}
\begin{proof}
By the pressure bound for individual droplets in \eqref{e:pbound}
together with the second condition in \eqref{c:oor}, we have
the pointwise bound
\begin{equation}
0\le p^\eps \le  K_0\eps\,,\qquad 
K_0 = K_1^2 d\,,
\quad K_1=\diam\Omega_1.
\end{equation}
Next consider the velocity.  
The boosted elliptical Euler droplet that transports $B_i$ to $S^\eps(B_i)$ 
is translated by $x_i$, and 
boosted by the vector
\begin{equation}
b_i := (1+\eps)T(x_i)-x_i = \dot T_t(x_i) + \eps T(x_i)\,.
\end{equation} 
In this ``$i$-th droplet,'' the velocity satisfies, by the estimate 
\eqref{e:vboundboost}, 
\begin{equation}
|\nabla\phi^\eps-b_i| =
|v^\eps-b_i|\le K_0\eps\,.
\end{equation}
Now the particle velocity for the Euler spray compares to that
of the Wasserstein geodesic according to 
\begin{align}
|\dot X^\eps(z,t)- \dot T_t(z)| &\le
|\dot X^\eps-b_i| + |b_i-\dot T_t(z)|
\nonumber\\&
\le K_0\eps + \eps|T(x_i)|+ |T(z)-z - (T(x_i)-x_i)|
\nonumber\\&
\le K_0\eps + \eps|T(x_i)|+ r_i \max_j |\lambda_j-1|+\omega(x_i,r_i)
\nonumber\\&
\le {K_0\eps}+3K_1\eps \,. 
\end{align}
(Here $\lambda_j$ denote the eigenvalues of $\Hess\psi(x_i)$,
and we use \eqref{e:alex1} with
the fact that $|\lambda_j-1|r_i\le\olam(x_i)r_i< K_1 \eps$
and $\omega(x_i,r_i)<\eps r_i$ by \eqref{c:oor}.)
Upon integration in time we obtain both bounds in \eqref{e:XTbound}.
\end{proof}

\begin{proof}[Proof of Proposition~\ref{p:wkconv}]
Now, let $(\rho,v)$ be the density and velocity of the particle
paths for the Wasserstein geodesic, from \eqref{rhoinv} and \eqref{e:Wv}.
To prove $\rho^\eps\wksto\rho$ weak-$\star$ in $L^\infty$,
it suffices to show that as $\eps\to0$,
\begin{equation}
\int_0^1\int_{\R^d} (\rho^\eps-\rho) q\,dx\,dt  \to 0\,,
\end{equation}
for every smooth test function $q\in C_c^\infty(\R^d\times[0,1],\R)$.
Changing to Lagrangian variables using $X^\eps$ for the term with 
$\rho^\eps=\one_{Q^\eps}$
 and $T_t$ for the term with $\rho$, the left-hand side becomes
\begin{equation}
\int_0^1\int_{\Omega_0} 
(q(X^\eps(z,t),t)-q(T_t(z),t))\,dz\,dt \,.
\end{equation}
Evidently this does approach zero as $\eps\to0$, due to \eqref{e:XTbound}.

Next, we claim $\rho^\eps v^\eps\wksto \rho v$ weak-$\star$ in $L^\infty$.
Because these quantities are uniformly bounded, it suffices to show that
as $\eps \to 0$,
\begin{equation}
\int_0^1\int_{\R^d} (\rho^\eps v^\eps-\rho v)\cdot\tilde v\,dx\,dt \to 0
\end{equation}
for each $\tilde v\in C_c^\infty(\R^d\times[0,1],\R^d)$.
Changing variables in the same way, the left-hand side becomes
\begin{equation}
\int_0^1\int_{\Omega_0} 
\left( \dot X^\eps(z,t)\cdot \tilde v(X^\eps(z,t),t) - 
\dot T_t(z)\cdot \tilde v(T_t(z),t) \right)\,dz\,dt\,.
\end{equation}
But because $\tilde v$ is smooth and due to the bounds in \eqref{e:XTbound},
this also tends to zero as $\eps\to0$.  

It remains to prove $\rho^\eps v^\eps\otimes v^\eps\wksto \rho v\otimes v$
weak-$\star$ in $L^\infty$. Considering the terms componentwise,
the proof is extremely similar to the previous steps.
This finishes the proof of Theorem~\ref{th3}.
\end{proof}

To generalize Proposition~\ref{p:wkconv} to handle general densities 
$\rho_0,\rho_1\colon\R^d\to[0,1]$,
we will use a double approximation argument, comparing Euler sprays
to optimal Wasserstein geodesics for open sets whose characteristic functions
approximate $\rho_0$, $\rho_1$ in the sense of Lemma~\ref{l:approxrho},
then comparing these
to the Wasserstein geodesic that connects $\rho_0$ to $\rho_1$. 
We prove weak-star convergence for the second comparison by extending the
results from \cite{AGS} and \cite{Villani09} on
weak-$\star$ stability of transport plans to establish weak-$\star$ stability
of Wasserstein geodesic flows (in the Eulerian framework).

\begin{proposition}\label{p:Wgeostab}
Let $(\rho,v)$ be the density and transport velocity determined by 
the Wasserstein geodesic 
that connects the measures with given densities $\rho_0,\rho_1
\colon\R^d\to[0,1]$, measurable with compact support such that
\[
 \int_{\R^d}\rho_0 = \int_{\R^d}\rho_1\,.
\]
Let $(\bar\rho^k,\bar v^k)$ be the density and transport velocity determined
by the Wasserstein geodesic 
that connects the measures with densities $\one_{\Omega_0^k}$ and $\one_{\Omega_1^k}$,
where 
$\Omega_0^k$, $\Omega_1^k$, $k=1,2,\ldots$, are bounded open sets such that 
$|\Omega_0^k|=|\Omega_1^k|=\int_{\R^d}\rho_0$ and 
\[
d_\infty(\rho_0,\one_{\Omega_0^k})
+d_\infty(\rho_1,\one_{\Omega_1^k}) \to 0 \quad\mbox{as $k\to\infty$}.
\]
Then 
\begin{equation}\label{e:rvwk-Wass}
\bar\rho^k\wksto \rho,\qquad
\bar\rho^k \bar v^k\wksto \rho v, \qquad
\bar\rho^k \bar v^k\otimes \bar v^k \wksto \rho v\otimes v ,
\end{equation}
weak-$\star$ in $L^\infty$ on $\R^d\times[0,1]$.
Consequently $0\le\rho\le1$ a.e.\ in $\R^d\times[0,1]$.
\end{proposition}

 \begin{proof}
Let $\pi$ (resp. $\pi^k$) be the optimal transport plan connecting $\rho_0$ to $\rho_1$
(resp. $\one_{\Omega_0^k}$ to $\one_{\Omega_0^k}$).
These plans take the form 
$\pi = (\id\times T)_\sharp\rho_0$ 
(resp.  $\pi^k = (\id\times T^k)_\sharp\one_{\Omega_0^k}$)
where $T$ (resp. $T^k$) is the Brenier map.
Then by \cite[Theorem~5.20]{Villani09} or \cite[Proposition 7.1.3]{AGS}, 
we know that $\pi^k$ converges weak-$\star$ to $\pi$ in the space of Radon measures
on $\R^d\times\R^d$.  

We will prove that $\bar\rho^k\bar v^k\wksto\rho v$; it will be clear that the remaining results
in \eqref{e:rvwk-Wass} are similar.  Let $\varphi\colon \R^d\times[0,1]\to \R^d$ be smooth
with compact support.  We claim that 
\begin{equation}\label{e:wkgoal}
\int_0^1\int_{\R^d} \bar\rho^k \bar v^k \varphi(x,t) \,dx\,dt \to
\int_0^1\int_{\R^d} \rho v \varphi(x,t) \,dx\,dt .
\end{equation}
Recall from \eqref{e:v}
that the geodesic velocities $\bar v^k(x,t)$ satisfy
\[
\bar v^k((1-t)z+tT^k(z),t) = T^k(z)-z. 
\]
Hence the left-hand side of \eqref{e:wkgoal} can be written 
in the form 
\[
\int_0^1 \int_{\R^d\times\R^d} (y-z)\varphi((1-t)z+ty,t)\,d\pi^k(z,y) \,dt
= \int_{\R^d\times\R^d}\psi(z,y)\,d\pi^k(z,y) \,,
\]
where
\[
\psi(z,y) = \int_0^1 (y-z)\varphi((1-t)z+ty,t)\,dt.
\]
Due to the fact that $\pi^k\wksto\pi$ and all these measures are supported in
a fixed compact set, as $k\to\infty$ we obtain the limit
\begin{equation}
\begin{split} \label{e:wklim}
 \int_{\R^d\times\R^d}\psi(z,y)\,d\pi(z,y) 
&=\int_0^1 \int_{\R^d\times\R^d} (y-z)\varphi((1-t)z+ty,t)\,d\pi(z,y) \,dt
\\
&= 
\int_0^1 \int_{\R^d} (T(z)-z)\varphi(T_t(z),t)\,\rho_0(z)\,dz \,dt\,,
\end{split}
\end{equation}
where $T_t(z)= (1-t)z+tT(z)$. 
To conclude the proof, we need to recall how $\rho$ and $v$ are
determined by displacement interpolation, in a precise technical sense
for the present case when $\rho_0$ and $\rho_1$ lack smoothness.
Indeed, from the results in Lemma~5.29 and Proposition~5.30 of 
\cite{Santa} (also see Proposition~8.1.8 of \cite{AGS}), 
we find that with the notation
\[
x_t(z,y) = (1-t)z+ty \,,
\]
the measure $\mu_t$ with density $\rho_t$ is given by the pushforward
\begin{equation}\label{e:pushrho}
    \mu_t = (x_t)_\sharp\pi = (x_t)_\sharp 
(\id\times T)_\sharp \mu_0
= (T_t)_\sharp (\rho_0\,dz)  \,,
\end{equation}
and the transport velocity is given by 
\begin{equation}\label{e:pushv}
v(x,t)=(T-\id)\circ (T_t)\inv(x).
\end{equation}
Thus we may use $T_t$ to push forward the measure $\rho_0(z)\,dz=d\mu_0(z)$
in \eqref{e:wklim} to write, for each $t\in[0,1]$,
\begin{equation}\label{e:pushTt}
\int_{\R^d} (T(z)-z)\varphi(T_t(z),t)\,\rho_0(z)\,dz 
= \int_{\R^d} v(x,t)\varphi(x,t)\,\rho_t(x)\,dx.
\end{equation}
It then follows that \eqref{e:wkgoal} holds, as desired.
\end{proof}

\begin{remark} 
    The validity of the continuity equation \eqref{e:euler1p} for $(\rho,v)$ 
    is well known and established in several sources, e.g., see \cite[Theorem~5.51]{Villani03}
    or \cite[Chapter~5]{Santa}.
    The step above going from \eqref{e:wklim} to \eqref{e:pushTt} provides an answer to 
    the related exercise 5.52 in \cite{Villani03}.  We are not aware, however, of any source where
    the momentum equation \eqref{e:euler2p} for $(\rho,v)$ is explicitly and rigorously justified.
\end{remark}

\begin{proof}[Proof of Theorem~\ref{th3}.]
Let us now finish the proof of Theorem~\ref{th3}.
As any ball in $L^\infty(\R^d\times[0,1])$ is metrizable \cite[Theorem~V.5.1]{DSbook},
we may fix a metric $d$ in a large enough ball, and select $\eps_k>0$ for each $k\in\N$
such that for the quantities
$(\rho^k,v^k,p^k):=(\rho^{\eps_k},v^{\eps_k},p^{\eps_k})$
coming from the Euler sprays of Proposition~\ref{p:wkconv},
the components of $\rho^k$, $\rho^k v^k$ and $\rho^k v^k\otimes v^k$
approximate the corresponding quantities 
$\bar\rho^{k}$, $\bar\rho^{k}\bar v^{k}$ and $\bar\rho^{k}\bar v^{k}\otimes \bar v^{k}$
that appear in Proposition~\ref{p:Wgeostab}, within distance $1/k$.
That is,
\[
\max\left(\ d(\rho^k,\bar\rho^k), \ 
d(\rho^kv^k_i,\bar\rho^k \bar v^k_i), \ 
d(\rho^kv^k_iv^k_j,\bar\rho^k \bar v^k_i\bar v^k_j) \ \right) < \frac1k.
\]
Then the convergences asserted in \eqref{e:rvwk} evidently hold.
\end{proof}

\subsection{Convergence in the stronger \pdfx{$\tlp$}{TLp} sense.} \label{r:TLp} 
The convergences described in Propositions~\ref{p:wkconv} and \ref{p:Wgeostab}
and Theorem~\ref{th3}
actually hold in a stronger sense related to the $\tlp$ metric that was introduced in
\cite{GTS} to measure differences between functions defined with respect to different measures. 
We recall the definition of the $TP^p$ metric and a number of its properties
in appendix \ref{app:TLp}.

Our first result strengthens the conclusions drawn in Proposition~\ref{p:wkconv}.
\begin{proposition} \label{p:weakconvtlp}
Under the same hypotheses as Proposition~\ref{p:wkconv} and Lemma~\ref{lem:vlv},
the map that associates $T_t(x)$ with $X^\eps_t(x)=X^\eps(x,t)$, defined by 
$Y^\eps_t = X^\eps_t\circ T_t\inv$, pushes forward $\rho_t$ to $\rho^\eps_t$ 
and we have the estimate
\begin{align}
 |x-Y^\eps_t(x)| + |v_t(x)-v^\eps_t(Y^\eps_t(x))| \le \hat K \eps
\end{align}
for all $t\in[0,1]$ and $\rho_t$-a.e. $x$. By consequence, for all $t\in[0,1]$ we have
\[
d_{T\!L^\infty}((\rho_t,v_t),(\rho^\eps_t,v^\eps_t)) \le \hat K\eps.
\]
\end{proposition}

This result follows immediately from estimate~\eqref{e:XTbound} of Lemma~\ref{lem:vlv}.
Expressed in terms of couplings,
using the transport plan that associates $X^\eps(z,t)$ with $T_t(z)$
given by the pushforward
\[
\pi^\eps = (X^\veps( \tacka, t) \times T_t)_\sharp \rho_0\,,
\]
the estimate \eqref{e:XTbound} implies that
for $\pi^\eps$-a.e.\ $(x,y)$, for all $t\in[0,1]$ we have
\[
|x-y| + |v_\veps(x,t) - v(y,t)|  \le \hat K{\eps}\,.
\]

Next we improve the conclusions of Proposition~\ref{p:Wgeostab} by 
invoking the results of 
Proposition~\ref{p:atime} in the Appendix. 

\begin{proposition}\label{p:Wgstabtlp}
Under the assumptions of Proposition~\ref{p:Wgeostab}, there exist transport maps
$\bar S^k$ that push forward $\rho_0$ to $\bar\rho_0^k=\one_{\Omega_0^k}$, 
such that 
\begin{equation}\label{e:barSk}
\|\id - \bar S^k\|_{L^\infty(\rho_0\,dx)} \to 0 \quad\mbox{as $k\to\infty$}\,,
\end{equation}
and for any such sequence of transport maps, the maps given by 
\[
\bar S^k_t = T^k_t\circ \bar S^k \circ T_t\inv
\]
push forward $\rho_t$ to $\bar\rho^k_t$ and satisfy, as $k\to\infty$,
\begin{align}
&\sup_{t\in[0,1]} \int |x - \bar S^k_t(x)|^2 \,\rho_t(x)\,dx \to 0,
\label{e:brconv}\\
&\sup_{t\in[0,1]} \int |v_t(x)-\bar v^k_t(\bar S^k_t(x))|^2\,\rho_t(x)\,dx \to0\,,
\label{e:brvconv}\\
&\sup_{t\in[0,1]} 
\int |(v_t\otimes v_t)(x)- (\bar v_t^k\otimes \bar v_t^k)(\bar S^k_t(x))|\,\rho_t(x)\,dx\to0\,.
\end{align}
\end{proposition}
\begin{proof} The existence of the maps $\bar S^k$ follow from the fact that
$d_\infty(\rho_0,\bar\rho_0^k)\to0$ as $k\to\infty$, and existence of 
optimal transport maps for these distances, see Theorem~3.24 of \cite{Santa}.
The remaining statements follow from Proposition~\ref{p:atime} in the Appendix.
\end{proof}

By combining the last two results, we obtain the following improvement
of the conclusions of Theorem~\ref{th3}.

\begin{theorem}\label{t:tlptime}
Under the same hypotheses as Theorem~\ref{th3}, we have the following. 
Let $\rho^k_0$, $\bar S^k$ be as in Proposition~\ref{p:Wgstabtlp},
and let $(\rho^k,v^k,p^k)$ be solutions of the Euler system~\eqref{e:euler1p}--\eqref{e:euler2p}
coming from the Euler sprays of Theorem~\ref{th2}, chosen as in the proof
of Theorem~\ref{th3}.  Define
\begin{equation}
S^k_t = X^k_t\circ S^k\circ T_t\inv\,.
\end{equation}
Then $(S^k_t)_\sharp (\rho_t\,dx) = \rho^k_t\,dx$, and 
\begin{align}
& \sup_{t\in[0,1]} \int |x-S^k_t(x)|^2 \,\rho_t(x)\,dx \to 0\,,
\\& \sup_{t\in[0,1]} \int |v_t(x)-v^k_t(S^k_t(x))|^2\,\rho_t(x)\,dx \to0\,,
\\& \sup_{t\in[0,1]} \int |(v_t\otimes v_t)(x)-(v^k_t\otimes v^k_t)(S^k_t(x))|\,\rho_t(x)\,dx \to0.
\end{align}

\end{theorem}

\begin{proof}
Using Proposition~\ref{p:weakconvtlp}, we can deduce that 
\begin{align}
&\sup_{t\in[0,1]} \int |S^k_t(x) - \bar S^k_t(x)|^2 \,\rho_t(x)\,dx \to 0,
\label{e:brconv2}\\
&\sup_{t\in[0,1]} \int |v^k_t(S^k_t(x))-\bar v^k_t(\bar S^k_t(x))|^2\,\rho_t(x)\,dx \to0\,,
\label{e:brvconv2}\\
&\sup_{t\in[0,1]} 
\int |(v^k_t\otimes v^k_t)(S^k_t(x))- (\bar v_t^k\otimes \bar v_t^k)(\bar S^k_t(x))|\,\rho_t(x)\,dx\to0\,.
\end{align}
Combining these with the results of Proposition~\ref{p:Wgstabtlp} finishes the proof.
\end{proof}


\section{Relaxed least-action principles for two-fluid incompressible flow and displacement interpolation}\label{s:brenier}

In a series of papers that includes 
\cite{Brenier89,Brenier92,Brenier97,Brenier99,Brenier2008,Brenier2013}, 
Brenier studied Arnold's least-action principles
for incompressible Euler flows by introducing relaxed versions 
that involve convex minimization problems, for which duality 
principles yield information about minimizers and/or minimizing sequences.

In this section, we describe a simple variant of Brenier's theories
that provides a relaxed least-action principle for a two-fluid 
incompressible fluid mixture in an Eulerian framework,
in which one fluid can be taken as vacuum.  
For this mixture model in the fluid/vacuum case,
we show that the displacement interpolant (Wasserstein geodesic)
provides the unique minimizer, and moreover the concatenated Euler sprays
that we constructed to prove Theorem~\ref{t:complete}
provide a minimizing sequence for the relaxed problem.

Our mixture model can be cast in an equivalent weaker form
that compares closely with work of 
Lopes Filho et al.~\cite{LopesNP}, who studied a 
variant of Brenier's relaxed least-action principles 
for variable density incompressible flows in a Lagrangian framework  
when the fluid density is positive everywhere.
We describe this weaker equivalent in Appendix~\ref{a:LAP},
where we also use it to prove existence of a minimizer in the general case.

\subsection{Kinetic energy and least-action principle for two fluids}

We recall that a key idea behind Brenier's work is that kinetic energy
can be reformulated in terms of convex duality, based on the idea that
kinetic energy is a jointly convex function of density and momentum. 
In order to handle possible vacuum, we extend this idea in the following way. 
Let $\hat\varrho\ge0$ be a constant (representing the density of one fluid).
We define $\hat K_{\hat\varrho}$ as the Legendre transform of the indicator function 
of the paraboloid 
\begin{equation}\label{d:parabol}
P_{\hat\varrho} = \left\{(a,b)\in \R\times\R^d: a+\tfrac12\hat\varrho|b|^2\le 0 \right\},
\end{equation}
given for $(x,y)\in \R\times\R^d$ by
\begin{equation}\label{d.hK}
\hat K_{\hat\varrho}(x,y) = \sup_{(a,b)\in P_{\hat\varrho}} ax+b\cdot y .
\end{equation}
We find the following.
\begin{lemma} Let $\hat\varrho\ge0$ and define $\hat K$ by \eqref{d.hK}.
Then $\hat K_{\hat\varrho}$ is convex, and 
\begin{equation}
\hat K_{\hat\varrho}(x,y) = \begin{cases}
\displaystyle 
\frac12 \frac{|y|^2}{\hat\varrho x} 
& \mbox{if $y\ne0$ and $\hat\varrho x>0$,}\\
0 & \mbox{if $y=0$ and $x\ge 0$,}\\
+\infty &  \mbox{else}.
\end{cases}
\end{equation}
In case $\hat\varrho>0$, we have the scaling property
\begin{equation}\label{s.hK}
\hat K_{\hat\varrho}(\hat\varrho x,\hat\varrho y)= \hat K_1(x,y).
\end{equation}
\end{lemma}
The proof of this lemma is a straightforward calculation based on cases
that we leave to the reader.  We emphasize that $\hat\varrho=0$ is allowed.
Indeed, for $\hat\varrho=0$, $\hat K_0$ reduces to the indicator function for the 
closed half-line 
\[
\{(x,y):y=0,\ x\ge0\}. 
\]

Suppose $c\in\R$ represents the `concentration' of one fluid 
and $m\in\R^d$ represents the `momentum' of this fluid, at some 
point in the flow.  If $\hat K_{\hat\varrho}(c,m)<+\infty$, then $c\ge0$ and
$m=\hat\varrho c v$ for some `velocity' $v\in\R^d$ which satisfies 
\begin{equation} 
\hat K_{\hat\varrho}(c,m) = \frac12 \hat\varrho c|v|^2.
\end{equation}

Next we begin to describe our relaxed least-action principle for 
two-fluid incompressible flow. 
Consider fluid flow inside a large box for unit time, with
\[
\Omega = [-L,L]^d \,,\qquad Q = \Omega\times[0,1]\,.
\]
Let $\hat\varrho_i$,  $i=0,1$,
be constants representing the densities of two fluids, with 
$\hat\varrho_1>\hat\varrho_0\ge0$.  (More fluids could be considered,
but we have no reason to do so at this point.)
Next we let $c_i(x,t)$, $i=0,1$, represent the concentration of fluid $i$ 
at the point $(x,t)\in Q$.
For classical flows, the fluids should occupy non-overlapping regions of space-time,
meaning that the concentrations are characteristic functions
$c_i=\one_{Q_i}$ with 
\begin{equation}\label{d:Omega2}
\Qdom_i = \bigcup_{t\in[0,1]} \Omega_{i,t}\times\{t\} \,,
\qquad Q=\bigsqcup_i Q_i \,.
\end{equation}
The requirement $c_i(x,t)\in\{0,1\}$ will be relaxed, however,
to the requirement $c_i(x,t)\in[0,1]$. This provides a convex restriction
that heuristically allows `mixtures' to form (by taking weak limits, say).

Writing $m_i(x,t)$ for the momentum of fluid $i$ at $(x,t)\in Q$, 
the action to be minimized is the total kinetic energy
\begin{equation}\label{d:Ktot}
K(c,m) = \sum_i \int_Q \hat K_{\hat\varrho_i}(c_i,m_i) \,dx\,dt \,,
\end{equation}
subject to three types of constraints---incompressibility, 
transport that conserves the total mass of each fluid,
and endpoint conditions. We require
\begin{equation}\label{c:masspt}
\sum_i c_i = 1 \qquad\mbox{a.e.\ in $Q$},
\end{equation}
\begin{equation}
\hat\varrho_i \D_t c_i +\nabla\cdot m_i = 0 
\qquad\mbox{in $Q$ for all $i$,}
\end{equation}
\begin{equation}\label{c:massi}
\frac{d}{dt} \int_\Omega \hat\varrho_ic_i =0
\qquad\mbox{for  $t\in [0,1]$ for all $i$},
\end{equation}
and fixed endpoint conditions at $t=0,1$:
\begin{equation}
c_i(x,0)= c_{i0}(x)\,,\qquad 
c_i(x,1)= c_{i1}(x)\,, 
\end{equation}
where $c_{i0},c_{i1}\in L^\infty(\Omega,[0,1])$ are prescribed for each $i$
in a fashion compatible with the constraints \eqref{c:masspt} and \eqref{c:massi}, satisfying
\begin{equation}\label{c:ICs}
\sum_i c_{i0} = \sum_i c_{i1} = 1 \quad\mbox{a.e.\ in $Q$},
\qquad
\int_{\Omega} c_{i0} = \int_{\Omega} c_{i1}\,,\quad i=0,1.
\end{equation}
For unmixed, classical flows, these data take the form of characteristic 
functions:
\begin{equation}\label{c:classical}
c_{i0}(x) = \one_{\Omega_{i,0}}\,, \qquad
c_{i1}(x) = \one_{\Omega_{i,1}}\,.
\end{equation}

The constraints above are more properly written and collected in the following weak form,
required to hold for all test functions $p$, $\phi_i$ in the space $C^0(Q)$ 
of continuous functions on $Q$, having $\D_t\phi_i$, $\nabla_x\phi_i$
also continuous on $Q$, for $i=0,1$:
\begin{align}
0 = &\int_Q p-\sum_i 
\int_Q 
\bigl( \left(p+\hat\varrho_i\D_t\phi_i\right)c_i
+ \nabla_x\phi_i\cdot m_i
\bigr)
\nonumber\\&
+\sum_i \hat\varrho_i \left(
 \int_{\Omega}  c_{i1}(x)\phi_i(x,1)\,dx
- \int_{\Omega}  c_{i0}(x)\phi_i(x,0)\,dx\right)\,.
\label{e:wkcon}
\end{align}

Let us now describe precisely the set $\calA_K$ of
functions $(c,m)$ that we take as admissible for the relaxed least-action principle.
We require $c_i\in L^\infty(Q,[0,1])$.
As we shall see below, it is natural to require that the path
\[
t\mapsto c_i(\cdot, t) \,dx
\]
is weak-$\star$ continuous into the space of 
 signed Radon measures on $\Omega$,
and that 
$m_i=\hat\varrho_i c_iv_i$ with 
$v_i\in L^2(c_i\,dx\,dt)$ if $\hat\varrho_i>0$.
Then the action in \eqref{d:Ktot} becomes
\begin{equation}
K(c,m) = \sum_i \int_Q \frac12\hat\varrho_i c_i |v_i|^2 \,.
\end{equation}
When $\hat\varrho_0=0$, we require $m_0=0$ a.e.,
since this condition is necessary to have
$K(c,m)<\infty$ in \eqref{d:Ktot}. In this case we have
\begin{equation}
K(c,m) = \int_Q \frac12\hat\varrho_1c_1|v_1|^2 \,,
\end{equation}
and the constraints on $c_0$ from  \eqref{e:wkcon} reduce simply
to the requirement that $c_0=1-c_1$. 

We let $\calA_K$ denote the set of functions $(c,m)$ that have the properties
required in the previous paragraph
and satisfy the weak-form constraints \eqref{e:wkcon}.
Our \textit{relaxed least-action two-fluid problem} is to find  $(\bar c,\bar m)\in\calA_K$
with
\begin{equation}\label{la:relax1}
K(\bar c, \bar m) = \inf_{(c,m)\in\calA_K} K(c,m).
\end{equation}

A formal description of classical  critical points of the action in \eqref{la:relax1},
subject to the constraints in \eqref{e:wkcon}, 
and with each $c_i$ a characteristic function of smoothly
deforming sets as in \eqref{d:Omega2},
will lead to classical Euler equations for two-fluid incompressible flow,
along the lines of our calculation in section~\ref{s:geo}, which applies 
in the case $\hat\varrho_0=0$.

%
%

\subsection{Wasserstein geodesics as minimizers of relaxed action}

We defer to Appendix~\ref{a:LAP} any further general discussion of the
relaxed least-action problem \eqref{la:relax1},
and focus now on the case $\hat\varrho_0=0$ 
corresponding to a fluid/vacuum mixture. 
We take $\hat\varrho_1=1$ for convenience. 
\begin{theorem}\label{t:W} Suppose $\hat\varrho_0=0$, and 
$\rho_0$, $\rho_1\colon\R^d\to[0,1]$ are measurable 
with compact support and equal integrals  in $(-L,L)^d$.
Then the relaxed least-action problem in \eqref{la:relax1}, 
with endpoint data determined by $c_{10}=\rho_0$, $c_{11}=\rho_1$, 
has a unique solution $(\bar c,\bar m)$ given inside $Q$ by
\begin{equation}\label{d:barcm}
\bar c_1 = \rho, \quad \bar m_1 = \rho v, 
\qquad \bar c_0 = 1-\rho, \quad \bar m_0=0,
\end{equation}
in terms of the 
displacement interpolant $(\rho,v)$ 
(described in section 2 and \eqref{e:pushrho},\eqref{e:pushv})
between the measures $\mu_0$ and $\mu_1$ with densities 
$\rho_0$ and $\rho_1$.
\end{theorem}

\begin{proof} It is clear from the pushforward description of 
\eqref{e:pushrho}--\eqref{e:pushv} that 
$(\bar c,\bar m)$ as defined in \eqref{d:barcm}
belongs to the admissible set $\calA_K$, due to the facts
that (i) $0\le\rho\le 1$ by the last assertion of Proposition~\ref{p:Wgeostab} and
(ii) the support of $(\rho,v)$ is compactly contained in $\Omega$
due to \eqref{e:Tt} and \eqref{e:pushv}.
We then have, since $\hat\varrho_0=0$,
\[
K(\bar c,\bar m) =  \int_Q \frac12 \rho|v|^2 
=  \int_0^1\int_{\R^d} \frac12 \rho |v|^2\,dx\,dt
\]
because the pair $(\rho,v)$ is defined on $\R^d\times[0,1]$
and is zero outside $Q$.  But similarly,
for \textit{any} admissible pair $(c,m)\in \calA_K$,
if we extend $(c_1,v_1)$ by zero outside $Q$, we have
\[
K(c,m) =  \int_0^1\int_{\R^d} \frac12 c_1 |v_1|^2\,dx\,dt\,
\]
and $(c_1,v_1)$ determines a narrowly continuous path
of measures $t\mapsto \mu_t = c_1\,dx$ on $\R^d$
with $v\in L^2(\mu)$ that satisfies the continuity equation.
It is known that $(\rho,v)$ minimizes this expression 
over this wider class of paths of measures, due to the characterization
of Wasserstein distance by Benamou and Brenier \cite{BenBre00}, 
see \cite[Thm. 8.1]{Villani03}. 
By consequence we obtain that $(\bar c,\bar m)$ as defined by 
\eqref{d:barcm} is indeed a minimizer of the relaxed least-action problem
\eqref{la:relax1}.

 Because the Wasserstein minimizing path is unique (as discussed in section 2),
 it follows that any minimizer in \eqref{la:relax1} must be given
 as in \eqref{d:barcm}.
 \end{proof}

 \begin{proposition} The family of incompressible flows (concatenated Euler sprays)
 given for all small $\eps>0$ by  Theorem~\ref{t:complete}(b)
 determine a minimizing family $(c^\eps,m^\eps)$ 
 for the relaxed least-action principle \eqref{la:relax1} according to 
 \[
 c_1^\eps = \rho^\eps , \quad m_1^\eps = \rho^\eps v^\eps,
 \qquad c_0^\eps = 1-\rho^\eps , \quad m_0^\eps = 0.
 \]
 That is, $(c^\eps,m^\eps)\in\calA_K$ and $\lim_{\eps\to0} K(c^\eps,m^\eps)= \inf_{\calA_K} K(c,m)$. 
 \end{proposition}
 
We remark that in case the endpoint data are classical (unmixed) as in \eqref{c:classical},
 we are not able to use
 the individual Euler sprays that we construct for the proof of Theorem~\ref{th2}
 to obtain a similar result. The reason is  
 that the target set $\Omega_{1,1}=\Omega_1$ is not hit exactly by our Euler sprays, 
 and this means that the corresponding concentration-momentum 
 pair $(c^\eps,m^\eps)\notin \calA_K$ because it would not satisfy the constraint \eqref{e:wkcon} as required. 
 We conjecture, however, that for small enough $\eps>0$, Euler sprays
 can be constructed that hit an arbitrary target shape $\Omega_1$ (up to a set of measure zero).
 If that is the case, these Euler sprays would 
 similarly provide a minimizing family for the relaxed least-action principle
 \eqref{la:relax1}.


\section{A Schmitzer-Schn\"orr-type shape distance without volume constraint}\label{s:discuss}

Theorem \ref{th1} establishes that restricting the Wasserstein metric
to paths of shapes of fixed volume does not provide 
a new notion of distance on the space of such shapes. 
Namely it shows that for paths in the space of shapes of fixed volume,
the infimum of the length of paths between two given shapes is the Wasserstein distance. 

\emph{Volume change.}
It is of interest to consider a more general space of shapes in order to compare
shapes of different volumes.  
In particular, the Schmitzer and Schn\"orr
\cite{SchSch2013}  considered a space that corresponds to the set of bounded,
simply connected domains in $\R^2$ with smooth boundary and arbitrary positive area.  
To each such shape $\Omega$ one associates as its corresponding \emph{shape measure}
the probability measure having uniform density on $\Omega$, denoted by  
\begin{equation}\label{e:shmeas}
    \calU_{\Omega} = \frac{1}{|\Omega|} \one_{\Omega}.  
\end{equation}

We consider here this same association between sets and shape measures, 
but allow for more general shapes.  Namely for fixed dimension $d$, 
let us consider shapes as bounded measurable subsets of $\R^d$ with positive volume.  
Let $\mathcal C$ be the set of all shape measures corresponding to such shapes.  
Thus $\calC$ is the set of all uniform probability distributions of bounded support.

One can formally consider $\mathcal C$ as a submanifold of the space of
probability measures of finite second moment, endowed with Wasserstein distance.  
Then we define a distance between shapes as we did in \eqref{d:ds}, requiring
\begin{equation}\label{e:actionC}
    d_{\mathcal C}(\Omega_0, \Omega_1)^2 = 
\inf {\mathcal A}\,,\qquad{\mathcal A}=\int_0^1\!\!\int_{\R^d}\, \rho|v|^2\,dx\,dt\,,
\end{equation}
where $\rho=(\rho_t)$ is now required to be a path of shape measures in $\mathcal C$,
with endpoints
\begin{equation}\label{e:ss-ends}
    \rho_0 = \calU_{\Omega_0}\,, \qquad \rho_1=\calU_{\Omega_1}\,,
\end{equation}
and transported according to the continuity equation \eqref{e:inf1a} 
with a velocity field $v\in L^2(\rho\,dx\,dt)$.

Because the characteristic-function restriction \eqref{e:inf1b} 
is replaced by the weaker requirement that $\rho_t$ has a uniform density,
for any two shapes of equal volume scaled to unity for convenience, it is clear that
\begin{equation}
    d_s(\Omega_0,\Omega_1) \ge d_{\calC}(\Omega_0,\Omega_1)  \ge d_W({\Omega_0},{\Omega_1}).
\end{equation}
Then as a direct consequence of Theorem \ref{th1}, we have
\begin{equation}
d_{\mathcal C}(\Omega_0, \Omega_1) = 
d_W \left({\Omega_0},{\Omega_1} \right). 
\end{equation}
By a minor modification of the arguments of section~\ref{s:dsdw}, in general we have the following.
\begin{theorem} \label{t:distC}
Let $\Omega_0$ and $\Omega_1$ be any two shapes of positive volume. Then
\[ 
d_C(\Omega_0,\Omega_1)=d_W(\calU_{\Omega_0},\calU_{\Omega_1}).
\] 
\end{theorem}
\begin{proof} By a simple scaling argument, we may assume 
$\min\{|\Omega_0|,|\Omega_1|\}\ge 1$ without loss of generality,
so that both $\rho_0,\rho_1\le 1$.
Then the concatenated Euler sprays provided by Theorem~\ref{t:complete}(b) 
supply a path of shape measures in $\mathcal C$ (actually shape densities), 
with action converging to $d_W(\calU_{\Omega_0},\calU_{\Omega_1})^2$.
\end{proof}

\emph{Smoothness.}
For dimension $d=2$, Theorem~\ref{t:distC} does not apply to describe distance 
in the space of shapes considered by Schmitzer and Schn\"orr in \cite{SchSch2013},
however, for as we have mentioned, they consider shapes to be 
bounded simply connected domains with smooth boundary. 

One point of view on this issue is that it is nowadays reasonable for many
purposes to consider `pixelated' images and shapes, 
made up of fine-grained discrete elements, to be valid approximations 
to smooth ones.  Thus the microdroplet constructions considered in this paper,
which fit with the mathematically natural regularity conditions 
inherent in the definition of Wasserstein distance, need not be thought unnatural
from the point of view of applications.

Nevertheless one may ask whether the infimum of path length in the 
space of smooth simply connected shapes is
still the Wasserstein distance, as in Theorem \ref{t:distC}.
Our proof of Theorem~\ref{th1} in Section \ref{s:dsdw} does not provide paths
in this space because the union of droplets is disconnected. 
However, the main mechanism by which we efficiently
transport mass, namely by ``dividing" the domain into small pieces
(droplets) which almost follow the Wasserstein geodesics, is still available. 
In particular, by creating many deep creases in the domain it might be effectively
`divided' into such pieces while still remaining connected and smooth. 
Thus we conjecture that even in the class of smooth sets considered in \cite{SchSch2013}, 
the geodesic distance is the Wasserstein distance 
between uniform distributions as in Theorem \ref{t:distC}.

\emph{Geodesic equations.} It is also interesting to compare our Euler droplet equations
from subsection~\ref{ss:geo} with the formal geodesic equations for smooth
critical paths of the action $\calA$ in the space $\calC$ of uniform distributions. 
These equations correspond to equation (4.12) of
Schmitzer and Schn\"orr in \cite{SchSch2013}. 

These geodesic equations may be derived in a manner almost identical to the
treatment in subsection~\ref{ss:geo} above. The principal difference is that 
due to \eqref{e:drho}, the divergence of the Eulerian velocity 
may be a nonzero function of time, constant in space:
\[
    \nabla\cdot v = c(t),
\]
and the same is true of virtual displacements $\tilde v$.
The variation of action now satisfies
\begin{equation}
    \frac{\delta\calA}2 = \left.\int_{\Omega_t} v\cdot \tilde v\,\rho\,dx\,\right|_{t=1}
-\int_0^1  \int_{\Omega_t}
(\D_t v+v\cdot\nabla v)\cdot \tilde v \,\rho\,dx\,dt.
\end{equation}
Now, the space of vector fields orthogonal to all constant-divergence fields on
$\Omega_t$ is the space of gradients $\nabla p$ such that $p$ vanishes on the
boundary and has average zero in $\Omega_t$, satisfying
\begin{equation}\label{e:ssp}
    p=0 \quad\mbox{on $\D\Omega_t$}, \qquad \int_{\Omega_t} p \,dx = 0.
\end{equation}
Because $\rho$ is spatially constant and $\tilde v$ can be (locally in time) arbitrary
with spatially constant divergence, necessarily $u = -(\D_t v+v\cdot\nabla v)$
is such a gradient.  The remaining considerations in section~\ref{ss:geo}
apply almost without change, and we conclude that $v=\nabla\phi$ where
\begin{equation}\label{e:ssphi}
    \D_t \phi + \frac12|\nabla\phi|^2 + p =0, \qquad \Delta\phi = c(t),
\end{equation}
where $c(t)$ is spatially constant in $\Omega_t$. 

These fluid equations differ from those in section~\ref{ss:geo} in that
$\phi$ gains one degree of freedom (a multiple of the solution of $\Delta
\phi=1$ in $\Omega_t$ with Dirichlet boundary condition) while the pressure $p$
loses one degree of freedom (as its integral is constrained).

They have elliptical droplet solutions given by displacement interpolation of
elliptical Wasserstein droplets as in subsection~\ref{ss:Wdrop}, 
because pressure vanishes and density is indeed spatially constant for these interpolants.
Because they are Wasserstein geodesics,  these particular solutions are also
length-minimizing geodesics in the shape space $\calC$. 

We remark that unlike in the case of Euler sprays, disjoint superposition will not
yield a geodesic in general. This is because the requirement of spatially uniform density
leads to a global coupling between all shape components.  It seems likely that
length-minimizing paths in $\calC$ will not generally exist even locally, but
we have no proof at present.

\appendix

\section{Some simple facts about subgradients}\label{a:subg}

For the convenience of readers, we include here proofs of a few facts about
subgradients that we use in section~\ref{s:spray} for the proof of Theorem~\ref{th2}.
The proofs are standard but may be 
hard for readers to extract from monographs on the subject.


\begin{proposition} \label{p:subg}
Let $H$ be a Hilbert space, and let $\varphi\colon H\to(-\infty,\infty]$ 
be convex, lower semi-continuous, and proper (i.e., somewhere finite).
Let $S(x) = \frac12\|x\|^2+\varphi(x)$.
Then:
\begin{itemize}
\item[i.]The subgradient $\D\varphi$ is a monotone operator.
\item[ii.] $\D S(x) = x + \D\varphi(x)$ for all $x\in H$.
\item[iii.] The range of $\D S$ is all of $H$. 
I.e., for all $y\in H$ there exists $x\in H$ and $z\in\D\varphi(x)$ 
such that $y= x+z$.
\end{itemize}
\end{proposition}
\begin{proof}
i. Given any $x,\hat x \in H$, 
$z\in\D\varphi(x)$, $\hat z\in\D\varphi(\hat x)$, 
by the definition of
$\D\varphi(x)$ and $\D\varphi(\hat x)$ respectively we have
\[
\varphi(\hat x) - \varphi(x) \ge \ip{z,\hat x - x}, \qquad
\varphi(x) - \varphi(\hat x) \ge \ip{\hat z,x -\hat x},
\]
whence
$ 0\le \ip{z-\hat z,x-\hat x}$. 
This proves $\D\varphi$ is monotone.

ii. 1. Let $z\in\D\varphi(x)$. We claim $x+z\in\D S(x)$.
Indeed, for all $h\in H$,
\[
\frac12\|x+h\|^2 + \varphi(x+h) \ge \frac12\|x\|^2 + \varphi(x)
+ \ip{z+x,h}\,.
\]

2. Suppose $z\notin\D\varphi(x)$. We claim $z+x\notin\D S(x)$. 
We know there exists $h\in H$ such that 
\[
t\inv(\varphi(x+th)-\varphi(x)) - \ip{z,h} < 0
\]
for $t=1$, hence for all $t\in(0,1]$ by convexity. 
Then for sufficiently small $t>0$ we can add $\frac12 t\|h\|^2$
to the left-hand side and conclude that for small $t>0$,
\[
\frac12 \|th\|^2 + \varphi(x+th) < \varphi(x) + \ip{z,th}
\]
whence
\[
\frac12 \|x+th\|^2 + \varphi(x+th) < \frac12\|x\|^2+\varphi(x) + \ip{z+x,th}
\]
Thus $z+x\notin\D S(x)$.

iii. Let $y\in H$ and define 
$\hat S(x)=S(x)-\ip{y,x}=\frac12\|x\|^2+\varphi(x)-\ip{y,x}$.
Due to our hypotheses, $\hat S$ has a minimum at some $x\in H$.
This implies that for all $h\in H$,
\[
\frac12\|x+h\|^2 +\varphi(x+h)\ge 
\frac12\|x\|^2 +\varphi(x)-\ip{y,h},
\]
which means that $y\in\D S(x)=x+\D\varphi(x)$.
\end{proof}

\section{\pdfx{$\tlp$}{TLp} convergence and stability of Wasserstein geodesics} \label{app:TLp}

Here we recall the notion of $\tlp$ convergence as introduced in \cite{GTS}, 
which provides a more precise comparison between 
Wasserstein geodesics than the notion of weak convergence does alone.
We recall some of the basic properties, establish new ones and use them to prove the 
convergence of Wasserstein geodesics considered as weak solutions to pressureless 
Euler equation.

The $\tlp$ metric provides a natural setting for comparing optimal transport maps between different probability measures. 
Let $\calP_p(\R^d)$ be the space of probability measures on $\R^d$ with finite $p$-th moments.
On the space $\tlp(\R^d)$, consisting of all ordered pairs 
$(\mu,g)$ where $ \mu\in\calP_p(\R^d)$ 
and $ g \in L^p(\mu)$, the metric is given as follows: 
For $1\le p<\infty$,
\[ d_{\tlp} (( \mu_0,   g_0), ( \mu_1,   g_1)) = 
\inf_{\pi \in \Pi( \mu_0,  \mu_1)}\left(
\int |x-y|^p + | g_0(x) -  g_1(y)|^p\  d \pi(x,y) \right)^{1/p}\,,
\]
and \[
d_{T\!L^\infty}((\mu_0,g_0),(\mu_1,g_1)) = 
\inf_{\pi \in \Pi( \mu_0,  \mu_1)}
\esssup_\pi (|x-y|+|g_0(x)-g_0(y)| ) \,,
\]
where $ \Pi( \mu_0,  \mu_1)$ is the set of transportation plans (couplings) 
between $ \mu_0$ and $ \mu_1$. 

The following result estabilishes a stability property for optimal transport maps,
as a consequence of a known general stability property for optimal plans.

\begin{proposition} \label{otstab}
Let $\mu,\mu_k\in \calP_p(\R^d)$ be probability measures absolutely continuous with respect to Lebesgue measure, and let $\nu,\nu_k\in \calP_p(\R^d)$, for each $k\in\N$.
Assume that 
\[ d_p(\mu_k, \mu) \to 0\; \te{ and } \;  d_p(\nu_k,\nu) \to 0 \; \te{ as }  k \to \infty. \]
Let $T_k$ and $T$ be the optimal transportation maps between $\mu_k$ and $\nu_k$, and $\mu$ and $\nu$, respectively. Then
\[ (\mu_k,T_k) \tlpto (\mu,T)  \te{  as } k \to \infty.\]
\end{proposition}
\begin{proof}
The measures $\pi_k= (\id \times T_k)_\sharp \mu_k$ and $\pi = (\id \times T)_\sharp \mu$ are the optimal transportation plans between $\mu_k $ and $\nu_k$, and $\mu$ and $\nu$, respectively.  
By stability of optimal transport plans 
(Proposition 7.1.3 of \cite{AGS} or Theorem 5.20 in \cite{Villani09}) 
the sequence $\pi_k$ is precompact with respect to weak convergence and each of
its subsequential limits is an optimal transport plan between $\mu$ and $\nu$.
Since $\pi$ is the unique optimal transportation plan between $\mu$ and $\nu$
the sequence $\pi_k$ converges to $\pi$. Furthermore, by 
Theorem~5.11 of \cite{Santa} or Remark~7.1.11 of \cite{AGS},
 \begin{align*}
 \lim_{k \to \infty}  \int |x|^p + |y|^p \ d \pi_k(x,y)  & =
 \lim_{k \to \infty}     \int |x|^p\, d\mu_k + \int |y|^p\, d \nu_k \\
 & =   \int |x|^p\, d\mu + \int |y|^p\, d \nu  =     \int |x|^p + |y|^p\  d \pi(x,y).
 \end{align*}
By Lemma~5.1.7 of \cite{AGS}, it follows the $\pi_k$ have uniformly integrable $p$-th moments, 
therefore 
\[
d_p(\pi_k, \pi)\to 0 \quad\mbox{ as $k \to \infty$},
\]
 by  Proposition 7.1.5 in \cite{AGS}.
 Hence there exists (optimal) $\gamma_k \in \Pi(\pi, \pi_k)$ such that
 \begin{equation} \label{gamkto0}
 \int |x - \tilde x|^p + |y- \tilde y|^p \ d\gamma_k(x,y,\tilde x, \tilde y) \to 0 \quad \te{ as } k \to \infty.
 \end{equation}
Since $\pi$-almost everywhere $y = T(x)$ and $\pi_k$-almost everywhere $\tilde y = T_k(\tilde x)$ and the support $\supp\gamma_k$ of $\gamma_k$ is contained in 
$\supp\pi\times\supp\pi_k$,
we conclude that $\gamma_k$-almost everywhere $(x,y, \tilde x, \tilde y) = (x,T(x), \tilde x, T_k(\tilde x))$. Therefore 
 \begin{equation*} \label{loc13}
 \int |x - \tilde x|^p + |T(x)- T_k(\tilde x)|^p\  d\gamma_k(x,y,\tilde x, \tilde y) \to 0 \quad \te{ as } k \to \infty.
 \end{equation*}   
 Finally let $\theta_k$ be the projection of $\gamma_k$ to $(x, \tilde x)$ variables. By its definition 
 $\theta_k \in \Pi(\mu, \mu_k)$ and by above
 \begin{equation}\label{e:thetaconv}
  \int |x - \tilde x|^p + |T(x)- T_k(\tilde x)|^p\ d\theta_k(x, \tilde x) \to 0 \quad \te{ as } k \to \infty.
\end{equation}
 Thus $(\mu_k, T_k) \tlpto (\mu,T)$.
\end{proof}

We now consider the convergence of Wasserstein geodesics between measures 
$\mu_k$ and $\nu_k$ as in the Lemma \ref{otstab}, treating only the case $p=2$. 
We recall that particle paths along these geodesics are given by 
\[
T_{k,t}(x) = (1-t)x + t T_k(x)\,.
\]
The displacement interpolation between $\mu_k$ and $\nu_k$, and particle
velocities (in Eulerian variables) along the geodesics, are given by  
(cf.~\eqref{e:pushrho}--\eqref{e:pushv})
\[
\mu_{k,t} = {T_{k,t}}_\sharp \mu_k\,, \quad  v_{k,t}= (T_k-\id)\circ T_{k,t}\inv
\,, \quad t\in[0,1). 
\]
If $\nu_k$ is absolutely continuous with respect to Lebesgue measure, then $t=1$ is allowed.
We also recall that 
\[
\int |v_{k,t}(z)|^2 d \mu_{k,t}(z) = \int  |v_{k,0}(x)|^2 d \mu_k(x) = d_2^2(\mu_k, \nu_k).
\]
Furthermore it is straightforward to check that $t \mapsto (\mu_{k,t}, v_{k,t})$ 
is Lipschitz continuous into $T\!L^2(\R^d)$, satisfying
for $0\le s<t<1$ 
\begin{equation}
(t-s)d_2(\mu_k,\nu_k)=d_2(\mu_{k,t},\mu_{k,s}) \le 
d_{\tltwo}((\mu_{k,t},v_{k,t}),(\mu_{k,s},v_{k,s}))
\le (t-s)d_2(\mu_k,\nu_k).
\end{equation}

\begin{proposition} \label{tlpconvE}
Under the  assumptions of Proposition \ref{otstab} for the case $p=2$,
as $k\to\infty$ we have
\begin{align}
& \sup_{t\in[0,1]} d_2(\mu_{k,t},\mu_t) \to 0,\label{tlp1}\\
& \sup_{t\in[0,1)} d_{\tltwo}( (\mu_{k,t},v_{k,t}),(\mu_t,v_t)) \to 0,\label{tlp2}\\
& \sup_{t\in[0,1)} d_{T\!L^1}( (\mu_{k,t},v_{k,t}\otimes v_{k,t}),(\mu_t,v_t\otimes v_t)) \to 0.
\label{tlp3}
\end{align}
If the measures $\nu_k$ and $\nu$ are absolutely continuous
with respect to Lebesgue measure then the convergence in \eqref{tlp2} and \eqref{tlp3}
hold also for $t\in[0,1]$.
\end{proposition}

\begin{proof}
Let $\pi \in \Pi(\mu, \nu)$, $\pi_k \in \Pi(\mu_k, \nu_k)$, and 
$\gamma_k \in \Pi( \pi, \pi_k)$ be as in the proof of Proposition \ref{otstab}. 
Similarly to $\theta_k$, we define $\theta_{k,t} = (z_t \times  z_t)_{\sharp} \gamma_k$ 
where 
\[
z_t(x,y) = (1-t)x+ty \quad\mbox{and}\quad
(z_t \times z_t)(x,y,\tilde x, \tilde y) = (z_t(x,y), z_t(\tilde x, \tilde y)). 
\]
We note that $\theta_{k,t} \in \Pi(\mu_t, \mu_{k,t})$ and hence, for all $t\in[0,1]$,
\begin{align*}
d_2(\mu_t, \mu_{k,t})^2  & \leq \int |z-\tilde z|^2 d \theta_{k,t}(z,\tilde z) \\
& = \int |(1-t)(x-\tilde x) + t(y - \tilde y)|^2 d \gamma_k(x, y,  \tilde x, \tilde y) \\
& \leq 2 \int  |x - \tilde x|^2 + |y - \tilde y|^2 d  \gamma_k(x, y,  \tilde x, \tilde y) \,,
\end{align*}
which by \eqref{gamkto0} converges to $0$ as $k \to \infty$. 

We use the same coupling $\theta_{k,t}$ to compare the velocities.
Using that $\gamma_k$-almost everywhere $(x,y, \tilde x, \tilde y) = (x,T(x), \tilde x, T_k(\tilde x))$, for any $t\in[0,1)$ we obtain
\begin{align*}
& \int |v_t(z) - v_{k,t}(\tilde z)|^2\ d \theta_{k,t}(z, \tilde z)   \\
&\qquad  = 
  \int |v_t((1-t)x + ty) - v_{k,t}((1-t) \tilde x + t \tilde y)|^2\ d \gamma_k(x, y,  \tilde x, \tilde y) \\
&\qquad = \int |v(T_t(x)) - v_{k,t}(T_{k,t}(\tilde x))|^2\ d \theta_k(x, \tilde x) \\
&\qquad = \int |v_0(x)-v_{k,0}(\tilde x)|^2 \ d \theta_k(x, \tilde x) \\
&\qquad \leq 2 \int |x - \tilde x|^2 + |T(x) - T_k(\tilde x)|^2 \ d \theta_k(x, \tilde x)\,, 
\end{align*}
which converges to $0$ as $k \to \infty$, as in \eqref{e:thetaconv}.

The convergence in \eqref{tlp3} 
is a straightforward consequence through use of Cauchy-Schwarz inequalities.
\end{proof}

\begin{remark}
If the target measure $\nu_k$ is not 
absolutely continuous with respect to Lebesgue measure, 
then $T_k$ may fail to be invertible on the support
of $\nu_k$ and $(\mu_{k,t},v_{k,t})$ may fail to converge as $t\to1$
to some point in $\tltwo(\R^d)$ due to oscillations in velocity.
However, if $\nu_k$ and $\nu$ are 
absolutely continuous with respect to Lebesgue measure, 
then the curves 
$t \mapsto (\mu_{k,t}, v_{k,t})$, $t\mapsto (\mu_t,v_t)$ extend 
as continuous maps into $\tltwo$ for all $t\in[0,1]$, and the
uniform convergences in \eqref{tlp2}--\eqref{tlp3} hold on $[0,1]$.
\end{remark}

A number of properties of the $\tlp$ metric are established in Section 3 of \cite{GTS}
for measures supported in a fixed bounded set 
One useful characterization of $\tlp$-convergence in this case 
is stated in Proposition 3.12 of \cite{GTS}, which implies the following.

\begin{proposition}\label{p:ax}
Let $D\subset \R^d$ be open and bounded, and  
let $\mu$ and $\mu_k$ ($k\in\N$) be probability measures on $D$,
and suppose $\mu$ is absolutely continuous with respect to Lebesgue measure. 
Call a sequence of transport maps $(S_k)$ that push forward 
$\mu$ to $\mu_k$ (satisfying ${S_k}_\sharp\mu=\mu_k$) {\em stagnating} if 
\begin{equation}
\lim_{n\to\infty}\int_D |x-S_k(x)|\,d\mu(x) = 0 \,.
\end{equation}
Then the following are equivalent, for $1\le p<\infty$.
\begin{itemize}
\item[(i)] $(\mu_k,f_k)\tlpto (\mu,f)$ as $k\to\infty$.
\item[(ii)] $\mu_k$ converges weakly to $\mu$ and for every 
stagnating sequence $(S_k)$ we have
\begin{equation}\label{e:fnf}
\int_D |f(x)-f_k(S_k(x))|^p \,d\mu(x) \to 0 \quad\mbox{as $k\to\infty$}.
\end{equation}
\end{itemize}
Moreover, for (ii) to hold it suffices that \eqref{e:fnf} holds
for any \textit{single} stagnating sequence $(S_k)$.
\end{proposition}

This result together with Proposition~\ref{tlpconvE} yields the following. 
\begin{proposition}\label{p:atime}
Make the same assumptions as in Proposition~\ref{tlpconvE}, and 
assume all measures $\mu_k$, $\mu$, $\nu_k$, $\nu$ are absolutely continuous
with respect to Lebesgue measure and have support in a bounded open set $D$. 
Then for any stagnating sequence of transport maps $(S_k)$ 
that push forward $\mu$ to $\mu_k$, with the notation
\[
S_{k,t} = T_{k,t}\circ S_k\circ T_t\inv
\]
the sequence $(S_{k,t})$ pushes forward $\mu_t$ to $\mu_{k,t}$ and is stagnating,
and as $k\to\infty$,
\begin{align}
 &\sup_{t\in[0,1]} \int |x-S_{k,t}(x)|^2\,d\mu_t(x) \to 0\,, \label{e:ttlp1}
\\& \sup_{t\in[0,1]}  \int|v_t(x)-v_{k,t}(S_{k,t}(x))|^2\,d\mu_t(x)  \to 0\,, \label{e:ttlp2}
\\& \sup_{t\in[0,1]} \int |(v_t\otimes v_t)(x)-(v_{k,t}\otimes v_{k,t})(S_{k,t}(x))|\,d\mu_t(x)  \to 0\,. \label{e:ttlp3}
\end{align}
\end{proposition}

\begin{proof} First we note that indeed 
\[
\mu_{k,t}=(T_{k,t})_\sharp\mu_k = 
(T_{k,t}\circ S_k)_\sharp\mu = (S_{k,t})_\sharp\mu_t.
\]
Next, fix any $t\in[0,1]$.
Because $d_2(\mu_{k,t},\mu_t)\to0$ by \eqref{tlp1} and $T_{k,t}$ is the optimal
transport map pushing forward $\mu_k$ to $\mu_{k,t}$, by Proposition~\ref{otstab}
we have $d_2( (\mu_k,T_{k,t}),(\mu,T_t))\to0$. 
Now by Proposition~\ref{p:ax}, because $(T_t)_\sharp\mu=\mu_t$ we have
\begin{equation}\label{e:Ttkconv}
\int |x-S_{k,t}(x)|^2 \,d\mu_t(x) = \int |T_t(z)-T_{k,t}(S_k(z))|^2\,d\mu(z)\to0\,.
\end{equation}
We infer that $(S_{k,t})$ is stagnating and the convergence in \eqref{e:ttlp1}
holds pointwise in $t$. 
But now, the middle quantity in \eqref{e:Ttkconv} is a quadratic function of $t$,
so the uniform convergence in \eqref{e:ttlp1} holds.

Next, we note that the quantity in \eqref{e:ttlp2} is actually independent of $t$.
We have
\[
\int|v_t(x)-v_{k,t}(S_{k,t}(x))|^2\,d\mu_t(x) 
=\int |v_0(z)-v_{k,0}(S_k(z))|^2\,d\mu(z)
 \to 0\,,
\]
due to Proposition~\ref{p:ax}. The proof of \eqref{e:ttlp3} is similar.
\end{proof}

\section{Extended relaxed least-action principle for mixtures}
\label{a:LAP}
 
In this appendix we discuss an extension of the least-action principle
\eqref{la:relax1} which facilitates comparison with previous works
and provides an existence theorem.
The extension involves expanding the class of admissible concentration-momentum pairs, 
and is a kind of hybrid of Brenier's `homogenized vortex sheet'
formulation in \cite{Brenier97} and the variable-density formulation in \cite{LopesNP}
for geodesic flow in the diffeomorphism group. 
The extended formulation reduces, however, to the formulation in \eqref{la:relax1}
whenever the action is finite---see Proposition~\ref{p:reduce} below.
 
The formulations of \cite{Brenier97,Brenier99,LopesNP} were designed to make it possible
to establish existence of minimizers through  convex analysis. 
The key is to express kinetic energy and the weak-form constraints
\eqref{e:wkcon} through duality.
We start with the space $C^0(Q)$ of continuous functions on 
$Q = [-L,L]^d\times[0,1]$, whose dual is 
the space $\calM(Q)$ of signed Radon measures. The duality pairing is 
 \[
 \ip{F,c} = \int_Q F\,dc \qquad\mbox{for $F\in C^0(Q)$,\  $c\in \calM(Q)$}.
 \]
Similarly we write $\ip{G,m} = \int_Q G\cdot dm$ for 
$G\in C^0(Q)^d$ and $m\in \calM(Q)^d$. 
 
Next, let $\hat\varrho\ge0$ be a constant representing fluid density.
We let 
\[
\hat E = C^0(Q)\times C^0(Q)^d, \quad \hat E^*=\calM(Q)\times \calM(Q)^d,
\]
and define $\hat\calK_{\hat\varrho}: \hat E^*\to \R$ as the Legendre transform
of the indicator function $I_{P(\hat\varrho)}$ of the parabolic set
\begin{equation}\label{d:Prho}
P(\hat\varrho) = \{ (F,G)\in \hat E: 
F+\frac12\hat\varrho |G|^2 \le 0 \ \mbox{ in $Q$}\},
\end{equation}
given for $(c,m)\in \hat E^*$ by
\begin{equation}\label{d:hatKrho}
\hat\calK_{\hat\varrho}(c,m) = \sup_{(F,G)\in P(\hat\varrho)} \ip{F,c}+\ip{G,m}.
\end{equation}
(To compare with \cite[eq. (3.8)]{LopesNP} it may help to note 
$\hat\varrho P(\hat\varrho)=P(1)$ when $\hat\varrho>0$.)

The following result follows from  \cite[Proposition 3.4]{Brenier97} 
in the case $\hat\varrho>0$, and is straightforward to show in the case 
$\hat\varrho=0$, when the conclusion entails $m=0$.
\begin{proposition}\label{p:cmabscont} 
Let $\hat\varrho\ge0$, and let $(c,m)\in \hat E^*$. 
If $\hat K_{\hat\varrho}(c,m)<\infty$,
then $c$ is a nonnegative measure and $m$ is absolutely continuous with respect
to $c$, with Radon-Nikod\'ym derivative $\hat\varrho v$ where $v\in L^2(c)$, and
\[
\hat\calK_{\hat\varrho}(c,m) = \int_Q \frac12\hat\varrho |v|^2\,dc.
\]
\end{proposition}

Our reformulated least-action problem may now be specified, as follows.
Let $\hat\varrho_1>\hat\varrho_0\ge0$ and let $E=\hat E\times\hat E$. 
For $(c,m)\in E^*=\hat E^*\times\hat E^*$ we write 
\[
c = (c_0,c_1), \quad m = (m_0,m_1),
\]
and we define
\begin{equation}\label{d:calK}
\calK(c,m) = \sum_{i} \calK_{\hat\varrho_i}(c_i,m_i).
\end{equation}

We introduce the class $\hat\calA_{\calK}$ 
of admissible pairs $(c,m)\in E^*$ as follows, in order to enforce
essentially the same weak-form constraints \eqref{e:wkcon} as before.
Let us say that the test functions $p$, $\phi_1$, $\phi_2$ 
appearing in \eqref{e:wkcon} are {\em suitable} if 
$p\in C^0(Q)$ and $\phi_1$, $\phi_2\in C^1(Q)$.
We say $(c,m)\in E^*$ is admissible if \eqref{e:wkcon} 
is satisfied (with $c_i\,dx\,dt$, $m_i\,dx\,dt$ replaced respectively by $dc_i$, $dm_i$)
for all suitable $p$, $\phi_1$ and $\phi_2$.

The extended relaxed least-action problem is to find 
$(\hat c,\hat m)\in \hat\calA_\calK$ such that
\begin{equation}\label{la:relax2}
\calK(\hat c,\hat m) = \inf_{(c,m)\in \hat\calA_\calK} {\calK}(c,m). 
\end{equation}

This form of the relaxed least-action problem may be compared  
rather directly with the homogenized vortex sheet model of Brenier~\cite{Brenier97}
and with the variable-density model of Lopes Filho et al.~\cite{LopesNP}.
Both of these models deal with the endpoint problem for diffeomorphisms
rather than mass distributions as is done here. 
Brenier's model involves a sum over `phases' as in our model \eqref{d:calK},
but the fluid density in each phase is the same.  
The variable-density model of \cite{LopesNP} 
allows for mixture density (called $c$, corresponding to $\hat\varrho c$ here) 
to depend upon both Eulerian and Lagrangian spatial coordinates 
(called $x$ and $a$ respectively), similar to the formulation in \cite{Brenier99}.

In both \cite{Brenier97} and \cite{LopesNP} as well as related works for 
relaxed least-action principles formulated in a space of measures,
the existence of minimizers is established by using the Fenchel-Rockafellar
theorem from convex analysis. One expresses the objective function
corresponding to $\calK(c,m)$ as a sum of Legendre transforms of 
indicator functions of two sets,  one corresponding to kinetic energy
and another that accounts for the constraints in \eqref{e:wkcon}.

For the degenerate case $\hat\varrho_0=0$ 
that is most relevant to the main body of this paper,
existence of a unique minimizer follows from Theorem~\ref{t:W} above
and Proposition~\ref{p:reduce} below. 
When $\hat\varrho_0\ge0$ in general, we find the following.
\begin{theorem} Assume $\hat\varrho_1>\hat\varrho_0\ge0$, 
and the endpoint functions $c_{i0},c_{i1}\in L^\infty(\Omega,[0,1])$ 
satisfy \eqref{c:ICs}.  Assume $\hat\calA_{\calK}$ is nonempty.
Then a minimizer $(\hat c,\hat m)\in \hat\calA_{\calK}$ 
for \eqref{la:relax2} exists.
\end{theorem}
\begin{proof}
We define convex functions $\alpha,\beta\colon E\to \R\cup\{\infty\}$
as follows.  We let $\alpha$ be the indicator function $I_{P_\times}$
for the convex set 
\begin{equation}\label{d:aFR}
P_{\times} = P(\hat\varrho_0)\times P(\hat\varrho_1)\,.
\end{equation}
Then by \eqref{d:hatKrho}, the Legendre transform 
$\alpha^* = \calK$ as defined in \eqref{d:calK}.

Next we fix an admissible pair $(c_*,m_*)\in\hat\calA_{\calK}$,
so that \eqref{e:wkcon} holds for $(c_*,m_*)$,  
and define
\begin{equation}\label{d:bFR}
\beta(F,G) =  \ip{F,c_*}+\ip{G,m_*} + I_{W}(F,G),
\end{equation}
where 
and $I_W$ is the indicator function on the set $W\subset E$
consisting of all pairs $(F,G) = ((F_1,G_1),(F_2,G_2))$
such that {\em there exist} some suitable $p$, $\phi_1$, $\phi_2$
such that
\begin{equation}
F_j = p +\hat\varrho_j \D_t\phi_j\,, \qquad G_j = \nabla_x\phi_j,\quad j=1,2.
\end{equation}
With this definition, 
the Legendre transform $\beta^*:E^*\to\R\cup\{\infty\}$
is given by 
\[
\beta^*(c,m) = \sup_{(F,G)\in W} \ip{F,c-c_*}+\ip{G,m-m_*},
\]
and one verifies that 
\begin{equation}\label{e:betastar}
\beta^*(c,m) = \begin{cases}
0 & \mbox{if \eqref{e:wkcon} 
holds for all suitable $p$, $\phi_1$, $\phi_2$}, \\
+\infty &\mbox{if \eqref{e:wkcon} 
fails for some suitable $p$, $\phi_1$, $\phi_2$.}
\end{cases}
\end{equation}
That is, $\beta^*$ is the indicator function for the admissible class
$\hat\calA_{\calK}$. 

By consequence, the extended relaxed least
action problem \eqref{la:relax2} is equivalent to finding 
a minimizer $(\hat c,\hat m)\in E^*$ for the problem
\begin{equation}\label{la:relax*}
\min_{E^*} \alpha^*(c,m)+\beta^*(c,m).
\end{equation}
We may now obtain existence by invoking the Fenchel-Rockafellar 
theorem \cite[Thm.~1.12]{Brezis},
after noting that $\alpha$ is continuous at some point $(F,G)$ 
in the domain of both $\alpha$ and $\beta$, given by $F_i=-p$, $G_i=0$
where $p$ is a positive constant.
\end{proof}


We claim that the relaxed least-action problem \eqref{la:relax2}
always reduces to the previous problem \eqref{la:relax1}, due to the following fact.
\begin{proposition}\label{p:reduce} Suppose $(c,m)\in\hat\calA_\calK$ and $\calK(c,m)<\infty$. 
Then for some $(\bar c,\bar m)\in \calA_K$ we have  $\calK(c,m) = K(\bar c,\bar m)$
and
\begin{equation}\label{e:dcdm}
dc_i = \bar c_i\,dx\,dt, \quad dm_i = \bar m_i\,dx\,dt, \quad i=0,1.
\end{equation}
Consequently, the infimum in \eqref{la:relax2} is the same as that in \eqref{la:relax1}.
\end{proposition}

\begin{proof} To prove this result, we first invoke Proposition~\ref{p:cmabscont}
to infer that $c_i$ is a nonnegative measure and $m_i$ is absolutely continuous
with respect to $c_i$ for $i=0,1$.  Next we note that $\sum_i c_i=1$ by taking 
$\phi_i=0$ and $p$ arbitrary in \eqref{e:wkcon}. Hence the representation
in \eqref{e:dcdm} holds with $\bar c_i\in L^\infty(Q,[0,1])$ and
$m_i=\hat\varrho_i \bar c_iv_i$ with $v_i\in L^2(\bar c_i\,dx\,dt)$.

Finally, we claim $t\mapsto \bar c_i(\cdot,t)$ is weak-$\star$ continuous
into $\calM(Q)$. By choosing $p=0$ and $\phi_i$ to depend only on $t$ in \eqref{e:wkcon}
we infer that $\int_{\Omega} \bar c_i(x,t)\,dx$ is independent of $t$. 
Thus, because $\Omega$ is compact, we can invoke Lemma~8.1.2 of \cite{AGS} to 
conclude that $t\mapsto \bar c_i(\cdot,t)$ is narrowly, hence weak-$\star$, continuous.

It is clear that the infimum in \eqref{la:relax1} is greater or equal to that in \eqref{la:relax2},
because the admissible set $\calA_K$ is naturally embedded in $\hat\calA_\calK$,
and the two are equal if either is finite.  Recalling that $\inf \emptyset=+\infty$,
equality follows in general. 
\end{proof}

\begin{remark}\label{r:consist}
As a last comment, we note that for variable-density flow with strictly positive
density, the relaxed least-action problem studied by Lopes et al.~\cite{LopesNP}
was shown to be \textit{consistent} with the classical Euler equations, 
in the sense that  classical solutions of the Euler system induce weak solutions 
of relaxed Euler equations, and for sufficiently short time 
the induced solution is the unique minimizer of the relaxed problem. 
In the case that we consider with $\hat\varrho_0=0$, however, 
this consistency property cannot hold in general when the space dimension
$d>1$, for the reason that in general we can expect the Wasserstein density 
$\rho<1$ in Theorem~\ref{t:W} (see Theorem~\ref{th4}), while necessarily $\rho\in\{0,1\}$ for any 
classical solution of the incompressible Euler equations.
\end{remark}

\section*{Acknowledgements}
The authors thank Yann Brenier for enlightening discussions and generous hospitality.
Thanks also to Matt Thorpe for the computation of the optimal transport map 
appearing in Figure~\ref{fig:vitali}, and to Yue Pu for careful reading and corrections.


This material is based upon work supported by 
the National Science Foundation under 
NSF Research Network Grant no.\ RNMS11-07444 (KI-Net)
and partially supported by the Center for Nonlinear Analysis (CNA)
under National Science Foundation PIRE Grant no.\ OISE-0967140.
The first author was partially supported by the National Science Foundation
with grant DMS 1514826. 
The second author was partially supported by the National Science Foundation
with grants DMS 1211161 and DMS 1515400, 
and by the Simons Foundation under grant 395796. 
The third author was partially supported by the National Science Foundation
with grants CCF 1421502 and DMS 1516677. 


\bibliographystyle{siam}
\bibliography{eulerrefs}

\def\cprime{$'$}
\begin{thebibliography}{10}

\bibitem{AmbrosioFigalli09}
{\sc L.~Ambrosio and A.~Figalli}, {\em Geodesics in the space of
  measure-preserving maps and plans}, Arch. Ration. Mech. Anal., 194 (2009),
  pp.~421--462.

\bibitem{AGuser}
{\sc L.~Ambrosio and N.~Gigli}, {\em A user's guide to optimal transport}, in
  Modelling and optimisation of flows on networks, vol.~2062 of Lecture Notes
  in Math., Springer, Heidelberg, 2013, pp.~1--155.

\bibitem{AGS}
{\sc L.~Ambrosio, N.~Gigli, and G.~Savar{\'e}}, {\em Gradient flows in metric
  spaces and in the space of probability measures}, Lectures in Mathematics ETH
  Z\"urich, Birkh\"auser Verlag, Basel, second~ed., 2008.

\bibitem{Arnold66}
{\sc V.~Arnold}, {\em Sur la g\'eom\'etrie diff\'erentielle des groupes de
  {L}ie de dimension infinie et ses applications \`a l'hydrodynamique des
  fluides parfaits}, Ann. Inst. Fourier (Grenoble), 16 (1966), pp.~319--361.

\bibitem{BMTY05}
{\sc M.~F. Beg, M.~I. Miller, A.~Trouv{\'e}, and L.~Younes}, {\em Computing
  large deformation metric mappings via geodesic flows of diffeomorphisms},
  Int. J. Comput. Vis., 61 (2005), pp.~139--157.

\bibitem{BenBre00}
{\sc J.-D. Benamou and Y.~Brenier}, {\em A computational fluid mechanics
  solution to the {M}onge-{K}antorovich mass transfer problem}, Numer. Math.,
  84 (2000), pp.~375--393.

\bibitem{Braides}
{\sc A.~Braides}, {\em {$\Gamma$}-convergence for beginners}, vol.~22 of Oxford
  Lecture Series in Mathematics and its Applications, Oxford University Press,
  Oxford, 2002.

\bibitem{Brenier89}
{\sc Y.~Brenier}, {\em The least action principle and the related concept of
  generalized flows for incompressible perfect fluids}, J. Amer. Math. Soc., 2
  (1989), pp.~225--255.

\bibitem{Bre91}
{\sc Y.~Brenier}, {\em Polar factorization and monotone rearrangement of
  vector-valued functions}, Comm. Pure Appl. Math., 44 (1991), pp.~375--417.

\bibitem{Brenier92}
{\sc Y.~Brenier}, {\em The dual least action problem for an ideal,
  incompressible fluid}, Arch. Rational Mech. Anal., 122 (1993), pp.~323--351.

\bibitem{Brenier97}
{\sc Y.~Brenier}, {\em A homogenized model for vortex sheets}, Arch. Rational
  Mech. Anal., 138 (1997), pp.~319--353.

\bibitem{Brenier99}
\leavevmode\vrule height 2pt depth -1.6pt width 23pt, {\em Minimal geodesics on
  groups of volume-preserving maps and generalized solutions of the {E}uler
  equations}, Comm. Pure Appl. Math., 52 (1999), pp.~411--452.

\bibitem{Brenier2008}
\leavevmode\vrule height 2pt depth -1.6pt width 23pt, {\em Generalized
  solutions and hydrostatic approximation of the {E}uler equations}, Phys. D,
  237 (2008), pp.~1982--1988.

\bibitem{Brenier2013}
\leavevmode\vrule height 2pt depth -1.6pt width 23pt, {\em Remarks on the
  minimizing geodesic problem in inviscid incompressible fluid mechanics},
  Calc. Var. Partial Differential Equations, 47 (2013), pp.~55--64.

\bibitem{BOS11}
{\sc Y.~Brenier, F.~Otto, and C.~Seis}, {\em Upper bounds on coarsening rates
  in demixing binary viscous liquids}, SIAM J. Math. Anal., 43 (2011),
  pp.~114--134.

\bibitem{Brezis}
{\sc H.~Brezis}, {\em Functional analysis, {S}obolev spaces and partial
  differential equations}, Universitext, Springer, New York, 2011.

\bibitem{BrMiMu14}
{\sc M.~Bruveris, P.~W. Michor, and D.~Mumford}, {\em Geodesic completeness for
  {S}obolev metrics on the space of immersed plane curves}, Forum Math. Sigma,
  2 (2014), pp.~e19, 38.

\bibitem{BrVi17}
{\sc M.~Bruveris and F.-X. Vialard}, {\em On completeness of groups of
  diffeomorphisms}, J. Eur. Math. Soc. (JEMS), 19 (2017), pp.~1507--1544.

\bibitem{caff91}
{\sc L.~A. Caffarelli}, {\em Some regularity properties of solutions of {M}onge
  {A}mp\`ere equation}, Comm. Pure Appl. Math., 44 (1991), pp.~965--969.

\bibitem{CoutShko2007}
{\sc D.~Coutand and S.~Shkoller}, {\em Well-posedness of the free-surface
  incompressible {E}uler equations with or without surface tension}, J. Amer.
  Math. Soc., 20 (2007), pp.~829--930.

\bibitem{CoutShko2010}
\leavevmode\vrule height 2pt depth -1.6pt width 23pt, {\em A simple proof of
  well-posedness for the free-surface incompressible {E}uler equations},
  Discrete Contin. Dyn. Syst. Ser. S, 3 (2010), pp.~429--449.

\bibitem{figall_dephil_BAMS}
{\sc G.~De~Philippis and A.~Figalli}, {\em The {M}onge-{A}mp\`ere equation and
  its link to optimal transportation}, Bull. Amer. Math. Soc. (N.S.), 51
  (2014), pp.~527--580.

\bibitem{dPF2015}
\leavevmode\vrule height 2pt depth -1.6pt width 23pt, {\em Partial regularity
  for optimal transport maps}, Publ. Math. Inst. Hautes \'Etudes Sci., 121
  (2015), pp.~81--112.

\bibitem{DSbook}
{\sc N.~Dunford and J.~T. Schwartz}, {\em Linear {O}perators. {I}. {G}eneral
  {T}heory}, With the assistance of W. G. Bade and R. G. Bartle. Pure and
  Applied Mathematics, Vol. 7, Interscience Publishers, Inc., New York;
  Interscience Publishers, Ltd., London, 1958.

\bibitem{DuGrMi98}
{\sc P.~Dupuis, U.~Grenander, and M.~I. Miller}, {\em Variational problems on
  flows of diffeomorphisms for image matching}, Quart. Appl. Math., 56 (1998),
  pp.~587--600.

\bibitem{EbinMarsden}
{\sc D.~G. Ebin and J.~Marsden}, {\em Groups of diffeomorphisms and the motion
  of an incompressible fluid.}, Ann. of Math. (2), 92 (1970), pp.~102--163.

\bibitem{EvansGariepy}
{\sc L.~C. Evans and R.~F. Gariepy}, {\em Measure theory and fine properties of
  functions}, Studies in Advanced Mathematics, CRC Press, Boca Raton, FL, 1992.

\bibitem{Fig10}
{\sc A.~Figalli}, {\em Regularity properties of optimal maps between nonconvex
  domains in the plane}, Comm. Partial Differential Equations, 35 (2010),
  pp.~465--479.

\bibitem{FK2010}
{\sc A.~Figalli and Y.-H. Kim}, {\em Partial regularity of {B}renier solutions
  of the {M}onge-{A}mp\`ere equation}, Discrete Contin. Dyn. Syst., 28 (2010),
  pp.~559--565.

\bibitem{FuchsEtal}
{\sc M.~Fuchs, B.~J{\"u}ttler, O.~Scherzer, and H.~Yang}, {\em Shape metrics
  based on elastic deformations}, J. Math. Imaging Vision, 35 (2009),
  pp.~86--102.

\bibitem{GanMacshape}
{\sc W.~Gangbo and R.~J. McCann}, {\em Shape recognition via {W}asserstein
  distance}, Quart. Appl. Math., 58 (2000), pp.~705--737.

\bibitem{GTS}
{\sc N.~{Garc{\'i}a Trillos} and D.~Slep\v{c}ev}, {\em {Continuum Limit of
  Total Variation on Point Clouds}}, Arch. Ration. Mech. Anal., 220 (2016),
  pp.~193--241.

\bibitem{Holm2013}
{\sc F.~Gay-Balmaz, D.~D. Holm, and T.~S. Ratiu}, {\em Geometric dynamics of
  optimization}, Commun. Math. Sci., 11 (2013), pp.~163--231.

\bibitem{Grenander_Miller_CA}
{\sc U.~Grenander and M.~I. Miller}, {\em Computational anatomy: An emerging
  discipline}, Quart. Appl. Math., LVI (1998), pp.~617--694.

\bibitem{HZTA}
{\sc S.~Haker, L.~Zhu, A.~Tannembaum, and S.~Angenent}, {\em Optimal mass
  transport for registration and warping}, Int. J. Comput. Vis., 60 (2004),
  pp.~225--240.

\bibitem{Holm2009}
{\sc D.~D. Holm, A.~Trouv{\'e}, and L.~Younes}, {\em The {E}uler-{P}oincar\'e
  theory of metamorphosis}, Quart. Appl. Math., 67 (2009), pp.~661--685.

\bibitem{Jost}
{\sc J.~Jost}, {\em Riemannian geometry and geometric analysis}, Universitext,
  Springer, Heidelberg, sixth~ed., 2011.

\bibitem{Knosmi84}
{\sc M.~Knott and C.~S. Smith}, {\em On the optimal mapping of distributions},
  J. Optim. Theory Appl., 43 (1984), pp.~39--49.

\bibitem{Lindblad}
{\sc H.~Lindblad}, {\em Well-posedness for the motion of an incompressible
  liquid with free surface boundary}, Ann. of Math. (2), 162 (2005),
  pp.~109--194.

\bibitem{DLPv1}
{\sc J.-G. Liu, R.~L. Pego, and D.~Slep{\v c}ev}, {\em Euler sprays and
  {W}asserstein geometry of the space of shapes}.
\newblock arXiv:1604.03387v1.

\bibitem{LopesNP}
{\sc M.~C. Lopes~Filho, H.~J. Nussenzveig~Lopes, and J.~C. Precioso}, {\em
  Least action principle and the incompressible {E}uler equations with variable
  density}, Trans. Amer. Math. Soc., 363 (2011), pp.~2641--2661.

\bibitem{McCann97}
{\sc R.~J. McCann}, {\em A convexity principle for interacting gases}, Adv.
  Math., 128 (1997), pp.~153--179.

\bibitem{MicMum05}
{\sc P.~W. Michor and D.~Mumford}, {\em Vanishing geodesic distance on spaces
  of submanifolds and diffeomorphisms}, Doc. Math., 10 (2005), pp.~217--245.

\bibitem{MicMum06}
\leavevmode\vrule height 2pt depth -1.6pt width 23pt, {\em Riemannian
  geometries on spaces of plane curves}, J. Eur. Math. Soc. (JEMS), 8 (2006),
  pp.~1--48.

\bibitem{Mignot}
{\sc F.~Mignot}, {\em Contr\^ole dans les in\'equations variationelles
  elliptiques}, J. Functional Analysis, 22 (1976), pp.~130--185.

\bibitem{Otto01}
{\sc F.~Otto}, {\em The geometry of dissipative evolution equations: the porous
  medium equation}, Comm. Partial Differential Equations, 26 (2001),
  pp.~101--174.

\bibitem{RubTomGui}
{\sc Y.~Rubner, C.~Tomassi, and L.~J. Guibas}, {\em The earth mover's distance
  as a metric for image retrieval}, Int. J. Comput. Vis., 40 (2000),
  pp.~99--121.

\bibitem{RuWi13}
{\sc M.~Rumpf and B.~Wirth}, {\em Discrete geodesic calculus in shape space and
  applications in the space of viscous fluidic objects}, SIAM J. Imaging Sci.,
  6 (2013), pp.~2581--2602.

\bibitem{Santa}
{\sc F.~Santambrogio}, {\em Optimal transport for applied mathematicians},
  Progress in Nonlinear Differential Equations and their Applications, 87,
  Birkh\"auser/Springer, Cham, 2015.
\newblock Calculus of variations, PDEs, and modeling.

\bibitem{SchSch2013}
{\sc B.~Schmitzer and C.~Schn{\"o}rr}, {\em Contour manifolds and optimal
  transport}.
\newblock arXiv:1309.2240, 2013.

\bibitem{SchSch15}
\leavevmode\vrule height 2pt depth -1.6pt width 23pt, {\em Globally optimal
  joint image segmentation and shape matching based on {W}asserstein modes}, J.
  Math. Imaging Vision, 52 (2015), pp.~436--458.

\bibitem{Shnirelman94}
{\sc A.~I. Shnirel{\cprime}man}, {\em Generalized fluid flows, their
  approximation and applications}, Geom. Funct. Anal., 4 (1994), pp.~586--620.

\bibitem{Thompson}
{\sc D.~W. Thompson}, {\em On Growth and Form}, Cambridge University Press,
  Cambridge, 1917.

\bibitem{Trouv95}
{\sc A.~Trouv{\'e}}, {\em Action de groupe de dimension infinie et
  reconnaissance de formes}, C. R. Acad. Sci. Paris S\'er. I Math., 321 (1995),
  pp.~1031--1034.

\bibitem{Villani03}
{\sc C.~Villani}, {\em Topics in optimal transportation}, vol.~58 of Graduate
  Studies in Mathematics, American Mathematical Society, Providence, RI, 2003.

\bibitem{Villani09}
\leavevmode\vrule height 2pt depth -1.6pt width 23pt, {\em Optimal transport},
  vol.~338 of Grundlehren der Mathematischen Wissenschaften [Fundamental
  Principles of Mathematical Sciences], Springer-Verlag, Berlin, 2009.
\newblock Old and new.

\bibitem{WOSLCR}
{\sc W.~Wang, J.~A. Ozolek, D.~Slep\v{c}ev, A.~B. Lee, C.~Chen, and G.~K.
  Rohde}, {\em An optimal transportation approach for nuclear structure-based
  pathology}, IEEE Transacions on Medical Imaging, 30 (2011), pp.~621 -- 631.

\bibitem{LOT}
{\sc W.~Wang, D.~Slep\v{c}ev, S.~Basu, J.~A. Ozolek, and G.~K. Rohde}, {\em A
  linear optimal transportation framework for quantifying and visualizing
  variations in sets of images}, International Journal of Computer Vision, 101
  (2013), pp.~254--269.

\bibitem{WBRS11}
{\sc B.~Wirth, L.~Bar, M.~Rumpf, and G.~Sapiro}, {\em A continuum mechanical
  approach to geodesics in shape space}, Int. J. Comput. Vis., 93 (2011),
  pp.~293--318.

\bibitem{Wu97}
{\sc S.~Wu}, {\em Well-posedness in {S}obolev spaces of the full water wave
  problem in {$2$}-{D}}, Invent. Math., 130 (1997), pp.~39--72.

\bibitem{Wu99}
\leavevmode\vrule height 2pt depth -1.6pt width 23pt, {\em Well-posedness in
  {S}obolev spaces of the full water wave problem in 3-{D}}, J. Amer. Math.
  Soc., 12 (1999), pp.~445--495.

\bibitem{Youn98}
{\sc L.~Younes}, {\em Computable elastic distances between shapes}, SIAM J.
  Appl. Math., 58 (1998), pp.~565--586 (electronic).

\bibitem{Younes-book}
\leavevmode\vrule height 2pt depth -1.6pt width 23pt, {\em Shapes and
  diffeomorphisms}, vol.~171 of Applied Mathematical Sciences, Springer-Verlag,
  Berlin, 2010.

\bibitem{YMSM_metric_shape}
{\sc L.~Younes, P.~W. Michor, J.~Shah, and D.~Mumford}, {\em A metric on shape
  space with explicit geodesics}, Atti Accad. Naz. Lincei Cl. Sci. Fis. Mat.
  Natur. Rend. Lincei (9) Mat. Appl., 19 (2008), pp.~25--57.

\end{thebibliography}

\end{document}